\documentclass[reqno,11pt]{amsart}
\usepackage{amsmath, latexsym, amsfonts, amssymb, amsthm, amscd}
\usepackage[english]{babel}
\usepackage[T1]{fontenc}
\usepackage{url}
\usepackage{verbatim} 

\usepackage{graphicx}
\def \R{\mathbb R}

\def \P{\mathbb P}
\def \E{\mathbb E}

\newcommand{\tZ}{\tilde{Z}}

\usepackage{bbm}

\usepackage{color}

\newtheorem{thm}{Theorem}[section]
\newtheorem{cor}[thm]{Corollary}
\newtheorem{lem}[thm]{Lemma}
\newtheorem{prop}[thm]{Proposition}
\newtheorem{defin}[thm]{Definition}
\newtheorem{rem}[thm]{Remark}

\newtheorem{lemma}[thm]{Lemma}
\newtheorem{proposition}[thm]{Proposition}

\newtheorem{assumption}{Assumption}

\newcommand{\Ind}{{\rm{\textbf{1}}}}

\renewcommand{\tilde}{\widetilde}

\newcommand{\cF}{{\ensuremath{\mathcal F}} }

\DeclareMathSymbol{\leqslant}{\mathalpha}{AMSa}{"36} 
\DeclareMathSymbol{\geqslant}{\mathalpha}{AMSa}{"3E} 
\DeclareMathSymbol{\eset}{\mathalpha}{AMSb}{"3F}     
\renewcommand{\leq}{\;\leqslant\;}                   
\renewcommand{\geq}{\;\geqslant\;}                   

\makeatletter
\def\captionfont@{\footnotesize}
\def\captionheadfont@{\scshape}

\long\def\@makecaption#1#2{%
  \vspace{2mm}
  \setbox\@tempboxa\vbox{\color@setgroup
    \advance\hsize-6pc\noindent
    \captionfont@\captionheadfont@#1\@xp\@ifnotempty\@xp
        {\@cdr#2\@nil}{.\captionfont@\upshape\enspace#2}%
    \unskip\kern-6pc\par
    \global\setbox\@ne\lastbox\color@endgroup}%
  \ifhbox\@ne 
    \setbox\@ne\hbox{\unhbox\@ne\unskip\unskip\unpenalty\unkern}%
  \fi
  \ifdim\wd\@tempboxa=\z@ 
    \setbox\@ne\hbox to\columnwidth{\hss\kern-6pc\box\@ne\hss}%
  \else 
    \setbox\@ne\vbox{\unvbox\@tempboxa\parskip\z@skip
        \noindent\unhbox\@ne\advance\hsize-6pc\par}%
\fi
  \ifnum\@tempcnta<64 
    \addvspace\abovecaptionskip
    \moveright 3pc\box\@ne
  \else 
    \moveright 3pc\box\@ne
    \nobreak
    \vskip\belowcaptionskip
  \fi
\relax
}
\makeatother
\def\writefig#1 #2 #3 {\rlap{\kern #1 truecm
\raise #2 truecm \hbox{#3}}}

\numberwithin{equation}{section}

\title
[Self-excitatory singular particle systems]
{Particle systems with a singular mean-field self-excitation.
Application to neuronal networks.}

\author{F. Delarue, J. Inglis, S. Rubenthaler, E. Tanr{\'e}}
\address{Fran\c{c}ois Delarue, Sylvain Rubenthaler, Laboratoire J.-A. Dieudonn\'e, Universit\'e de Nice Sophia-Antipolis, Parc Valrose, 06108 Nice Cedex 02, France.}
\email{delarue@unice.fr,rubentha@unice.fr}
\address{James Inglis, Etienne Tanr\'e, Equipe Tosca, INRIA Sophia-Antipolis M\'editerran\'ee,
2004 route des lucioles, BP 93,  06902 Sophia Antipolis Cedex,  France} 
\email{James.Inglis@inria.fr,Etienne.Tanre@inria.fr}

\begin{document}


\maketitle
\begin{abstract}
We discuss the construction and approximation of 
solutions to a nonlinear McKean-Vlasov equation driven by a singular self-excitatory interaction of 
the mean-field type. Such an equation is intended to describe  
an infinite population of neurons which interact with one another. Each time a proportion of neurons `spike', the whole network 
instantaneously receives an excitatory kick. The instantaneous nature of the excitation makes the system singular and prevents the application of standard results from the literature. Making use of the 
Skorohod M1 topology, we prove that, for the right notion of a `physical' solution, 
the nonlinear equation can be approximated either by a finite particle system or by a delayed equation. 
As a by-product, we obtain the existence of `synchronized' solutions, for which a macroscopic proportion of neurons may spike at the same time. 

\vspace{5pt}

\end{abstract}

\section{Introduction}

Recently several rigorous studies (\cite{CCP, caceres:perthame, carillo:gonzales:gualdani:schonbek, DIRT}) have been concerned with a mean-field equation modeling the behavior of a very large (infinite) network of interacting spiking neurons proposed in \cite{ostojic:brunel:hakim}
(see also \cite{baladron-fasioli-touboul-faugeras, masi-galves-locherbach-presutti, faugeras-touboul-cessac, lucon-stannat} and references therein for other types of mean-field models motivated by neuroscience).  
As a nonlinear SDE in one-dimension the equation for the electrical potential $X_t$ across any typical neuron in the network at time $t$ takes the form
\begin{equation}
\label{nonlinear eq}
X_t = {X_0} + \int_0^tb(X_s)ds +\alpha\E(M_t) + W_t - M_t, \qquad t\geq0,
\end{equation}
where $X_0<1$ almost surely, $(W_t)_{t\geq0}$ is a standard Brownian motion
 and $b$ is a Lipschitz function of linear growth.  Here $\alpha$ is a parameter in $(0, 1)$ and the process $M = (M_t)_{t\geq0}$ counts the number of times that $X=(X_t)_{t\geq0}$ reaches $1$ before time $t$, so that it is integer-valued (see Section \ref{se:2} for a precise description).  The idea is that when $X$ reaches the threshold $1$, $M$ instantly increases by $1$ so that $X$ is reset to a value below the threshold, and we say that the neuron has \textit{spiked}.  Throughout the article we will write $e(t):=\E(M_t)$.


Equation \eqref{nonlinear eq} is in fact nontrivial, since the form of the nonlinearity is not regular enough for the application of the standard  McKean-Vlasov theory (\cite{meleard,sznitman}).  Indeed, 
the problem is that, on the infinitesimal level,
the mean-field term in \eqref{nonlinear eq} reads as $e'(t) =
[d/dt]\E(M_t)$, which is by no means regular with respect to the law of $X_{t}$. In \cite{DIRT}, it is proven that
$e'(t) = - (1/2) \partial_y p(t, 1)$, where 
$p(t, y)dy = \mathbb{P}(X_t\in dy)$ is the marginal density of $X_{t}$, 
which shows how singular the dependence of $e'(t)$ upon the law of $X_{t}$ is.
As such, most of the previous work studying this equation has been focused on the existence of a solution and its properties, bringing to light some non-trivial mechanisms. 

The main point is that, for some choices of parameters 
($\alpha$ too big for fixed $X_0$ concentrated close to the boundary), any solution to 
\eqref{nonlinear eq} must exhibit what has been described as a `blow-up' in finite time.  More precisely this means that $e'(t)$ (which is the mean-firing rate of the network at time $t$) must become infinite for some finite $t$.  This was done in \cite{CCP} by means of a PDE method.  Interpreting \eqref{nonlinear eq} as a description of an infinite network of neurons, a blow-up is thus a time at which a proportion of all the neurons in the network spike at exactly the same time, which we refer to as a synchronization.  Despite the interest in this phenomena, up until now it has been unclear how to continue a solution after a blow-up.  On the other hand, in \cite{DIRT} it was shown by probabilistic arguments that for other choices of parameters ($\alpha$ small enough for fixed $X_0=x_0$), \eqref{nonlinear eq} has a unique solution for all time which does not exhibit the blow-up phenomenon.  These two complementary results are made precise in Theorems \ref{thm:CCP} and \ref{thm:DIRT} below.

The aim of the present work is to provide further insight into this nonlinear equation by providing two ways of 
approximating (and moreover constructing) a solution.  The first is via the natural particle system associated to \eqref{nonlinear eq}, which describes the behavior of the finite network of neurons.  In fact, the introduction of \eqref{nonlinear eq} in \cite{ostojic:brunel:hakim} is inspired from this finite dimensional system: it is there asserted that, when the size of the network becomes infinite, neurons become independent and evolve according to \eqref{nonlinear eq}.  However, the proof of this fact (which is a propagation of chaos result) is not given.  
The first of our main objectives is to fill this gap and to rigorously show that any weak limit of the particle system 
must be a solution to \eqref{nonlinear eq} (see Theorem \ref{weak existence thm}). In particular, we show that the particle system converges to
the solution of \eqref{nonlinear eq} whenever uniqueness holds, in which case propagation of chaos 
holds as well. Again, due to the irregularity and nature of the particle system, this result is in fact more difficult than it might appear. The second objective is to recover a similar result when approximating 
the self-interaction in \eqref{nonlinear eq} by delayed self-interactions (see
Theorem \ref{weak existence thm:delay}). 
The motivation for considering the delayed equation (which is still nonlinear) is that it never exhibits a blow-up phenomenon, even with $\alpha$ close to $1$, making it easier to handle (see Proposition \ref{prop: delayed existence}).  

In both cases, the strategy relies on two ingredients. First, we show that there exists a notion of `physical' solutions to 
equation \eqref{nonlinear eq} for which spikes occur physically, in a `sequential' way. 
The interesting feature of `physical'
solutions is that we allow the function $(e(t))_{t \geq 0}$ to be discontinuous, but
characterize the size of any jumps in a precise way.
Second, we show that there is a particularly suitable topology on the space of `continus \`a droite avec limites \`a gauche' paths
(c\`adl\`ag paths in acronym)  for handling both approximations. The point is indeed to prove that the approximating families are tight for the so-called M1 Skorohod topology on the space of c\`adl\`ag paths, which is much less popular than the 
J1 topology, but which turns out to be very convenient for  handling non-decreasing c\`adl\`ag processes such as 
the counting process $(M_{t})_{t \geq 0}$.   

As a significant by-product, the paper 
shows the existence of `physical' solutions to \eqref{nonlinear eq} for which 
the function $(e(t))_{t \geq 0}$ may be discontinuous, but where we explicitly specify the size of any jump. 
This is a completely new fact in the literature, and is of real importance  
in neuroscience, as the size of the discontinuity of the function $e$ indicates the proportion of neurons that synchronize at any time.
 The notion of `physical' solutions together with 
 the existence result thus permits the continuation of the solution after the synchronization
 and, therefore, allows the circumvention of the blow-up phenomenon experienced in 
 \cite{CCP} and \cite{DIRT}. In particular, this gives a rigorous framework for investigating 
the long time behavior of synchronization events, which is a fundamental question in neuroscience. 
This also raises the question of uniqueness of `physical' solutions that experience a synchronization. We feel that it must be true, but the question is 
left open. We refrain from addressing this problem in the paper as it would require additional materials, including a careful discussion about the shape of the solution after some synchronization has occurred. We plan to go back to this question in a future work. 
%
%
%

We would finally like to remark that variations of the model we present here could well be of interest in other contexts. In particular in a financial setting, a similar system has indeed been used to model the default rate of a large portfolio (\cite{gisecke_2013, spiliopoulos}) where a default occurs when a particle reaches a threshold.  Our model is however more delicate than the one considered there, since the interactions we consider are more singular and produce the blow-up phenomenon that is not present in their setting.

The organization of the paper is as follows. In Section \ref{se:2}, we discuss the notion of 
`physical' solutions to \eqref{nonlinear eq}. The approximating systems are introduced in 
Section \ref{se:3}, in which we prove that both the associated particle system and the delayed equation 
are solvable. The main results are exposed in Section 
 \ref{se:results}, where we also give a rough presentation of the M1 topology.   Proofs are given in Section 
 \ref{se:proofs}.


\section{The nonlinear equation: background}
\label{se:2}
The central nonlinear equation considered in this article is equation \eqref{nonlinear eq}
where $X_0<1$ almost surely, $\alpha\in(0, 1)$ and 
$(W_t)_{t\geq0}$ is a standard Brownian motion, defined on some probability space $(\Omega,{\mathcal A},\P)$
   with respect to a filtration $(\cF_t)_{t\geq0}$ satisfying the usual conditions.  
 The number of spikes of the equation until time $t$ (inclusive) is given by
\begin{equation}
\label{M def}
M_t := \sum_{k\geq1}\Ind_{[0, t]}(\tau_k),
\end{equation}
where the sequence of stopping times $(\tau_k)_{k\geq0}$ is defined by $\tau_0=0$ and 
\begin{equation}
\label{tau def}
\tau_k = \inf\bigl\{t>  \tau_{k-1}: X_{t-} +  \alpha\Delta e(t) \geq1 \bigr\}, \quad k\geq1.
\end{equation}
We have here used the notation
$e(t) := \E(M_t)$, $t\geq0$,
and, for a given c\`adl\`ag function $f: [0, \infty) \rightarrow \R$,
$\Delta f(t) := f(t) - f(t-)$, $t\geq0$, which will be fixed throughout the article.
The pair \eqref{M def}--\eqref{tau def} is highly coupled as the definition of 
$(M_{t})_{t \geq 0}$ relies on its own expectation. This asks for a careful description of the notion of a solution.

\subsection{The right notion of a solution}
 As noted above, the process $M=(M_t)_{t\geq0}$ is intended to count the number of times $X=(X_t)_{t\geq0}$ spikes before time $t$.  At any time $t$ such that $X_{t-} \geq 1 - \alpha \Delta e(t)$,
the process $M$ registers a new spike (pay attention that the presence of the `$-$'
in the condition $X_{t-} \geq 1 - \alpha \Delta e(t) $ is crucial for ensuring the c\`adl\`ag property 
of the process).
In the case when the mapping $e$ is continuous at point $t$, 
the particle spikes if and only if $X_{t-}=1$. It is then reset to $0$ exactly after the spike, that is $X_t=0$.
Whenever $e$ jumps at time $t$, the jump $\Delta e(t) $ must be of positive size so that, because of the 
self-interaction, $X$ may spike
even if $X_{t-} <1$. Immediately after a spike occurs, i.e. when $X_{t-} \geq 1 -   \alpha\Delta e(t) $, 
$X_{t}$ is equal to $X_{t-}-1+\alpha \Delta e(t)$ and
may be strictly positive: it is as if, at time $t$, the particle is first reset to $0$ and then given a kick of magnitude 
$\alpha\Delta e(t)  - (1-X_{t-})$. Actually, such a description requires some precaution as the kick could force the particle to cross 
the barrier again at the same time $t$. This might happen if the kick 
$\alpha\Delta e(t) $ is greater than or equal to 1. Anyhow, such a phenomenon is expected to be `non-physical': under the condition $\alpha<1$, it does not make any sense to allow the system to spike twice (or more) at the same time. 
The argument for this is discussed at length below when making the connection with the 
finite particle system. In short, it says that physical spikes occur sequentially. 

The fact that the jumps of the 
process $(M_{t})_{t \geq 0}$ cannot exceed 1 provides some insight into the sequence of spiking times $(\tau_{k})_{k \geq 1}$.
First, given a solution satisfying $\P(\Delta M_{t}  \leq 1)=1$ for all $t \geq 0$, 
the sequence $(\tau_{k})_{k \geq 1}$ must be (strictly) increasing: there is no way for 
two spiking times to coincide if labelled by different indices. Moreover, the sequence $(\tau_{k})_{k \geq 1}$ cannot accumulate in finite time, as otherwise it would contradict the c\`adl\`ag nature of $(X_{t})_{t \geq 0}$.
Indeed, if $\tau_{\infty} := \lim_{k \rightarrow + \infty} \tau_{k} < + \infty$, then
$X_{\tau_{\infty}-}$ is equal to both
$\lim_{k \rightarrow + \infty} X_{\tau_{k}-}$ and
$\lim_{k \rightarrow + \infty} X_{\tau_{k}}$, which gives a contradiction since 
$X_{\tau_{k}}=X_{\tau_{k}-}-1+\alpha \Delta e(\tau_{k}) < X_{\tau_{k}-}-1+\alpha$. 

It also gives some insight into the jumps of the function $e$, summarized in the following proposition.

\begin{proposition}
Assume that the pair $(X_{t},M_{t})_{t \geq 0}$ of c\`adl\`ag processes is such that
\begin{enumerate}
\item $(M_{t})_{t \geq 0}$ has integrable marginal distributions;
\item for all $t \geq 0$, $\P(\Delta M_{t}  \leq 1)=1$;
\item $\P$-almost surely, \eqref{nonlinear eq}, \eqref{M def}
and 
\eqref{tau def} hold true.
\end{enumerate}
Then, for any time $t \geq 0$, the jump $\Delta e(t)$ satisfies
\begin{equation}
\label{eq:13:2:14:1}
\Delta e(t) = \P \bigl(X_{t-} + \alpha \Delta e(t) \geq 1 \bigr).
\end{equation}
\end{proposition}

\begin{proof}
Given some time $t \geq 0$, a necessary and sufficient condition for registering a spike (that is to have 
$M_{t} - M_{t-}=1$), is 
$X_{t-} + \alpha \Delta e(t) \geq 1$. Therefore, the probability of observing a spike is $\P(X_{t-} + \alpha \Delta e(t) \geq 1)$, which proves that 
$\Delta e(t) = \P ( \Delta M_{t} = 1 ) = \P (X_{t-} + \alpha \Delta e(t) \geq 1 ). 
$
\end{proof}

Unfortunately, equation \eqref{eq:13:2:14:1} is not sufficient to characterize the size of the jumps. Indeed one can guess simple examples
of distributions for the law of $X_{t-}$ such that the equation (in $\eta$) 
\begin{equation}
\label{eq:13:2:14:2}
\eta = \P \bigl(X_{t-} + \alpha \eta \geq 1 \bigr) 
\end{equation}
has several solutions. For instance, if $X_{t-}$ has a uniform distribution on $[1-\alpha,1]$, then
the equation is satisfied for every $\eta \in [0,1]$.
In order to determine which solution to \eqref{eq:13:2:14:2} characterizes the size of the jump, we refer again to what 
a physical solution to \eqref{nonlinear eq} must be. In \eqref{eq:13:2:14:2}, 
$\alpha \eta$ is intended to stand for the magnitude of the kick felt by the particle. The idea behind this is that 
we consider all the $\omega\in\Omega$ for which the kick is large enough to make the particle cross the barrier. 
To put it differently there must be enough mass near 1 
in the distribution of $X_{t-}$ to `absorb' the particle from $1-\alpha \eta$ to $1$. 
Implicitly, this requires that there is no gap in the mass. If, for some $\eta' < \eta$, 
the probability $\P(X_{t-} + \alpha \eta' \geq 1)$ is (strictly) less than $\eta'$, then 
the kick is not strong enough to absorb the particle when at distance $\alpha \eta'$ from $1$. 
This suggests that, physically, the magnitude of the kick 
must be given as the largest magnitude for which `absorption' can occur. Therefore, a reasonable characterization 
for $\Delta e(t)$ is 
\begin{equation*}
\begin{split}
\Delta e(t) &= \sup \bigl\{ \eta \geq 0 : \forall \eta' \leq \eta,
  \P \bigl(X_{t-} + \alpha \eta' \geq 1 \bigr) \geq \eta' \bigr\}
  \\
  &= \inf \bigl\{ \eta \geq 0 : 
\P \bigl(X_{t-} + \alpha \eta \geq 1 \bigr) < \eta \bigr\}.
 \end{split} 
\end{equation*}
At this stage of the paper, we will keep this characterization as a necessary condition for  
a `physical' solution 
to \eqref{nonlinear eq}. Again, we will justify this choice in a more detailed way below.
With this in mind, we thus make the following precise definition.
\begin{defin}
\label{def:nonlineareq}
We call a (physical) solution to \eqref{nonlinear eq} a pair $(X_{t},M_{t})_{t \geq 0}$ 
of c\`adl\`ag adapted processes such that
\begin{enumerate}
\item $(M_{t})_{t \geq 0}$ has integrable marginal distributions;
\item for all $t \geq 0$, $\P(\Delta M_{t} \leq 1)=1$;
\item $\P$-almost surely, \eqref{nonlinear eq}, \eqref{M def}
and 
\eqref{tau def} hold true;
\item the
discontinuity points of the
 function $e: [0,+\infty) \ni t \mapsto \E(M_{t})$ satisfy
\begin{equation*}
\Delta e(t) = \inf \bigl\{ \eta \geq 0 :
\P \bigl(X_{t-} + \alpha \eta \geq 1 \bigr)  < \eta \bigr\}. 
\end{equation*}
\end{enumerate}
\end{defin}

We underline that a physical solution satisfies \eqref{eq:13:2:14:1}, but that we need (4) to characterize the size of the jump $\Delta e(t)$ (and hence avoid non-physical phenomenon as discussed in the above examples -- see also the paragraph `Non-physical solutions' on page \pageref{Non-physical solutions para}).  A sufficient condition for a physical solution is given in Proposition \ref{prop:assumption:physical:criterion} below.

\subsection{Standing assumptions and related literature}

We will make the following two assumptions throughout the article.
\begin{assumption}[Globally Lipschitz drift]\label{assumption 1} 
The drift $b:(-\infty, 1] \to \R$ is Lipschitz continuous such that
$|b(x) - b(y)| \leq K|x-y|$, for all $x, y \in (-\infty, 1]$.
\end{assumption}

\begin{assumption}[Initial condition]\label{assumption 2} 
The initial condition $X_{0} \in (-\infty, 1 - \varepsilon_0]$ almost surely for some $\varepsilon_0>0$ and
\(X_{0}\in L^p(\Omega)\) for any  $p \geq 1$. 
\end{assumption}

The assumption that the distribution of the initial condition has support in $(-\infty, 1-\varepsilon_0]$, rather than in $(-\infty, 1)$, is a slight simplification. It is motivated by technical reasons that will be specified in the core of the proofs.  

As mentioned in the Introduction, the existence and uniqueness of a solution to \eqref{nonlinear eq} is a nontrivial problem.  It is addressed in \cite{CCP} and \cite{DIRT}, as well as \cite{carillo:gonzales:gualdani:schonbek}, but in the smaller class of pairs $(X_{t},M_{t})_{t \geq 0}$ for which the mapping $e : [0,+\infty) \ni t \mapsto {\mathbb E}(M_{t})$ is continuous (which renders the conditions 
(2) and (4) in Definition \ref{def:nonlineareq} useless). The following two theorems summarize the results in \cite{CCP} and \cite{DIRT} relevant for the present study.

\begin{thm}[\cite{CCP}]
\label{thm:CCP}
For every $\alpha\in(0,1)$, it is possible to find an initial condition $X_0$, such that 
there is no global solution (global meaning defined on the entire interval $[0,+\infty)$)
where the mapping $e$ is continuously differentiable.
Equivalently, it is possible to find an initial condition 
$X_{0}$ such that any solution 
to \eqref{nonlinear eq} experiences a blow-up, in the sense that $e'(t) = + \infty$
for some $t \geq 0$. 
\end{thm}

\begin{thm}[\cite{DIRT}]
\label{thm:DIRT}
For all initial conditions $X_0 = x_0 <1$, it is possible to find an $\alpha_0(x_0)\in(0, 1)$ such that, whenever $\alpha\in(0, \alpha_0)$, equation \eqref{nonlinear eq} possesses a unique (pathwise and thus in law) global solution 
such that the mapping $e$ is continuously differentiable.
\end{thm}


So far existence and uniqueness within the framework of Definition \ref{def:nonlineareq}
are completely open problems. As mentioned in the Introduction, the purpose of this paper is to provide a general compactness method for approximating solutions to \eqref{nonlinear eq}, and as a by-product prove the existence of a solution according to Definition \ref{def:nonlineareq}, for which the map $e$ may be discontinuous.
Inspired by the earlier paper \cite{DIRT}, we will make use of the following reformulation of equation \eqref{nonlinear eq}.

\begin{rem}[Reformulation]
\label{rem: reformulation limit}
It will be very convenient throughout the article to sometimes work instead with a reformulated version of \eqref{nonlinear eq}, which describes the evolution of the process $Z=(Z_t)_{t\geq0}$, defined simply by
\[
Z_t := X_t + M_t, \quad t\geq0.
\]
It is then plain to see that $M_t$ can be completely expressed in terms of $(Z_s)_{0\leq s\leq t}$ as
\begin{equation}
\label{M def Z}
M_t = \bigl\lfloor \bigl( \sup_{0 \leq s \leq t} Z_{s} \bigr)_{+} \bigr\rfloor = \sup_{0 \leq s \leq t} \left\lfloor (Z_{s})_{+} \right\rfloor, \quad t\geq0,
\end{equation}
where $\lfloor x \rfloor$ and $(x)_+$ indicate the integer part of $x$ and $\max\{x, 0\}$ respectively, for any $x\in\R$.
Indeed, as $X_{t}<1$ and $(M_{s})_{s \geq 0}$ is non-descreasing, it is clear that $M_{t} \geq
\lfloor ( \sup_{0 \leq s \leq t} Z_{s} )_{+} \rfloor$. Conversely, for a given $k \geq 0$ such that 
$\tau_{k} \leq t < \tau_{k+1}$,
$X_{\tau_{k}} \geq 0$, so that
 $M_{\tau_{k}}=k
\leq Z_{\tau_{k}}$, which completes the proof of the equality.
  The reformulated version of \eqref{nonlinear eq} is then given by
\begin{equation}
\label{nonlinear eq Z}
Z_t = Z_{0} + \int_0^tb(Z_s - M_s)ds +\alpha\E(M_t) + W_t, \qquad t\geq0,
\end{equation}
where $Z_0 (= X_0)<1$, and $(M_t)_{t\geq0}$ is defined by \eqref{M def Z}.  One big advantage of any solution $Z=(Z_t)_{t\geq0}$ to \eqref{nonlinear eq Z} over a solution $X=(X_t)_{t\geq 0}$ to \eqref{nonlinear eq} is that discontinuity points of $Z$ are 
dictated by those of the deterministic mapping $e: [0,+\infty) \ni t \mapsto {\mathbb E}(M_{t})$ only. 

Conversely, given a solution $(Z_{t}',M_{t}')_{t \geq 0}$ to \eqref{nonlinear eq Z} and \eqref{M def Z}, we  recover a (possibly non-physical) solution 
to the original equation \eqref{nonlinear eq} by setting $X_{t}' = Z_{t}' - M_{t}'$.
\end{rem}

\subsection{A criterion for a physical solution}
The following lemma is an adaptation of \cite[Prop. 3.1]{DIRT}. The proof is left to the reader. It
relies on Gronwall's lemma and
\eqref{M def Z}--\eqref{nonlinear eq Z}. 
\begin{lem}
\label{lem:stab}
Consider a pair $(X_{t},M_{t})_{t \geq 0}$ 
of c\`adl\`ag adapted processes such that
(1) and (3) hold in Definition \ref{def:nonlineareq}. 
Then 
it holds that ${\mathbb E}[\sup_{t \in [0,T]} \vert Z_{t}\vert^p] <  + \infty$ for any $p\geq1, T>0$.
\end{lem}

Next we present a useful application.  The reader may skip the proof on a first reading.
\begin{proposition}
\label{prop:assumption:physical:criterion}
Assume that the pair $(X_{t},M_{t})_{t \geq 0}$ of c\`adl\`ag processes is such that
(1), (2) and (3) hold in Definition \ref{def:nonlineareq}.
Assume also that, at any discontinuity time $t \geq 0$ of the mapping
$e : [0,+\infty) \ni s \mapsto {\mathbb E}(M_{s})$, it holds that
\begin{equation}
\label{eq:assumption:physical:criterion}
\forall \eta \leq \Delta e(t), \quad \P \bigl( X_{t-} \geq 1 - \alpha \eta\bigr) \geq \eta.
\end{equation}
Then the pair $(X_{t},M_{t})_{t \geq 0}$ is a physical solution.
\end{proposition}

\begin{proof}
In order to prove that $(X_{t},M_{t})_{t \geq 0}$ is a physical solution, we must check that, for any $t\geq0$, there exists 
a decreasing sequence $(\eta_{n})_{n \geq 1}$, with $\Delta e(t)$ as its limit, such that 
\begin{equation*}
\P \bigl( X_{t-} + \alpha \eta_{n} \geq 1 \bigr) < \eta_{n}, \quad n\geq1.
\end{equation*}
Together with \eqref{eq:assumption:physical:criterion}, this indeed implies (4) in 
Definition \ref{def:nonlineareq}.

We argue by contradiction. Fix $t\geq0$ and assume that there exists $\eta_{0} > \Delta e(t)$ such that 
\begin{equation*}
\forall \eta \in (\Delta e(t), \eta_{0}], \quad \P \bigl( X_{t-} \geq 1 - \alpha \eta
\bigr) \geq \eta.   
\end{equation*}
Then, recalling from 
\eqref{eq:13:2:14:1}
that $\Delta e(t) = \P (X_{t-} + \alpha \Delta e(t) \geq 1)$, we deduce that
\begin{equation*}
\begin{split}
\forall \eta \in (0, \eta_{0}], \quad &\P \bigl( 1 - \alpha \eta \leq  X_{t-} < 1 - \alpha \Delta e(t)
\bigr)
\\
&= 
\P \bigl( X_{t-} \geq 1 - \alpha \eta \bigr)
- \P  \bigl(X_{t-} + \alpha \Delta e(t) \geq 1\bigr)
 \geq \eta - \Delta e(t).
\end{split}
\end{equation*}
Notice that, on the event $\{ 1 - \alpha \eta\leq X_{t-} < 1 - \alpha \Delta e(t) \}$, 
$\Delta M_{t}=0$, so that $X_{t} = X_{t-} + \alpha \Delta e(t)$. Therefore,
with $\eta'=\eta - \Delta e(t)$, we obtain, $\forall \eta' \in (0, \eta_{0}-\Delta e(t)]$,
\begin{align}
\label{eq:19:12:1}
  \P \bigl( X_{t} \geq 1 - \alpha \eta' \bigr) &\geq \P \bigl( 1>X_{t} \geq 1 - \alpha \eta', X_{t-}+\alpha \Delta e(t) =X_t \bigr) \nonumber\\
&= \P \bigl( 1-\alpha \eta \leq X_{t-} < 1-\alpha \Delta e(t), X_{t-}+\alpha \Delta e(t) =X_t \bigr) \nonumber\\
& = \P \bigl(1-\alpha \eta \leq X_{t-} < 1-\alpha \Delta e (t) \bigr) \geq \eta'.
\end{align}
To simplify, we let $\eta_{0}' := \eta_{0} - \Delta e(t) >0$.

The strategy is then to prove that $\liminf_{h \downarrow 0}[e(t+h)-e(t)] >0$, 
which will contradict the right-continuity of $e$. To do so, we use a stochastic comparison argument. 
For some small $h \in (0,1)$, we indeed have
\begin{equation*}
e(t+h) - e(t) \geq \P \bigl( \exists s \in (t,t+h] : Y_{s-} \geq 1 \bigr),
\end{equation*}
where $(Y_{s})_{s \in [t,t+h]}$ solves the equation 
\begin{equation*}
Y_{s} = X_{t} + \int_{t}^s b(Y_{u}) du +  \alpha (e(s)-e(t)) + W_{s} - W_{t}, \quad 
t \leq s \leq t+h.  
\end{equation*}
Indeed, as long as $(X_{s})_{s \in [t,t+h]}$ does not spike, it coincides with 
$(Y_{s})_{s \in [t,t+h]}$. In particular, if $M_{s-} - M_{t} =0$ and $Y_{s-} \geq 1$, then $X_{s-} \geq 1$ and thus
$M_{s} - M_{t}=1$. Therefore, $\{\exists s \in (t,t+h] : Y_{s-} \geq 1\} \subset
\{M_{t+h}-M_{t} \geq 1\}$.
We then get, for some constant $C$ (the value of which is allowed to increase from line to line, but will remain independent of $h$ and $\alpha$),
\begin{equation}
\label{eq:22:02:14:1}
e(t+h) - e(t)
\geq \P \Bigl( X_{t} - Ch \bigl( 1 +\sup_{s \in [t,t+h]} \vert Y_{s} \vert
\bigr) + \sup_{s \in [t,t+h]} \Bigl[  \alpha (e(s)-e(t)) + W_{s} - W_{t} \Bigr] \geq 1
\Bigr).
\end{equation}
By a standard application of Gronwall's lemma (recalling that $h$ can be chosen small enough so that 
$e(s)-e(t) \leq 1$ for all $s\in[t, t+h]$ as $e$ is right continuous),
\begin{equation*}
\vert Y_{s}  \vert \leq C \bigl( 1 + \vert X_{t} \vert  + \sup_{s \in [t,t+h]} \vert 
W_{s} - W_{t} \vert \bigr), \quad 
t \leq s \leq t+h,
\end{equation*}
so that $\P( \sup_{s \in [t,t+h]} 
\vert Y_{s} \vert  \geq 3 C, \vert X_{t} \vert \leq 1) \leq Ch$.
Since $X_{t} \geq 0$ implies $\vert X_{t} \vert \leq 1$, we obtain 
$\P( \sup_{s \in [t,t+h]} 
\vert Y_{s} \vert  \geq 3 C,  X_{t} \geq 0) \leq Ch$, so that, 
by \eqref{eq:22:02:14:1},
\begin{equation*}
\begin{split}
&e(t+h) - e(t) 
\\
&\geq \P \Bigl( X_{t} -  Ch(1+3C)  + \sup_{s \in [t,t+h]} \Bigl[  \alpha(e(s) - e(t)) + W_{s} - W_{t} \Bigr] \geq 1, 
\ \sup_{s \in [t,t+h]} 
\vert Y_{s}  \vert  \leq 3 C
\Bigr)
\\
&\geq \P \Bigl( X_{t} - Ch  + \sup_{s \in [t,t+h]} \Bigl[  \alpha(e(s) - e(t)) + W_{s} - W_{t} \Bigr] \geq 1, 
\ X_{t} \geq 0
\Bigr) - Ch,
\end{split}
\end{equation*}
where we have adjusted $C$.

Assume now that there exists $c \geq 0$ such that
$\alpha (e(r) - e(t)) \geq c \sqrt{r-t}$ for all $r \in [t,t+h]$, 
which is (at least) true with $c=0$. Then, by the above bound, we get
\begin{equation}
\label{eq:19:02:14:2}
e(t+h) - e(t) \geq 
\int_{0}^{+\infty}
\P \bigl( X_{t}  - Ch
 +  u \geq 1, \ X_{t} \geq 0
\bigr) d \nu(u) - Ch, 
\end{equation}
where $\nu$ denotes the law of the supremum of $ c \sqrt{s} +  W_{s}$ over
$s\in [0,h]$. Notice that $u \leq 1$ and 
$X_{t}  - Ch
 +  u \geq 1$ implies $X_{t} \geq 0$. Assuming without any loss of generality that 
 $\alpha \eta_{0}' = \alpha (\eta_{0} - \Delta e(t)) \leq 1$, we deduce from
\eqref{eq:19:12:1} that
\begin{equation}
\label{eq:19:02:14:3}
e(t+h) - e(t)\geq  
\int_{Ch}^{\alpha \eta_{0}'}
\frac{u - Ch}{\alpha}
d \nu(u) - Ch
\geq \frac{1}{\alpha}  \int_{Ch}^{\alpha \eta_{0}'}
u d \nu(u)- 2\frac{Ch}{\alpha}, 
\end{equation}
the constant $C$ being independent of $c$. 
Recall now that $c \sqrt{h} \leq \alpha (e(t+h) - e(t)) \leq 
\alpha \eta_{0}'/2$ for $h$ small enough as $e$ is right continuous.
Therefore, using the fact that the tail of $\sup_{s \in [0,h]} W_{s}$ is Gaussian, 
we obtain
\begin{equation*}
\begin{split}
\int_{\alpha \eta_{0}'}^{+\infty}
u d \nu(u) 
&={\mathbb E} \biggl[ \sup_{s \in [0,h]} \bigl( c \sqrt{s} + W_{s} \bigr) 
{\mathbf 1}_{\{ 
\sup_{s \in [0,h]} ( c \sqrt{s} + W_{s} )
\geq \alpha \eta_{0}'\}}
\biggr]
\\
&\leq {\mathbb E} \biggl[ \sup_{s \in [0,h]} \left( \frac{\alpha \eta_{0}'}{2} + W_{s} \right) 
{\mathbf 1}_{\{ 
\sup_{s \in [0,h]}  W_{s} 
\geq \alpha \eta_{0}'/2\}}
\biggr] \leq C h.
\end{split}
\end{equation*}
Moreover, quite obviously, $\int_{0}^{Ch} u d\nu(u) \leq Ch$. Finally, by
\eqref{eq:19:02:14:3},
with $C$ independent of $c$,
\begin{equation*}
\begin{split}
\alpha (e(t+h) - e(t)) \geq  
\int_{0}^{+\infty}
u d \nu(u)- Ch 
&=  {\mathbb E} 
\biggl[ \sup_{s \in [0,h]} \bigl( c \sqrt{s} + W_{s} \bigr) \biggr] - C h
\\
&= h^{1/2} \biggl(   
{\mathbb E} 
\biggl[ \sup_{s \in [0,1]} \bigl( c \sqrt{s} + W_{s} \bigr) \biggr] - C h^{1/2} \biggr),
\end{split}
\end{equation*}
the last equality following from Brownian scaling. 
A similar inequality can be proved for any $r \in [t,t+h]$, that is
$\alpha (e(r) - e(t)) \geq f(c) \sqrt{r-t} $, where 
$$f(c)  =   
 {\mathbb E} 
\biggr[ \sup_{s \in [0,1]} \bigl( c \sqrt{s} + W_{s} \bigr) \biggr] - C \sqrt{h}.$$
We deduce that, if the inequality 
$\alpha (e(r)-e(t)) \geq c \sqrt{r-t}$ holds for all $r \in [t,t+h]$, 
then $\alpha (e(r)-e(t)) \geq f(c) \sqrt{r-t}$ for all $r \in [t,t+h]$. 
Letting $c_{0}=0$ and $c_{n+1}=f(c_{n})$ for all $n \geq 0$, we deduce that 
$\alpha (e(r)-e(t)) \geq c_{n} \sqrt{r-t}$
for all $r \in [t,t+h]$ and all $n \geq 0$.

Clearly, we can choose $h$ small enough so that $c_{1} >0=c_{0}$. Since 
$f$ is non-decreasing, we deduce that the sequence $(c_{n})_{n \geq 0}$
is non-decreasing. As $e$ is locally bounded, the sequence has a finite limit $c^*$. Then, 
as $f$ is obviously Lipschitz continuous, we have $c^*=f(c^*)$, that is, 
\[
\begin{split}
c^* = \mathbb{E} \biggl[ \sup_{s \in [0,1]} 
(c^* \sqrt{s} + W_{s})  \biggr] - C \sqrt{h}
 = \mathbb{E} \biggl[ \sup_{s \in [0,1]} 
(c^* \sqrt{s} + W_{s}) - (c^*+W_{1}) \biggr] + c^* - C \sqrt{h}.
\end{split}
\]
Therefore, by time reversal,
\begin{equation*}
\begin{split}
C \sqrt{h} &= {\mathbb E} \biggl[ \sup_{s \in [0,1]} 
(c^* \sqrt{s} + W_{s}) - (c^*+W_{1}) \biggr]
\\
&\geq {\mathbb E} \biggl[ \sup_{s \in [0,1]} 
\bigl( c^* (s-1) + W_{s} - W_{1} \bigr) \biggr]
\geq  
{\mathbb E} \biggl[ \sup_{s \in [0,1]} 
(- c^* s + W_{s})  \biggr],
\end{split}
\end{equation*}
which says that $c^*$ must be large when $h$ is small. In particular, we can assume $h$ small enough so that
$c^* \geq 1$. Then,
\begin{equation*}
C\sqrt{h}
\geq  {\mathbb E} \biggl[ \sup_{s \in [0,(c^*)^{-2}]} 
(- c^* s + W_{s})  \biggr] = \frac{1}{c^*} 
{\mathbb E} \biggl[ \sup_{s \in [0,1]} 
(- s + W_{s})  \biggr],
\end{equation*}
which proves that, for $h$ small enough,
$c^* \sqrt{h} \geq \beta$, for some constant $\beta >0$. This implies 
$\liminf_{h \downarrow 0} [e(t+h)-e(t)] \geq \beta/\alpha$, which is a contradiction.
\end{proof}

\section{Two candidates for approximate solutions}
\label{se:3}
In this section we present two alternative systems, which are candidates to be approximations of the nonlinear equation \eqref{nonlinear eq}.

\subsection{The particle system approximation}
\label{pres: particle system}

As noted above, one of the main motivations for studying \eqref{nonlinear eq} is the idea that it describes the behavior of a very large number of interacting spiking neurons in a fully connected network, each evolving according to the classical noisy integrate-and-fire model.  More precisely, this idea translates into the fact that we would like \eqref{nonlinear eq} to describe the behavior of the particle system
\begin{equation}
\label{particle system}
\left\{
\begin{split}
X_{t}^{i,N} &= X^{i, N}_{0} + \int_{0}^t b\bigl(X_{s}^{i, N}\bigr) ds + \frac{\alpha}{N} \sum_{j = 1}^N M_{t}^{j, N}  + W_{t}^i - M^{i, N}_t\\
X^{i, N}_0 &\stackrel{d}{=} X_0\ \mathrm{independent\ and\ identically\ distributed,}
\end{split}
\right.
\end{equation}
for $i\in\{1, \dots, N\}$ and $ t\geq0$ when $N$ is large. Here $(X^{i,N}_t)_{t\geq0}$ represents the electrical potential of the $i$th neuron, $X_0$ satisfies standing Assumption \ref{assumption 2}, $(W^{i}_t)_{t\geq0}$ are independent standard Brownian motions, and now $(M^{i, N}_t)_{t\geq0}$ is the process that counts the number of times the $i$th neuron has `spiked' up until time $t$.  Precisely, we define for $i\in\{1, \dots, N\}$ and $ t\geq0$
\[
M^{i, N}_t := \sum_{k\geq 1} \Ind_{[0,t]}(\tau_k^{i, N}),
\]
where $\tau_0^{i,N} = 0$ and
\begin{equation}
\label{eq: particle system spike times}
\tau_{k}^{i,N} := \inf \biggl\{ t > \tau_{k-1}^{i,N} : X_{t-}^{i,N}
+ \frac{\alpha}{N} \sum_{j =1}^N \bigl(M^{j,N}_{t} - M^{j,N}_{t-}\bigr)
 \geq 1 \biggr\},\quad k\geq1,
\end{equation}
which should be compared with \eqref{tau def} 
(and which is as involved as \eqref{tau def} since the definitions of 
$\tau_{k}^{i,N}$ and $M^{i,N}_{t}$ are fully coupled).
The idea is that the system spikes if
one of the particles reaches the threshold $1$, but this can cause other particles 
to instantaneously spike through the empirical mean type interaction.  
However, exactly as in the previous section, where we defined a solution to \eqref{nonlinear eq},
we must be careful about what we mean by a `physical' solution to the particle system \eqref{particle system}.  This is because there may in fact exist multiple solutions to \eqref{particle system} and \eqref{eq: particle system spike times} (see section on `non-physical' solutions below).  The `physical' solution we will identify is in fact the one in which we require the 
instantaneous spikes 
induced at a spike time to
be ordered in a
natural way, the first spike occuring
when one of the particles hits the barrier. See below for a precise description.
At time $t=\tau_{k}^{i,N}$,  $X_{t}^{i,N} = 
 X_{t-}^{i,N} - 1
+ (\alpha/N) \sum_{j =1}^N (M^{j,N}_{t} - M^{j,N}_{t-})$.
Again, it should also be noted that the presence of the `$-$' in $X_{t-}^{i,N}$ in \eqref{eq: particle system spike times} 
ensures that $M^{i,N}$ and $X^{i,N}$ are c\`adl\`ag.  

Anyway, the point is that the system \eqref{particle system} is mathematically equivalent to the one used by Ostojic, Brunel and Hakim in \cite{ostojic:brunel:hakim} to describe the behavior of a finite network of neurons, and that \eqref{nonlinear eq} is a good guess as to what happens in the limit as $N\to\infty$.  Indeed, the extremely well developed theory of mean-field/McKean-Vlasov equations provides many rigorous results about when an individual particle in a system that interacts through an empirical mean becomes independent 
in the limit as $N\to\infty$, and then behaves according to a distribution dependent (McKean-Vlasov) limit equation.
However, despite the use of such a result in \cite{ostojic:brunel:hakim}, we argue that the current situation is quite different to any that has been previously studied 
in the literature due to the nature of the nonlinearity.
Thus, one of the aims of this paper is to provide a complete rigorous proof of this convergence. 

\begin{rem}
The reader may argue that it makes more sense physically to replace the interaction term $N^{-1} \sum_{j = 1}^N M_{t}^{j, N}$ in \eqref{particle system} by $(N-1)^{-1}\sum_{j \neq i} M_{t}^{j, N}$, so that if a single neuron spikes at time $t$, it is reset from the threshold $1$ to $0$ (rather than to $\alpha/N$).  However, we choose to keep the stated interaction term since it renders the analysis notationally simpler, while remaining mathematically 
equivalent 
in the limit $N\to\infty$.
\end{rem}

\begin{rem}[Reformulation] 
Following Remark \ref{rem: reformulation limit}, it will be convenient to reformulate the particle system \eqref{particle system} in terms of the processes $(Z^{i, N}_t)_{t\geq0}$, defined by
\[
Z^{i, N}_t := X^{i, N}_t + M^{i, N}_t, \qquad t\geq 0.
\]
Then, similarly to \eqref{nonlinear eq Z}, the reformulated system is given by 
\begin{equation}
\label{particle system Z}
\begin{split}
&Z_{t}^{i, N} = Z^{i, N}_{0} + \int_{0}^t b\left(Z_{s}^{i, N}- M_{s}^{i, N} \right) ds + \frac{\alpha}{N} \sum_{j = 1}^N M_{t}^{j, N}  + W_{t}^i, \\
&M_{t}^{i, N} = \bigl \lfloor \bigl(\sup_{s \in [0,t]} Z_{s}^{i, N} \bigr)_{+} \bigr\rfloor =\sup_{s \in [0,t]}  \bigl \lfloor \bigl(Z_{s}^{i, N} \bigr)_{+} \bigr\rfloor,
\end{split}
\end{equation}
for all $t \geq 0$ and $i\in \{1, \dots, N\}$, where $Z^{i, N}_{0} = X^{i, N}_{0} \stackrel{d}{=} X_0$ are i.i.d and $X_0$ satisfies Assumption \ref{assumption 2}.  
We will refer to \eqref{particle system Z} as the $Z$-particle system (and the original system \eqref{particle system} as the $X$-particle system).
\end{rem}

\noindent {\bf Notation.} In the sequel, we will use the convenient notation
\begin{equation}\label{eq:def_en}
\bar{e}^N(t) = \frac{1}{N} \sum_{i=1}^N M_{t}^{i,N}, \quad t \geq 0.
\end{equation}
We will also often omit the superscript $N$ in the notations
$X^{i,N}$, $Z^{i,N}$, $M^{i,N}$ and $\tau^{i,N}_{k}$ for simplicity. When no confusion is possible, we thus write 
$X^{i}$, $Z^{i}$, $M^{i}$ and $\tau^i_{k}$ instead.

\subsubsection{\textbf{Non-physical solutions}}
\label{Non-physical solutions para}
As mentioned already, the particle system defined above is not well-posed, as it may admit a large number of solutions when $\alpha$ is close to $1$.

Actually, uniqueness may fail for several reasons. A first way for constructing different solutions is 
to allow one particle to admit several spikes at the same time. Indeed, consider the $Z$-system \eqref{particle system Z} with $b \equiv 0$ and suppose that $\alpha$ has the form $\alpha = 1 - 1/(2m)$, for some integer $m \geq 1$.  Suppose moreover that, at some time $t$, it holds that (the system being initialized at $Z_{0}^{i} = 0$, $i \in\{1,\dots, N\}$)
\begin{equation*}
\forall i\in\{1,\dots,N\}, \quad Z_{t-}^{i} = 1 - \delta_{i}, \ M_{t-}^{i}=0, 
\end{equation*}
with $\delta_{1} =0$ and $\delta_{i} \in ((i-2)/(4N),(i-1)/(4N))$ for $i = 2,\dots,N$,
which, by the support theorem for Brownian motion, happens with positive probability.
Then, the system is to spike at time $t$ since the first particle reaches the barrier, but the 
spike procedure may be arbitrarily chosen. Indeed, setting arbitrarily 
$M_{t}^{i} = \ell \geq 1$,
for all $i \in\{1, \dots, N\}$, the equation for $Z^{i}$ gives
\begin{equation*}
Z_{t}^{i} = 1 - \delta_{i} +  \frac{\alpha}{N} \sum_{i=1}^N 
\ell = \ell + 1 - \frac{\ell}{2m}- \delta_{i} 
\in \left[\ell + \frac{3}{4} - \frac{\ell}{2m}, \ell +1 \right),
\end{equation*}   
for all $i\in\{1, \dots, N\}$. Then, if $\ell/m \leq 1$,
$\lfloor ( \sup_{0 \leq s \leq t} Z_{s}^{i})_+ \rfloor = \ell = M_{t}^{i}$,
so that the second relationship in \eqref{particle system Z} is satisfied.  Since $\ell\leq m$ is arbitrary, the system \eqref{particle system Z} clearly does not possess a unique solution.
According to the discussion below, cases where $\ell \geq 2$ will be considered 
as \textit{non-physical}.

We give here a second example where uniqueness fails 
even if the property \(\mathbb{P}(\Delta M^i_t \leq 1) = 1\)
is fulfilled. 
We present it with \(N=3\) particles 
but it could be generalized in an obvious way.
Suppose that at some (random) time \(t\), \(M^1_{t-} = M^2_{t-} = M^3_{t-} = 0\),
\(Z^1_{t-} = 1\), \(Z^2_{t-}, Z^3_{t-} \in (1-2\alpha/3,1-\alpha/3)\). We can then make explicit two solutions to \eqref{particle system Z}: a first one where only particle 1 spikes, that is
\(
M^1_t = 1 \) and \( M^2_t = M^3_t = 0
\),
and a second one where all the particles spike at time \(t\), 
that is \( M^1_t = M^2_t = M^3_t = 1\). 
In this example, the second case will be said \textit{non-physical}.
Intuitively, particle 1 is indeed intended to spike `first'.  
After particle 1 has spiked, particles 2 and 3 are both strictly below $1$, which 
should prevent them from spiking immediately.



\subsubsection{\textbf{Physical solutions}} In view of the above discussion, the problem is that the ordering of the 
spike cascade is not determined i.e. how spiking neurons instantaneously cause others to spike.  We now argue that in fact there is a natural way of ordering this cascade, which then leads to unique `physical' solutions.

To this end, consider the $X$-particle system \eqref{particle system} and define the set
\[
\Gamma_0 := \{i\in\{1, \dots, N\} : X^{i}_{t-} =1\}.
\]
We say $t$ is a spike time when $\Gamma_0\neq\emptyset$.  At a spike time $t$, it is certain that all the neurons in $\Gamma_0$ spike. It then makes sense to introduce a second \textit{time axis}, called the \textit{cascade time axis at spike time $t$}, and to say that, along this axis, neurons in $\Gamma_{0}$ are the first ones to spike.

\label{page cascade}Then it is natural to determine exactly which other neurons spike \textit{given that those in $\Gamma_0$ have already spiked}.  Since the system says that all the other neurons should feel the effect of the ones in $\Gamma_0$ spiking by receiving a kick to their potential of size $\alpha|\Gamma_0|/N$, this in turn means that all the neurons in the set
\[
\Gamma_{1} :=  \left\{i\in\{1, \dots, N\}\backslash \Gamma_0:  X^{i}_{t-} + \alpha\frac{|\Gamma_0|}{N}\geq 1\right\},
\]
now have potentials that are instantaneously above the threshold, and so should also spike.  Thus we are now sure that all the neurons in $\Gamma_0\cup\Gamma_1$ spike at $t$.
Along the \textit{cascade time axis at time $t$}, the neurons in $\Gamma_{1}$ are 
said to spike after the neurons in $\Gamma_{0}$. Similarly, it is then natural to determine which other neurons spike, \textit{given that those in $\Gamma_0\cup\Gamma_1$ have already spiked}.  According to the definition of the system, this is exactly those in the set
\[
\Gamma_{2} :=  \left\{i\in\{1, \dots, N\}\backslash \Gamma_0\cup\Gamma_{1}:  X^{i}_{t-} +\alpha\frac{ |\Gamma_0\cup\Gamma_{1}|}{N}\geq 1\right\}.
\] 
By defining sequentially for general $k\in\mathbb{N}_0$
\begin{equation}
\label{eq:gamma:k+1}
\Gamma_{k+1} :=  \left\{i\in\{1, \dots, N\}\backslash \Gamma_0\cup\dots \cup\Gamma_k :  X^{i}_{t-} + \alpha\frac{|\Gamma_0\cup\dots \cup\Gamma_k|}{N}\geq 1\right\},
\end{equation}
the natural cascade is continued in this way until $\Gamma_{l} = \emptyset$ for some 
$l\in\{1, \dots, N\}$.  Note that this must happen, since by definition $\Gamma_{N} = \emptyset$
(if $\Gamma_{N} \not = \emptyset$, all the sets $\Gamma_{0},\dots,\Gamma_{N-1}$ contain
at least one element; since all of them are disjoint, we obtain a contradiction). 
Along the cascade time axis at time $t$, neurons in $\Gamma_{k+1}$ ($k+1< \ell$) spike
after neurons in $\Gamma_{0} \cup \dots \cup \Gamma_{k}$.
We can then define $\Gamma := \bigcup_{0\leq k \leq N-1} \Gamma_k$, which is exactly the set of all neurons that spike at time $t$, according the natural ordering of the spike cascade (see also  \cite{catsigeras_guiraud_2013}).
Having determined this, it is then straightforward to perform the final update of all the neurons in the network by setting
\begin{equation}
\label{eq:3:2:14:1}
X^{i}_{t} = 
X^{i}_{t-} + \frac{\alpha |\Gamma|}{N} \ \mathrm{if}\ i\not\in\Gamma, \quad
X^{i}_t =
X^{i}_{t-} + \frac{\alpha |\Gamma|}{N} - 1  \ \mathrm{if}\ i\in\Gamma.
\end{equation}
Note that now $X^{i}_t <1$ for all $i\in\{1, \dots, N\}$.  Indeed, if $i\not\in\Gamma$, then $i$ must be such that 
\[
X^{i}_{t-} + \frac{\alpha|\Gamma|}{N}<1 \quad \Rightarrow\quad  X^{i}_{t} <1.
\]
On the other hand, if $i\in\Gamma$ then, since $|\Gamma|\leq N$ and $\alpha <1$,
\[
X^{i}_{t} = X^{i}_{t-} + \frac{\alpha|\Gamma|}{N} - 1 \leq X^{i}_{t-} + \alpha - 1 <  
X^{i}_{t-} \leq 1.
\]
The above idea is completed by the following lemma.
\begin{lem}
\label{lem: existence and uniqueness particle system}
There exists a unique solution to the particle system \eqref{particle system} such that,
whenever $t$ is a spike time, the entire system jumps according to 
\[
X^{i}_t =X^{i}_{t-} + \frac{\alpha |\Gamma|}{N}  \ \mathrm{if}\ i\not\in\Gamma,\quad 
X^{i}_t =X^{i}_{t-} + \frac{\alpha|\Gamma|}{N} - 1  \ \mathrm{if}\ i\in\Gamma.
\]
where $\Gamma \subset \{1, \dots, N\}$ is as above. 
Such a solution will be known as a `physical' solution.
\end{lem}

\begin{proof}
It is clear that in between spike times of the system there is no problem of uniqueness (since the particles only interact at spike times).  Therefore, since we have specified a unique jumping procedure, any solution must be unique.

The proof of existence is more challenging. The issue is to prove that spike times of the system do not accumulate. We feel it is more convenient to give it at this stage of the paper, but the reader may skip ahead on a first reading. 

We in fact prove the existence of a solution to the associated $Z$-system \eqref{particle system Z} (this is completely equivalent to the existence of a solution to the original system \eqref{particle system}) with the given 
spike cascade. For any $1 \leq i \leq N$, we define $(Y^{1,i}_{t})_{t \geq 0}$ as the solution of the SDE
\begin{equation*}
Y_{t}^{1,i} = Z_{0}^i + \int_{0}^t b\bigl(Y_{s}^{1,i}\bigr) ds +  W_{t}^i, \quad t \geq 0. 
\end{equation*}
We set
$\tau^{1,i} = \inf \{ t \geq 0 : Y_{t}^{1,i} \geq 1\}$, $1 \leq i \leq N$.
Clearly, we have 
$0 < \tau^{1,i} < \infty$ (a.s.), so that $0 < \inf_{1 \leq i \leq N}(\tau^{1,i}) < \infty$ (a.s.). 
For $t \in [0,\tau^{1})$, with $\tau^1 = \inf_{1 \leq i \leq N}(\tau^{1,i})$, we set
\begin{equation*}
Z_{t}^i = Y_{t}^{1,i}, 
 \quad M_{t}^i = 0, \quad 0 \leq t < \tau^1, \quad 1 \leq i \leq N.
\end{equation*}
At time $\tau^1$, there exists $i^1 \in \{1,\dots,N\}$ such that $\tau^{1} = \tau^{1,i^1}$.
We then denote by $\Gamma^{(1)}$ the set of particles that spike at $\tau^1$
according to the physical procedure summarized in \eqref{eq:3:2:14:1}
(pay attention that $\Gamma^{(1)}$ stands for the $\Gamma$ in \eqref{eq:3:2:14:1} and
not for $\Gamma_{1}$: the positions of the indices are different). Then, according to the cascade, we know that the kick that 
 the particle $Z^i$ receives at time $\tau^1$
is $\alpha \vert \Gamma^{(1)}\vert/N$, so that  
%
\begin{equation*}
Z_{\tau^1}^i  = Y_{\tau^1}^{1,i} + \alpha \frac{\vert \Gamma^{(1)} \vert}{N}
, \quad 1 \leq i \leq N.
\end{equation*}
For a coordinate $i \in \Gamma^{(1)}$, it holds $M_{\tau^1}^i = 1$. Since $Z_{\tau^1-}^i \leq 1$ and 
the kick received by $i$ is less than $\alpha$, it holds $Z_{\tau^1}^i \leq 1+ \alpha < M_{\tau^1}^i + 1$. Moreover, we must also have $Z_{\tau^1}^i \geq 1$ so that 
$M_{\tau^1}^i \leq Z_{\tau^1}^i < M_{\tau^1}^i +1$, that is 
$\lfloor Z_{\tau^1}^i \rfloor = M_{\tau^1}^i$. Since 
$Z_{\tau^1}^i = \sup_{s \in [0,\tau^1]} Z_{s}^i
= ( \sup_{s \in [0,\tau^1]} Z_{s}^i)_{+}$, we deduce
$M_{\tau^1}^i  = \lfloor ( \sup_{s \in [0,\tau^1]}
Z_{s}^i)_{+} \rfloor$. On the other hand, for a coordinate 
$i \not \in \Gamma^{(1)}$, it holds that
$M_{\tau^1}^i = M_{\tau^1-}^i = 0$, and
$\sup_{s \in [0,\tau^1]} Z_{s}^i <1$, so that 
$M_{\tau^1}^i  = \lfloor ( \sup_{s \in [0,\tau^1]}
Z_{s}^i)_{+} \rfloor$ as well.
 
For any $1 \leq i \leq N$, we then define $(Y^{2,i}_{t})_{t \geq 0}$ as the solution of the SDE
\begin{equation*}
\begin{split}
Y_{t}^{2,i} 
 &= Z_{\tau^1}^{i} +
\int_{\tau^1}^t b\bigl(Y_{s}^{2,i} - M_{\tau^1}^{i}\bigr) ds  +
  \bigl( W_{t}^i - W_{\tau^1}^i \bigr), \quad t \geq \tau^1.
 \end{split}
\end{equation*}
\color{black}
Define then $\tau^{2,i} = \inf \{ t \geq \tau^1 : Y^{2,i}_{t} \geq M_{\tau^1}^i  + 1\}, \ 1 \leq i \leq N$. 
Since $Z^i_{\tau^1} < M^i_{\tau^1} + 1$, we have 
$\tau^{2,i} > \tau^1$.  
Then, with $\tau^2 = \inf_{1 \leq i \leq N} (\tau^{2,i})$, we set 
\begin{equation*}
Z_{t}^i = Y_{t}^{2,i}, \quad \tau^1 < t < \tau^2.
\end{equation*}
The spike procedure at time $\tau^2$ is defined according to the process summarized 
in \eqref{eq:3:2:14:1}, the set of particles jumping at $\tau^2$ being denoted by
 $\Gamma^{(2)}$. By iteration, we build an increasing sequence of stopping times $(\tau^k)_{k \geq 0}$ (with $\tau^0=0$) 
such that
\begin{equation}
\label{eq:12:5:3}
Z_{t}^i = Z_{0}^i + \int_{0}^t b \bigl( Z_{s}^i - M_{\tau^k}^i \bigr) ds
 + \frac{\alpha}{N} \sum_{j=1}^N M_{\tau^k}^j  +  W_{t}^i,
\end{equation}
for $1 \leq i \leq N$ and $\tau^{k} < t < \tau^{k+1}$, $k \geq 0$, 
with $\tau^{k+1} =   \inf_{1 \leq i \leq N} \tau^{k+1,i}$, where 
$\tau^{k+1,i} = \inf \{ t > \tau^k : Z_{t-}^i \geq M_{\tau^k}^i  + 1\}$, $1 \leq i \leq N$. The set of particles that jump at $\tau^{k}$ is then denoted by $\Gamma^{(k)}$. 
With such a construction, we notice that
\begin{equation}
\label{eq:12:5:1}
M_{t}^i = \bigl \lfloor \bigl( \sup_{s \in [0,t]}  Z_{s}^i \bigr)_{+} \bigr \rfloor,
\end{equation} 
for 
$1 \leq i \leq N$ and $t \leq \tau^k$, for any $k \geq 1$. Indeed, at time $\tau^k$, the proof is the same as at time $\tau^1$. At any time $t \in (\tau^k,\tau^{k+1})$, the equality follows from the fact that $Z_{t}^i < M_{\tau^k}^i + 1$.

To finish with the proof of existence, we prove that $\tau^k \rightarrow + \infty$ as $k \rightarrow + \infty$. 
Noting from \eqref{eq:12:5:1} that the drift part in \eqref{eq:12:5:3}
can be bounded by
$\vert b(Z_{s}^i - M_{s}^i) \vert \leq C ( 1 + \sup_{0 \leq r \leq s} \vert Z_{r}^i \vert)$,
for some constant $C \geq 0$, and taking the empirical mean over $i \in \{1,\dots,N\}$, we deduce that, for $t \leq \tau^k$, for $k \geq 1$,
\begin{equation*}
\begin{split}
\frac{1-\alpha}{N} \sum_{i=1}^N \sup_{s \in [0,t]} \vert Z_{s}^i \vert &\leq 
\frac{1}{N} \sum_{i=1}^N \bigl( \vert X_{0}^i \vert
+ \sup_{s \in [0,t]} \vert W_{s}^i \vert \bigr)
+ C \int_{0}^t 
\bigl( 1+ 
\frac{1}{N} \sum_{i=1}^N \sup_{r \in [0,s]} \vert Z_{r}^i \vert
\bigr) ds. 
\end{split}
\end{equation*}
We deduce from Gronwall's lemma that, for any $1 \leq i \leq N$ and $t \leq \tau^k$, for $k \geq 1$,
\begin{equation*}
\frac{1}{N} \sum_{i=1}^N \sup_{s \in [0,t]} \vert Z_{s}^i \vert
\leq C \exp(Ct) \biggl[ t +  \frac{1}{N} \sum_{i=1}^N \bigl( \vert X_{0}^i \vert + 
\sup_{s \in [0,t]} \vert W_{s}^i \vert \bigr) \biggr],
\end{equation*}
for a possibly new value of $C$. 
By \eqref{eq:12:5:1}, the same bound holds for $N^{-1} \sum_{i=1}^N M_{t}^i$. 
Going back to \eqref{eq:12:5:3} and using Gronwall's lemma again, we deduce, that for any $T>0$,
there exists a constant $C_{T} \geq 0$ such that, for any $1 \leq i \leq N$ and $t \leq \tau^k \wedge T$, for $k \geq 1$,
\begin{equation}
\label{eq:12:5:2}
\sup_{s \in [0,t]} \vert Z_{s}^i \vert \leq 
C_{T} \biggl[ 1+ \vert X_{0}^i \vert + \sup_{s \in [0,t]} \vert W_{s}^i \vert + 
\frac{1}{N} \sum_{j=1}^N  \bigl( \vert X_{0}^j \vert 
+ \sup_{s \in [0,t]} \vert W_{s}^j \vert \bigr) \biggr].  
\end{equation}
Again, by \eqref{eq:12:5:1}, the same bound holds
for $M_{t}^i$. In particular, if $\tau^k \leq T$, 
\begin{equation*}
\sup_{i \in \{1,\dots,N\}} M_{\tau^k}^{i} \leq C_{T} 
\Bigl( 1 + \sup_{i \in \{1,\dots,N\}}
\vert X_{0}^i \vert +
\sup_{i \in \{1,\dots,N\}}
\sup_{s \in [0,T]} \vert W_{s}^i \vert \Bigr).
\end{equation*}
Applying the above inequality with $Nk$ instead of $k$, we notice that 
$\sup_{i \in \{1,\dots,N\}} M_{\tau^{Nk}}^{i}$ is larger than $k$ (as, at time $\tau^{Nk}$, there have been $Nk$ spikes in the system, so that at least one of the particles has spiked at least $k$ times). Therefore, if $\tau^{Nk} \leq T$, then
\begin{equation*}
k \leq  C_{T} 
\Bigl( 1 + \sup_{i \in \{1,\dots,N\}}
\vert X_{0}^i \vert +
\sup_{i \in \{1,\dots,N\}}
\sup_{s \in [0,T]} \vert W_{s}^i \vert \Bigr).
\end{equation*}
In particular, the sequence $(\tau^k)_{k \geq 1}$ cannot have a finite limit $T$, as otherwise, passing to the limit, we would get 
$ \sup_{i \in \{1,\dots,N\}}
\vert X_{0}^i \vert + \sup_{i \in \{1,\dots,N\}} 
\sup_{s \in [0,T]} \vert W_{s}^i \vert = + \infty.$ 
\end{proof}

Physical solutions to the particle system satisfy a discrete version of
(4) in Definition \ref{def:nonlineareq}, which motivates the 
notion of physical solutions to the original equation 
\eqref{nonlinear eq}:
\begin{prop}
\label{prop:16:2:14:1}
For a physical solution as in Lemma \ref{lem: existence and uniqueness particle system}, 
it holds that,
 for all $t \geq 0$, 
\begin{equation}
\label{inf condition}
 N \Delta \bar{e}^N(t) = 
\sum_{i=1}^N \bigl( M^{i}_t - M^{i}_{t-} \bigr) = \inf\biggl\{k\in\{0,\dots,N\}: 
\sum_{i=1}^N{\mathbf 1}_{\{X^{i}_{t-} \geq 1 - \frac{\alpha k}{N}\}} \leq k \biggr\}.
\end{equation}
\end{prop}
\begin{proof} 
We use the description of a physical solution.
We know that the left-hand side in \eqref{inf condition} is 
equal to $\vert \Gamma\vert$ by definition of $\Gamma$ in page 
\pageref{eq:3:2:14:1}.
Clearly $\vert \Gamma \vert = \sum_{k=0}^{l} \vert \Gamma_{k} \vert$
where $l$ is the largest integer such that $\Gamma_{l} \not = \emptyset$. 
Then, for $k \in \{ \vert \Gamma_{0} \cup \Gamma_{1} \cup \dots \cup \Gamma_{j-1} \vert,\dots,
 \vert \Gamma_{0} \cup \Gamma_{1} \cup \dots \cup \Gamma_{j}\vert-1 \}$
 and $j \in \{1,\dots,l\}$ (or $k\in\{0, \dots, |\Gamma_0| -1\}$ if $j =0$), 
\begin{equation*}
\sum_{i=1}^N{\mathbf 1}_{\{X^{i}_{t-} \geq 1 - \frac{\alpha k}{N}\}} 
\geq\sum_{i=1}^N{\mathbf 1}_{\{X^{i}_{t-} \geq 1 - \frac{\alpha \vert \Gamma_{0} \cup \Gamma_{1} \cup \dots \cup \Gamma_{j-1} \vert}{N}\}} 
= \sum_{\ell=0}^{j } \vert \Gamma_{\ell} \vert >  k,
\end{equation*}
the equality following from \eqref{eq:gamma:k+1},
proving that the right-hand side in \eqref{inf condition} is greater than or equal to $\vert \Gamma\vert$.
On the other hand, by construction,
$\sum_{i=1}^N{\mathbf 1}_{\{X^{i}_{t-} \geq 1 - \frac{\alpha}{N}\vert \Gamma \vert\}}
\leq \vert \Gamma \vert + \sum_{i \not \in \Gamma}
{\mathbf 1}_{\{X^{i}_{t-} \geq 1 - \frac{\alpha}{N}\vert \Gamma \vert \}}$,
but the last term in the right hand-side is zero. 
This shows that the right-hand side 
in \eqref{inf condition} is less than or equal to $\vert \Gamma \vert$.
\end{proof}
\color{black}

\subsection{The system with delays}
\label{subse:delay}
In this section we introduce a second approximation of the nonlinear system \eqref{nonlinear eq}, by introducing delays.  As we will see below (Proposition \ref{prop: delayed existence}), the advantage of doing this is that the resulting system has a global in time solution, for which the mean-firing rate $e'$ remains finite for any value of the parameter $\alpha$ and initial condition (recall that this is in contrast to the system without delays \eqref{nonlinear eq} which may `blow-up' in finite time for some parameter values: see Theorem \ref{thm:CCP}).  

The point is that the introduction of a delay prevents a macroscopic proportion of the neurons all spiking at the same time.  Intuitively, this is because, even if other neurons are close enough to the threshold to be induced to spike as a result of the first neuron spiking, this will occur only after a positive amount of time.

Given that the delayed system does not experience a blow-up and has a global solution (see below), part of our work is dedicated to the 
analysis of the solutions when the delay converges to zero (see 
Subsection \ref{results: existence of weak solutions}). However,
the purpose of this current subsection is simply to introduce the system with delays and to check well-posedness. \color{black} To this end, let $\delta >0$ and consider the equation
\begin{equation}
\label{delayed equation}
X^\delta_{t}= X_{0} + \int_{0}^{t}b(X^\delta_{s})ds+ \alpha e_{\delta}(t)+W_{t}-M^\delta_{t}, \quad t\geq0.
\end{equation}
Here, similarly to above, $M^\delta_t$ counts the number of time $(X^\delta_s)_{s \in [0,t]}$ reaches the threshold.  Precisely, we write

\begin{equation}
\label{delayed M}
\begin{split}
M^\delta_{t}&=\sum_{k\geq1}\Ind_{[0,t]}(\tau^\delta_{k}),
\end{split}
\end{equation}
where $\tau_0^{\delta} = 0$ and, for $k \geq 1$,
\begin{equation}
\label{e_delta}
\tau^\delta_{k} = \inf\left\{t>\tau^\delta_{k-1} : X^\delta_{t-} + \alpha \Delta e_{\delta}(t) \geq 1\right\}, 
\quad
 \textrm{with} \
e_\delta(t) := 
\begin{cases}
0&\mathrm{if\ } t\leq\delta\\
\E(M^\delta_{t-\delta})&\mathrm{if\ } t>\delta.
\end{cases}
\end{equation}
\color{black}

We write the equation in this way, even though, as the following proposition shows, the delay guarantees 
that there is a unique solution to \eqref{delayed equation} such that $e_\delta$ is always continuously differentiable (so that $\Delta e_{\delta} \equiv 0$).  This makes any notion  of `physical' solutions irrelevant for the delayed equation.
As in Section \ref{se:2}, we take
$\alpha\in(0,1)$ and $(W_{t})_{t \geq 0}$ a standard real-valued Brownian motion,  
and we assume that Assumptions \ref{assumption 1} and \ref{assumption 2} are in force.

\begin{prop}
\label{prop: delayed existence}
Let $T>0$ and  $\alpha\in(0,1)$.  Then 
there exists a unique 
c\`adl\`ag process $(X^\delta_t, M^\delta_t)_{t \in [0,T]}$,
such that
$(M_{t}^{\delta})_{t \geq 0}$ has integrable marginal distributions,
satisfying
 \eqref{delayed equation} and \eqref{delayed M}.
The resulting map $e_{\delta}$ is continuously differentiable. 
\end{prop}


\begin{proof}
\noindent\textit{Step 1: Solution on $[0, \delta]$.}  For $t\leq \delta$ \eqref{delayed equation} reads
\begin{equation}
\label{equation delta}
X^\delta_{t}=X_0 + \int_{0}^{t}b(X^\delta_{s})ds + W_{t}-M^\delta_{t}.
\end{equation}
Clearly, this has a unique strong solution for $t\leq \delta$ (there is no difficult nonlinear term).
Moreover, by \cite[Proposition 4.5]{DIRT}, we have that $[0, \delta] \ni t\mapsto \E(M^\delta_t)$ is continuously differentiable, and moreover 
\begin{equation}
\label{IFR delta}
\frac{d}{dt}\E(M_{t}^\delta) = - \frac{1}{2}\int_0^t\partial_y p_0^{\delta}(t-s, 1)\frac{d}{ds}\E(M_s^\delta)ds - \frac{1}{2}\partial_y p_{X_0}^{\delta}(t, 1), \qquad t\in[0, \delta],
\end{equation}
where, for a random variable $\chi_{0}$, $p^{\delta}_{\chi_{0}}$ represents the density of the process $X^\delta$ killed at 1 with $X^{\delta}_0= \chi_0$, namely 
$p^{\delta}_{\chi_{0}}(t,y) dy := \P(X_{t}^{\delta} \in dy, \sup_{s \in [0,t]} X_{s}^{\delta} < 1 \vert X^{\delta}_{0}= \chi_{0})$.  Note that the shift by $s$ that is required in \cite[Proposition 4.5]{DIRT} is not necessary here, as $\eqref{equation delta}$ is time homogeneous.
By continuous differentiability,
$0 \leq \sup_{t \in [0,\delta]}
(d/dt)\E(M_{t}^\delta) <+\infty$.

\textit{Step 2: Solution on $[0, 2\delta]$.} For $t\leq 2\delta$ \eqref{delayed equation} reads
\begin{equation}
\label{equation 2delta}
X^\delta_{t}=X_0 + \int_{0}^{t}b(X^\delta_{s})ds + \alpha e_\delta(t) + W_{t}-M^\delta_{t},
\end{equation}
where $e_\delta(t) = 0$ on $[0, \delta]$ and $\E(M^\delta_{t-\delta})$ on $[\delta, 2\delta]$.  

We now claim that $[0, 2\delta]\ni t \mapsto e_\delta(t)$ is continuously differentiable.  This is clearly the case on $[0, \delta]$, and on $[\delta, 2\delta]$ by Step 1. It remains to check that it is also true at $t=\delta$, i.e. 
$\lim_{t\downarrow \delta}  e_{\delta}'(t)  = \lim_{t\downarrow 0}[d/dt]\E(M^\delta_{t}) =0$.  This follows from \eqref{IFR delta}, since $\lim_{t \downarrow 0} \partial_{y}p^{\delta}_{\chi_0}(t,1) =0$
for $\chi_{0}$ satisfying Assumption \ref{assumption 2}, see \cite[Lemma 4.2]{DIRT}. 
In particular, $0 \leq \sup_{t \in [0,2\delta]}
(d/dt)\E(M_{t}^\delta) <+\infty$.
By \cite[Section 3]{DIRT}, equation
\eqref{equation 2delta} has a unique strong solution.

Moreover, since $[0, 2\delta]\ni t \mapsto e_\delta(t)$ is continuously differentiable, we can apply  
\cite[Proposition 4.5]{DIRT} on this new interval to see that $[0, 2\delta] \ni t\mapsto \E(M^\delta_t)$ is continuously differentiable.

\textit{Step 3: Solution on $[0, 3\delta]$:}  We replicate Step 2 by proving that $[0, 3\delta]\ni t \mapsto e_\delta(t)$ is continuously differentiable. Indeed, $[\delta, 3\delta] \ni t \mapsto  e_\delta(t)$ is continuously differentiable by Step 2, and $ e_\delta(t)$ is equal to $0$ on $[0, \delta]$, but, as already noted, 
$[d/dt]|_{t=\delta^+}e_{\delta}(t) = 0$, so that the `join' is continuously differentiable.

\textit{Conclusion:}  Let $T>0$.  One may iterate this procedure up until Step  $\lceil T/\delta\rceil$.  This will yield the fact that there exists a unique strong solution to the system given by \eqref{delayed equation} and \eqref{delayed M} up until time $T$, such that $[0, T]\ni t \mapsto \E(M_{t}^\delta)$ is continuously differentiable.  
\end{proof}

\section{Results}
\label{se:results}
Given the setup described in the previous sections, 
we are now in a position to present our main results.   The objective is 
to pass to the limit as $N \rightarrow \infty$ in the particle system described in Subsection \ref{pres: particle system}  
and as $\delta \rightarrow 0$ in the delayed equation described in Subsection \ref{subse:delay}, deriving
as a by-product a new global in time solvability result for the 
original model \eqref{nonlinear eq} including solutions that blow up. 

With this in mind, the thrust of the paper is to identify a very convenient topology for tackling both problems. 
Basically, the strategy is to make use of the so-called M1 Skorohod topology, which is different from the more famous 
J1 topology and which turns out to be much more adapted to the problem at hand.  
The reason is that relative compactness for the M1 topology is indeed easily checked for sets of monotone c\`adl\`ag 
functions, which exactly fits the nature of the process $(M_{t})_{t \geq 0}$ in \eqref{nonlinear eq}.

\subsection{The M1 topology}
\label{sec:M1top}
We first supply the reader with some reminders about the M1 topology. For a complete overview, we refer to the original paper by 
Skorohod \cite{skorohod} and to the monograph by Whitt \cite{whitt:monograph}. 
We denote by $\hat{\mathcal D}\left([0,T], \R\right)$ the space of c\`adl\`ag functions from $[0, T]$ to $\R$ that are left-continuous at time $T$ \footnote{The condition forcing elements in $\hat{\mathcal D}([0,T],\R)$ 
to be left-continuous at the terminal time is implicitly done in 
Whitt \cite{whitt:monograph}: in Theorem 12.2.2 therein, the piecewise constant functions used for approximating c\`adl\`ag functions on $[0,T]$ are precisely assumed to be continuous at terminal time $T$.\label{footnote:whitt}}.  For a function $f\in\hat{\mathcal D}\left([0,T], \R\right)$, we denote by $\mathcal{G}_f$ the completed graph of $f$ i.e.
\[
\mathcal{G}_f := \left\{(x, t)\in\R\times[0, T]: x\in[f(t-), f(t)]\right\},
\]
where $[f(t-),f(t)]$ stands for the non-ordered segment between $f(t-)$ and $f(t)$ ($f(t-)$ could be bigger than $f(t)$). We define an order on $\mathcal{G}_f$ in the following way: for $(x_1, t_1), (x_2, t_2) \in \mathcal{G}_f$, we say that $(x_1, t_1)\leq(x_2, t_2)$ if either $t_1<t_2$, or $t_1= t_2$ and $|f(t_1-) - x_1|\leq |f(t_1-) - x_2|$.  In other words this is the natural order when the graph $\mathcal{G}_f$ is traced out from left to right. We then define a \textit{parametric representation} of $\mathcal{G}_f$ as being a continuous  function $(u, r)$ that maps $[0,T]$ onto $\mathcal{G}_f$ that is non-decreasing with respect to the order on $\mathcal{G}_{f}$ defined above i.e.
\[
(u,r): [0,T] \ni t\mapsto (u(t), r(t)) \in \mathcal{G}_f,
\]
where $u\in\mathcal{C}([0, T], \mathbb{R}), r\in\mathcal{C}([0, T], \R)$.  We define $\mathcal{R}_f$ as the set of all parametric representations of $\mathcal{G}_f$.  A parametric representation of $\mathcal{G}_f$ is thus a way of tracing it out 
`without going back on oneself' with respect to the natural order of the graph.

For $f_1, f_2 \in \hat{\mathcal D}\left([0,T], \R\right)$ we finally define the M1 distance between them as 
\[
d_{M_1}(f_1, f_2) := \inf_{\substack{(u_j, r_j)\in\mathcal{R}_{f_j}\\j=1,2}}\left\{\|u_1 - u_2\| \vee \|r_1 - r_2\|\right\},
\]
where $\|\cdot\|$ is the usual supremum norm on $\mathcal{C}([0, T], \mathbb{R})$.
In order to characterize the convergence in M1, we define for $f\in \hat{\mathcal{D}}\left([0,T], \R\right), t\in[0, T]$ and $\delta >0$,
\begin{equation}
\label{M1 wT}
w_{T}(f, t, \delta) := \sup_{0\vee(t-\delta)\leq t_1<t_2<t_3\leq T\wedge(t+\delta)}\Big\|f(t_2) - [f(t_1), f(t_3)]\Big\|
\end{equation}
where 
\[
\Big\|f(t_2) - [f(t_1), f(t_3)]\Big\| = \inf_{\theta\in[0,1]}\Big|\theta f(t_1) + (1-\theta)f(t_3) - f(t_2)\Big|
\] 
is the distance between $f(t_2)$ and the set $[f(t_1), f(t_3)]$.
In particular, if a function $f\in \hat{\mathcal{D}}\left([0,T], \R\right)$ is monotone (non-increasing or non-decreasing), then $w_T(f, t, \delta) = 0$. 

From \cite[Theorems 12.5.1, 12.4.1, 12.12.2]{whitt:monograph}, 
we have:
\begin{thm}
\label{thm:M1:1}
A sequence of functions $(f_n)_{n\geq1}\subset\hat{\mathcal D}\left([0,T], \R\right)$ converges to some $f\in\hat{\mathcal D}\left([0,T], \R\right)$ in the M1 topology if and only if $f_n(t) \to f(t)$ for each $t$ in a dense subset of full Lebesgue measure of $[0, T]$ that includes $0$ and $T$, and 
\[
\lim_{\delta \rightarrow 0} \limsup_{n\to\infty}\sup_{t\in[0, T]} w_{T}(f_n, t, \delta) = 0.
\]
\end{thm}
\begin{thm}
\label{thm:M1:2}
Suppose that the sequence $(f_n)_{n\geq1}\subset\hat{\mathcal D}\left([0,T], \R\right)$ converges to some $f\in\hat{\mathcal D}\left([0,T], \R\right)$ in the M1 topology.  Then for all points $t\in[0, T]$ at which $f$ is continuous
it holds that
\[
\lim_{\delta \rightarrow 0} 
\limsup_{n\to\infty} \sup_{s\in[0\vee(t-\delta), T\wedge(t+\delta)]}|f_n(s) - f(s)| = 0.
\]
In particular, if $f$ is continuous on the entire interval $[0,T]$, then the convergence of $(f_{n})_{n \geq 1}$ to $f$ is
uniform. Moreover, if $f_n$ is monotone for each $n$, then $f_n\to f$ in M1 if and only if $f_n(t) \to f(t)$ for all $t$ in a dense subset of full Lebesgue measure of $[0, T]$ including $0$ and $T$.
\end{thm}
\begin{thm}
\label{thm:M1:3}
A subset $A$ of $\hat{\mathcal D}\left([0, T], \R\right)$ has compact closure in the M1 topology if and only if
$\sup_{f\in A}\|f\| < \infty$ and 
\[
\lim_{\delta\to0}\sup_{f\in A}\Bigl\{\Bigl(\sup_{t\in[0,T]}w_{T}(f, t, \delta)\Bigr)
\vee v_{T}(f, 0, \delta) \vee v_{T}(f, T, \delta)\Bigr\} =0
\]
where $\displaystyle 
v_{T}(f, t, \delta) :=\sup_{0\vee(t-\delta)\leq t_1 \leq t_2\leq T\wedge(t+\delta)}\bigl|f(t_1) - f(t_2)\bigr|.$
\end{thm}


Finally, we mention that $\hat{\mathcal D}([0, T], \R)$, endowed with M1, is Polish, and that the Borel $\sigma$-field coincides
with the $\sigma$-field generated by the evaluation mappings (see  
\cite[Page 8]{whitt:paper}). 
This guarantees that the law of a process over $\hat{\mathcal D}([0, T], \R)$, endowed with M1, is characterized by its finite-dimensional distributions. The Polish property renders the Skorohod representation theorem licit, both on 
$\hat{\mathcal D}([0, T], \R)$ and on ${\mathcal P}(\hat{\mathcal D}([0, T], \R))$,
which is defined as the set of probability measures on $\hat{\mathcal D}([0, T], \R)$ (see
\cite[Theorem 6.7]{billingsley} and
\cite[Chapter III, Theorem 1.7]{ethier:kurtz}). 
It also renders the Prohorov theorem licit (see \cite[Chapter 1, Section 5]{billingsley}):
we let the reader derive the tightness criterion from Theorem \ref{thm:M1:3} (see \cite[Theorem 12.12.3]{whitt:monograph}).


\subsection{Existence of weak solutions with simultaneous spikes}
\label{results: existence of weak solutions}
The purpose of this Section is to state two results showing that the existence of a solution to \eqref{nonlinear eq}
can be deduced by extracting weakly convergent subsequences either along the distributions of the particle
systems (as the number of particles 
$N$ tends to $+ \infty$), or along the distributions of the delayed systems (as the delay tends 
to $0$). 

\begin{thm}
\label{weak existence thm}
Given $T>0$ and the (physical) solution $((Z^{i, N}_t)_{t\in[0, T]})_{i=1, \dots, N}$ to the particle system \eqref{particle system Z}, consider the extended system
\begin{equation}
\label{eq:Z-tilde1}
\tilde{Z}^{i, N}_t :=
\begin{cases}
Z^{i, N}_t, &\mbox{if } t \leq T,
\\
W_{t}^{i} - W_{T}^{i} + Z^{i, N}_{T}, &\mbox{if } t\in (T, T+1],
\end{cases}
\end{equation}
for $i\in\{i, \dots, N\}$.
Define the empirical measure
$\bar{\mu}_N := \frac{1}{N}\sum_{i=1}^N \emph{Dirac}(\tilde{Z}^{i, N})$, which reads as a random variable with values in 
$\mathcal{P}(\hat{\mathcal D}([0,T+1], \R))$.
Then, denoting by $\Pi_N$ the law of $\bar{\mu}_N$, the family $(\Pi_{N})_{N \geq 1}$ 
is tight in $\mathcal{P}(\mathcal{P}(\hat{\mathcal D}([0,T+1],\R)))$ endowed with the 
topology of weak convergence inherited from the M1 topology. 

Moreover,
for any weak limit $\Pi_{\infty}$, 
for $\Pi_{\infty}$-almost every measure $\mu \in \mathcal{P}(\hat{\mathcal D}([0,T+1],\R))$, the canonical process
$(z_{t})_{t \in [0,T+1]}$ on 
$\hat{\mathcal D}([0,T+1],\R)$ generates, under $\mu$, a physical solution to \eqref{nonlinear eq Z}, and hence to \eqref{nonlinear eq}, up until time $T$ i.e.
\begin{enumerate}
\item under $\mu$, $z_{0}$ is distributed according to the law of $X_{0}$;
\item under $\mu$, 
$( z_t - z_{0} - \int_0^t b(z_s - m_s)ds - \alpha \langle \mu,m_{t} \rangle)_{t \in [0,T)}$
is a Brownian motion, where $m_{t} = \lfloor ( \sup_{0 \leq s \leq t} z_{s})_{+} \rfloor$
and $\langle \mu,m_{t} \rangle$ denotes the expectation of $m_{t}$ under $\mu$
($\langle \cdot, \cdot \rangle$ is the duality bracket between a probability measure and a measurable function); 
\item under $\mu$, (1), (2) and (4) in 
Definition \ref{def:nonlineareq} are fulfilled. 
\end{enumerate}
\end{thm}
\color{black}

\begin{rem}
In the usual terminology of SDEs, the solution 
to \eqref{nonlinear eq Z} as given by  Theorem \ref{weak existence thm}
is weak as the Brownian motion is part of the solution. 

The extension of the $Z$-processes
in \eqref{eq:Z-tilde1} to the interval $(T,T+1]$ permits to get for free 
uniform bounds on the modulus of continuity of the particles at the final time $T+1$, which 
is a requirement for tightness for the M1 topology (see Lemma \ref{lem:tightness1} below).
As recalled in Footnote 
\ref{footnote:whitt} on Page \pageref{footnote:whitt}, 
elements of $\hat{\mathcal D}([0,S],\R)$, for $S>0$, are assumed to be left-continuous
at $S$, which requires bounds on the modulus of continuity at $S$ when addressing 
questions of convergence or compactness. With $S=T$, we see that $\tilde Z^{i,N}$ (or equivalently ${Z}^{i,N}$)
may not be continuous at $S$, but with $S=T+1$, 
$\tilde{Z}^{i,N}$ is obviously continuous at $S$, with the modulus of continuity 
at $S$ being controlled by the Brownian part only.  
Although rather arbitrary, the reason why we include Brownian oscillations in the definition 
of $\tilde{Z}^{i,N}$ on $[T,T+1]$ will be made clear in Lemma \ref{lem:13:7:2} below. 
Basically, noise is needed to guarantee that the counting process 
$[0,+\infty) \ni t \mapsto M_{t}$
in \eqref{nonlinear eq Z} is stable in law under perturbation of the dynamics.  Moreover, this avoids introducing any distinction between 
the dynamics on 
$[0,T]$ and on $[T,T+1]$ in the application of the stability property.

By Cantor's diagonal argument, it is possible to construct a solution 
on a sequence of intervals $([0,T_{n}))_{n \geq 1}$, with $T_{n} \rightarrow + \infty$, and thus to prove existence 
in
infinite time.
\end{rem}

We have a similar result for the delayed equation:
\begin{thm}
\label{weak existence thm:delay}
Given $T>0$ and the family of solutions $((X_t^\delta)_{t\in[0, T]})_{\delta\in(0,1)}$ to the delayed equation \eqref{delayed equation}, with $X^\delta_{0}=X_{0}$ satisfying Assumption \ref{assumption 2}, consider the family of extended paths
\begin{equation}
\label{eq:Z-tilde2}
\tilde{Z}^{\delta}_t :=
\begin{cases}
X^{\delta}_{t} + M^{\delta}_{t}, &\mbox{if } t \leq T,
\\
W_{t} - W_{T} + \tilde{Z}^{\delta}_{T}, &\mbox{if } t\in (T, T+1].
\end{cases} 
\end{equation}
Define by $\mu^{\delta}$ the law of $(\tilde{Z}^{\delta}_{t})_{t \in [0,T+1]}$ on 
$\hat{\mathcal D}([0,T+1], \R)$. Then, the family $(\mu^{\delta})_{\delta \in (0,1)}$ 
is tight in $\mathcal{P}(\hat{\mathcal D}([0,T+1],\R))$ endowed with the 
topology of weak convergence inherited from the M1 topology. Moreover,
under any weak limit $\mu$ as $\delta$ tends to $0$, the canonical process
$(z_{t})_{t \in [0,T+1]}$ on 
$\hat{\mathcal D}([0,T+1],\R)$ generates a physical solution to \eqref{nonlinear eq Z}, and hence to \eqref{nonlinear eq}, until time $T$, 
in the sense that (1), (2) and (3) in 
Theorem \ref{weak existence thm} hold true. 
\end{thm}

\subsection{Convergence of the particle system and propagation of chaos}
\label{results: particle system}
An important corollary to Theorem \ref{weak existence thm} is that when we have uniqueness for equation \eqref{nonlinear eq}, we also have propagation of chaos for the particle system $((Z^{i, N}_t, M^{i,N}_t)_{t\in[0, T]})_{i=1, \dots, N}$ given by \eqref{particle system Z}:

\begin{thm}
\label{convergence thm}
Assume that there exists a unique physical solution $(X_t, M_t)_{t\geq 0}$ 
to \eqref{nonlinear eq} and denote by
$(Z_{t},M_{t})_{t \geq 0}$ the reformulated solution (as defined in Remark \ref{rem: reformulation limit}). 
Denote also
by $J$ the (at most countable) set of discontinuity points of the function 
$[0,+\infty) \ni t \mapsto {\mathbb E}(M_{t})$. 
Then, for any $S \in [0,+\infty) \setminus J$ and any $k \in \{1, \dots, N\}$,
\begin{equation}
 \label{eq:21:02:14:1}
 \bigl( (\hat{Z}^{1,N}_{s},\hat{M}^{1,N}_{s}),\dots,(\hat{Z}^{k,N}_{s},\hat{M}^{k,N}_{s}) \bigr)_{s \in 
 [0,S]} \Rightarrow {\P_{(Z_{s},\hat{M}_{s})_{s \in [0,S]}}^{\otimes k}} \quad \textrm{as} \ N \rightarrow + \infty,
 \end{equation}
on the space $[\hat{\mathcal D}([0,S],\R) \times \hat{\mathcal D}([0,S],\R)]^{k}$ equipped with the product 
topology induced by the M1 topology, where 
  \begin{equation*}
\bigl( \hat{Z}_{s}^{i,N}, \hat{M}_{s}^{i,N} \bigr) = 
 \left\{
 \begin{array}{ll}
 \bigl( Z_{s}^{i,N},M_{s}^{i,N} \bigr) &{\rm if} \ s <S,
  \\
 \bigl( Z_{S-}^{i,N}, M_{S-}^{i,N} \bigr) &{\rm if} \ s=S.
 \end{array}
 \right.
\end{equation*} 
Here, for a random variable $X$, $\P_X$ stands for the law of $X$, and $\Rightarrow$ indicates weak convergence.
Moreover, as $N \rightarrow + \infty$, on $\hat{\mathcal D}([0,S],\R)$ equipped with the M1 topology,
\begin{equation}
\label{eq:21:02:14:2}
\biggl( 
\frac1N \sum_{i=1}^N \hat{M}_{s}^{i,N}
\biggr)_{s \in [0,S]} \rightarrow \bigl( {\mathbb E}(M_{s}) \bigr)_{s \in [0,S]}
\quad \textrm{in \ probability}. 
\end{equation}
\end{thm}

\begin{rem} \emph{(i)}
In the case when the unique solution $(Z_{t},M_{t})_{t \geq 0}$ has a continuous firing function 
$e : [0,+\infty) \ni t \mapsto {\mathbb E}(M_{t})$, then 
the process $(Z_{t})_{t \geq 0}$ has continuous paths. Such a situation is guaranteed for some initial conditions and values of $\alpha$ by Theorem \ref{thm:DIRT}. Then, by 
Theorem \ref{thm:M1:2}, the weak convergence of the law of the particles $\hat{Z}^{1,N},\dots,\hat{Z}^{k,N}$ in \eqref{eq:21:02:14:1} holds on the space $[\hat{\mathcal D}([0,S],\R)]^{k}$ equipped with the product uniform topology.
Similarly, in such a case, the convergence in \eqref{eq:21:02:14:2} holds on $\hat{\mathcal D}([0,S],\R)$ equipped with the uniform topology.

\emph{(ii)} In \eqref{eq:21:02:14:1}, we could replace 
$(Z_{s})_{s \in [0,S]}$ 
by 
$(\hat Z_{s})_{s \in [0,S]}$ 
but this would be useless as $Z$ is continuous at point $S$
for any realization of the randomness: Since $S \in J$, 
the (deterministic) jump function $[0,+\infty) \ni t \mapsto \E (M_{t})$
is continuous at point $S$. 
Similarly, we could replace $(\hat{M}_{s})_{s \in [0,S]}$
by $(M_{s})_{s \in [0,S]}$, by noticing that, with probability 
$1$ under $\P$, $M_{S}=M_{S-}$, but this 
would be slightly abusive as the paths 
of $(M_{s})_{s \in [0,S]}$ are
in $\hat{\mathcal D}([0,S],\R)$ with probability $1$ only (and not 
for all realizations of the randomness).

\emph{(iii)} Convergence of the $X$-particles in \eqref{eq:21:02:14:1}
follows from the relationship 
$X_{t}^{i,N} = Z_{t}^{i,N} - M_{t}^{i,N}$, for $i \in \{1,\dots,N\}$.
However, since addition may not be continuous for the M1 topology (see Chapter 12 in \cite{whitt:monograph}), we cannot deduce the convergence 
of the $X$-particles on $\hat{\mathcal{D}}([0,S],\R)$. 
By Theorem
\ref{thm:M1:2}, the best we can say is that, for any $k\in\{1, \dots, N\}$, any 
$\ell \geq 1$ 
and any $t_{1},\dots,t_{\ell} \not \in J$, the law of the
 random vector 
$((X^{i,N}_{t_{j}},M_{t_{j}}^{i,N})_{i \in \{1,\dots,k\}})_{j \in \{1,\dots,\ell\}}$
converges towards the finite-dimensional marginals, at times $t_{1},\dots,t_{\ell}$, of $k$ independent copies of $(X_t,M_{t})_{t \geq 0}$.

\emph{(iv)} Finally, we emphasize that \eqref{eq:21:02:14:2} is the keystone to switch from the finite system of particles to the dynamics of the McKean-Vlasov type.
\end{rem}

\subsection{Convergence of the delayed system}
\label{results: delays system}
Here is the analogue of the previous result for the delayed system $((Z^\delta_t, M^\delta_t)_{t\geq 0})_{\delta\in(0, 1)}$, where $Z^\delta_t := X^\delta_t + M^\delta_t$ and $(X^\delta_t, M^\delta_t)_{t\geq 0}$ is a solution to \eqref{delayed equation}.

\begin{thm}
\label{convergence thm:delay}
Assume that there exists a unique physical solution $(X_t, M_t)_{t\geq 0}$ 
to \eqref{nonlinear eq} and denote by
$(Z_{t},M_{t})_{t \geq 0}$ the reformulated solution (as defined in Remark \ref{rem: reformulation limit}). 
Denote also
by $J$ the (at most countable) set of discontinuity points of the function 
$[0,+\infty) \ni t \mapsto {\mathbb E}(M_{t})$. 
Then, for any $S \in [0,+\infty) \setminus J$,
\begin{equation}
 \label{eq:21:02:14:5}
 \bigl(Z^{\delta}_{s},\hat{M}^{\delta}_{s}\bigr)_{s \in [0,S]} 
 \Rightarrow \P_{(Z_{s},\hat{M}_{s})_{s \in [0,S]}}  \quad \textrm{as} \ \delta \rightarrow 0,
 \end{equation}
on the space $\hat{\mathcal D}([0,S],\R) \times \hat{\mathcal D}([0,S],\R)$ equipped with the product 
topology induced by the M1 topology, where 
$  \hat{M}_{s}^{\delta} = 
{M}_{s}^{\delta}$ if $s <S$ and $\hat{M}_{s}^{\delta} = 
{M}_{S-}^{\delta}$ if $s=S$.
Moreover,
\begin{equation}
\label{eq:21:02:14:6}
\bigl( 
{\mathbb E} \bigl( \hat{M}_{s}^{\delta} \bigr)
\bigr)_{s \in [0,S]} \rightarrow \bigl( {\mathbb E}(M_{s}) \bigr)_{s \in [0, S]}
\end{equation}
as $\delta \rightarrow 0$, on $\hat{\mathcal D}([0,S],\R)$ equipped with the M1 topology. 
\end{thm}

\section{Proofs}
\label{se:proofs}

\subsection{Preliminary estimates for the particle system}

This first subsection is devoted to the proof of two preliminary
technical lemmas.  The first one will be used for establishing suitable tightness properties of the 
particle system, while the second one will be needed to show that in the limit the solution does indeed satisfy the 
required properties to be physical in the sense defined above.

Throughout the section $((Z_{t}^{i},M_{t}^{i})_{t\geq0})_{i=1,\dots,N}$
will denote the physical solution to \eqref{particle system Z}. We
start with a moment estimate: \begin{lem} \label{lem:particle:5}
For any $p\geq1$ and $T\geq0$, there exists $C_{T}^{(p)} \geq 0$,
independent of $N$, such that 
\[
\forall i\in\{1,\dots,N\},\quad{\mathbb{E}}\bigl[\sup_{t \in [0,T]}\vert Z_{t}^{i}\vert^{p}+\bigl(M_{T}^{i}\bigr)^{p}\bigr]\leq C_{T}^{(p)}.
\]
 \end{lem}

\begin{proof} The proof is a consequence of \eqref{eq:12:5:1},
\eqref{eq:12:5:2} and Assumptions \ref{assumption 1} and \ref{assumption 2}. \end{proof}

 \begin{lem} \label{cor:particle:1}
For all $\eta>0$ there exists a constant $\lambda(\eta)>0$ that is
independent of $N$ (but depends on the constant $\varepsilon_{0}$
in Assumption~\ref{assumption 2}), such that 
\[
\P\bigl\{\forall t\in[0,\lambda(\eta)],\quad\bar{e}^{N}(t)\geq\bigl(\lambda(\eta)\bigr)^{-1}t^{1/4}\bigr\}\leq\eta,
\]
 where \ensuremath{\bar{e}^{N}(t)}
 is defined by \eqref{eq:def_en}. \end{lem} \begin{proof} Given
$T\in(0,1)$, define $\tau$ by 
\[
\tau=\inf\biggl\{ t\geq0:\frac{1}{N}\sum_{i=1}^{N}{\mathbf{1}}_{\{M_{t}^{i}\geq1\}}\geq T\biggr\}\quad(\inf\emptyset=+\infty),
\]
 which is the first time the proportion of particles that have spiked
at least once is bigger than $T$. For $t<\tau\wedge T$, the Cauchy-Schwarz
inequality yields 
\begin{equation}
\begin{split}\frac{1}{N}\sum_{i=1}^{N}M_{t}^{i} & \leq\biggl(\frac{1}{N}\sum_{i=1}^{N}{\mathbf{1}}_{\{M_{t}^{i}\geq1\}}\biggr)^{1/2}\biggl(\frac{1}{N}\sum_{i=1}^{N}\bigl(M_{t}^{i}\bigr)^{2}\biggr)^{1/2}\leq T^{1/2}\biggl(\frac{1}{N}\sum_{i=1}^{N}\bigl(M_{T}^{i}\bigr)^{2}\biggr)^{1/2}.\end{split}
\label{eq:particle:40}
\end{equation}
The point is now to investigate $\P\{\tau\leq T\}$. To this end,
we define the events 
\begin{equation}
\label{eq:events}
\begin{split}
&A=\biggl\{\frac{1}{N}\sum_{i=1}^{N}\bigl(M_{T}^{i}\bigr)^{2}\leq T^{-1/2}\biggr\},
\quad A'=\biggl\{ \frac1N \sum_{i=1}^{N}{\mathbf{1}}_{A^{i}}\leq T^{2}\biggr\}, \ \ \mathrm{where}
\\
&A^{i}=\biggl\{ C\int_{0}^{T}(1+\sup_{r \in [0,s]}\vert Z_{r}^{i}\vert)ds+\alpha T^{1/4}+\sup_{s \in [0,T]}\vert W_{s}^{i}\vert\geq\varepsilon_{0}/2\biggr\},
\end{split}
\end{equation}
and $C \geq 0$ (independent of $N$ and $T$) is chosen to be the constant such that
on $A$, for $t<\tau\wedge T$ and any $1\leq i\leq N$,
\[
\sup_{s \in [0,t]}\bigl(Z_{s}^{i}\bigr)_{+}\leq\bigl(Z_{0}^{i}\bigr)_{+}+C\int_{0}^{t}\bigl(1+\sup_{r \in [0,s]}\vert Z_{r}^{i}\vert\bigr)ds+\alpha T^{1/4}+\sup_{s \in [0,t]}\vert W_{s}^{i}\vert.
\]
The existence of such a constant follows from the definition
of the reformulated particle system \eqref{particle system Z} and \eqref{eq:particle:40}. 
As $(Z_{0}^i)_{+} \leq 1-\varepsilon_{0}$ by Assumption \ref{assumption 2}, we have, 
on $A$, for $t<\tau\wedge T$, 
\[
\begin{split}
&\sup_{s \in [0,t]}\bigl(Z_{s}^{i}\bigr)_{+}\geq1-\frac{\varepsilon_{0}}{2} 
  \Rightarrow C\int_{0}^{T}\bigl(1+\sup_{r \in [0,s]}\vert Z_{r}^{i}\vert\bigr)ds+\alpha T^{1/4}+\sup_{s \in [0,T]}\vert W_{s}^{i}\vert\geq\frac{\varepsilon_{0}}{2}.
\end{split}
\]
 If $\tau\leq T$, the above is true on $A$ for all $t<\tau$. We
deduce that, on $A\cap\{\tau\leq T\}$  
\[
\begin{split} & \frac{1}{N}\sum_{i=1}^{N}{\mathbf{1}}_{\{\sup_{s \in [0,\tau)}(Z_{s}^{i})_{+}\geq1-\varepsilon_{0}/2\}}\leq\frac{1}{N}\sum_{i=1}^{N}{\mathbf{1}}_{A^{i}},
\end{split}
\]
with $(A^{i})_{i=1,\dots,N}$ as in \eqref{eq:events}.
On $A\cap A'\cap\{\tau\leq T\}$, it thus holds that 
\[
\frac{1}{N}\sum_{i=1}^{N}{\mathbf{1}}_{\{\sup_{s \in [0,\tau)}(Z_{s}^{i})_{+}\geq1-\varepsilon_{0}/2\}}\leq T^{2}.
\]
 Assume that $T^{2}\leq\varepsilon_{0}/4$. Then the number of particles
such that $\sup_{s \in [0,\tau)}(Z_{s}^{i})_{+}\geq1-\varepsilon_{0}/2$
is at most $N\varepsilon_{0}/4$. The other particles cannot cross
1 at time $\tau$, since the size of the kick they receive due to
 those such that $\sup_{s \in [0,\tau)}(Z_{s}^{i})_{+}\geq1-\varepsilon_{0}/2$ is bounded by $\varepsilon_{0}/4$.
Therefore, the number that have crossed 1 up to and including $\tau$
must also be less than $NT^{2}$, i.e. 
$
(1/N) \sum_{i=1}^{N}{\mathbf{1}}_{\{M_{\tau}^{i}\geq1\}}\leq T^{2}$. 

 Since $T\leq\sqrt{\varepsilon_{0}}/2\leq1/2$, we have $T^{2} <T$. 
This yields a contradiction since,  by  definition of
$\tau$ and by right-continuity, 
$(1/N) \sum_{i=1}^N {\mathbf 1}_{\{M^i_{\tau} \geq 1\}} \geq T$. 
In other words $A\cap A'\cap\{\tau\leq T\}=\emptyset$, so
that $\{\tau\leq T\}\subset(A\cap A')^{\complement}$. Hence 
\begin{equation}
\P(\tau\leq T)\leq\P\bigl(A^{\complement}\bigr)+\P\bigl((A')^{\complement}\bigr).\label{eq:particle:42}
\end{equation}
 By Markov's inequality, 
$\P((A')^{\complement}) = \P( (1/N) \sum_{i=1}^{N}{\mathbf{1}}_{A^{i}} > T^2 ) 
 \leq (1/NT^{2})\sum_{i=1}^{N} \P(A^{i}).$ 
Thus, by \eqref{eq:events} and using the fact that $T\leq1$, 
\begin{equation}
\begin{split}\P\bigl((A')^{\complement}\bigr)
&\leq\frac{1}{NT^{2}}\sum_{i=1}^{N}
\Bigl[
\P\Bigl((\alpha+C)T^{1/4}+\sup_{s \in [0,T]}\vert W_{s}^{i}\vert\geq\frac{\varepsilon_{0}}{4}\Bigr)
+
\P\Bigl(\sup_{s \in [0,T]}\vert Z_{s}^{i}\vert\geq\frac{\varepsilon_{0}}{4CT}\Bigr) \Bigr].
\end{split}
\label{eq:particle:30}
\end{equation}
 By Lemma \ref{lem:particle:5} (with $p=3$) and Markov's inequality
again, we see that (since $T\leq1$),
\begin{equation}
\label{eq:particle:31}
\P\Bigl(
\sup_{s \in [0,T]}\vert Z_{s}^{i}\vert\geq\frac{\varepsilon_{0}}{4CT}\Bigr)
\leq4^{3}C^{3}C_{1}^{(3)}\varepsilon_{0}^{-3}T^{3}
\leq C'T^{3},
\end{equation}
for another constant $C'$ depending upon $\varepsilon_{0}$. Under the additional assumption that 
$(\alpha+C)T^{1/4}\leq\varepsilon_{0}/8$,
the first term in the right-hand
side of \eqref{eq:particle:30} can be bounded by 
\[
\begin{split}\frac{1}{NT^{2}}\sum_{i=1}^{N}\P\biggl(
(\alpha+C)T^{1/4}+\sup_{s \in [0,T]}\vert W_{s}^{i}\vert\geq\frac{\varepsilon_{0}}{4}\biggr) & 
\leq\frac{1}{NT^{2}}\sum_{i=1}^{N}\P\biggl(\sup_{s \in [0,T]}\vert W_{s}^{i}\vert\geq\frac{\varepsilon_{0}}{8}\biggr)
\\
 & \leq cT^{-2}\exp\bigl(-c^{-1}T^{-1}\bigr),
\end{split}
\]
 for some constant $c>0$, independent of $N$ and $T$ (but depending 
 upon $\varepsilon_{0}$). Here we have
used the reflection principle and an elementary bound on the Gaussian
distribution function.

In the end, by \eqref{eq:particle:30}, \eqref{eq:particle:31} and
the above inequality, we obtain 
\[
\P\bigl((A')^{\complement}\bigr)\leq C'T+cT^{-2}\exp\bigl(-c^{-1}T^{-1}\bigr)
\leq C'T,
\]
 for some \ensuremath{C'},
 independent of \ensuremath{N}
 and \ensuremath{T} and the value of which is allowed to increase from one inequality to another. In a similar way to the proof of \eqref{eq:particle:31}, we also
have that $\P(A^{\complement})\leq C'T$. Therefore, \eqref{eq:particle:42} yields 
\[
\P(\tau\leq T)\leq C'T,
\]
for $T^{2}\leq\varepsilon_{0}/4$
and $(\alpha+C)T^{1/4}\leq\varepsilon_{0}/8$. The point is that this probability is small in $T$. Finally, by
\eqref{eq:particle:40} and by Lemma \ref{lem:particle:5} again,
\[
\begin{split}\P\bigl( \bar{e}^{N}(T)\geq T^{1/4}\bigr) & \leq\P\bigl( \bar{e}^{N}(T)\geq T^{1/4},T<\tau\bigr) 
+\P\bigl( \tau\leq T\bigr)
\\
 & \leq\P\biggl( \frac{1}{N}\sum_{i=1}^{N}(M_{T}^{i})^{2}\geq T^{-1/2}\biggr) +
 \P\bigl( \tau\leq T\bigr) 
 \\
 & \leq\P( A^{\complement}) 
+\P\bigl( \tau\leq T\bigr)   \leq C'T.
\end{split}
\]

Choose now $T=T_{k}$ and $T_{k}=\lambda^{k}$ with $\lambda<1$ such
that $(\alpha+C)\lambda^{1/4}\leq\varepsilon_{0}/8$ and $\lambda^{2}\leq\varepsilon_{0}/4$.
Then by above, 
\[
\P\bigl( \bar{e}^{N}(T_{k})\geq T_{k}^{1/4}\bigr) \leq C'\lambda^{k},
\]
 so that, for any $k_{0}\geq1$, 
\begin{align*}
\P\biggl(\bigcup_{k\geq k_{0}}\bigl\{\bar{e}^{N}(T_{k})\geq T_{k}^{1/4}\bigr\}\biggr)
 & \leq\sum_{k\geq k_{0}}
 C'\lambda^{k} =:\eta(k_{0}),
\end{align*}
 where $\eta$ is finite since the sum converges, is independent of
$N$ and satisfies $\lim_{x\rightarrow+\infty}\eta(x)=0$. Observe
now that, for any $t\in(T_{k+1},T_{k}]$, $M_{t}^{i}\leq M_{T_{k}}^{i}$
and $\lambda^{-1/4}t^{1/4}\geq T_{k}^{1/4}$,
so that 
 $\bar{e}^N(t) \geq 
\lambda^{-1/4}t^{1/4}$ implies $\bar{e}^N(T_{k}) \geq T_{k}^{1/4}$
(recall that $\bar{e}^N$ is non-decreasing). 
 Therefore, 
\[
\P\bigl( \exists t\in[0,T_{k_{0}}]:\bar{e}^{N}(t)\geq\lambda^{-1/4}t^{1/4}\bigr) \leq\eta(k_{0}),
\]
 thus completing the proof. \end{proof}

The next proposition shows that there is a very small chance of observing
a macroscopic proportion of particles spiking twice or more in a small
interval and extends Proposition~\ref{prop:16:2:14:1} to intervals
of non-zero length.

\begin{proposition} \label{prop:20:02:14:1} For a given $T>0$,
consider $0\leq t<t+h\leq T$, $h\in(0, 1)$. Then, we can find $C \geq 0$,
independent of $h$, and an integer $N_{0}:=N_{0}(h)$, such that, for
$N\geq N_{0}$, 
\begin{equation*}
\P\bigl(\bar{e}^{N}(t+h)-\bar{e}^{N}(t-)>1+Ch^{1/16}\bigr)\leq Ch
\end{equation*}
 and 
\begin{equation*}
\P\biggl(\forall \lambda \leq 
\bigl( \bar{e}^{N}(t+h)-\bar{e}^{N}(t-)-Ch^{1/16} \bigr)_{+}, 
\ \frac{1}{N}\sum_{i=1}^{N}{\mathbf{1}}_{\{X_{t-}^{i}\geq1-\alpha\lambda-Ch^{1/16}\}}\geq \lambda \biggr)\geq1-Ch.
\end{equation*}
 \end{proposition}

\begin{proof} The first step of the proof is to show that the proportion
of particles that spike twice in a small interval tends to $0$ with
the length of the interval, uniformly in $N\geq1$. More precisely,
given an interval $[t,t+h]$ and $\beta\in(0,1)$, define 
\[
\tau(\beta)=\inf\Bigl\{ s\in[t,t+h]:\frac{1}{N}\sum_{i=1}^{N}{\mathbf{1}}_{\{M_{s}^{i}-M_{t-}^{i}\geq2\}}\geq\beta\Bigr\},\quad\inf\emptyset=+\infty.
\]
 Then we want to show that, for
 $\beta=h^{1/4}$, $N > h^{-1/2}$
and $p\geq1$, there exists a constant $C_{p}$ (independent of $h$), such that 
\begin{equation}
\P\bigl(\tau(\beta)\leq t+h\bigr)\leq C_{p}h^{p}.\label{eq:19:02:14:4b}
\end{equation}

In order to do this, we will have to enter into the spike cascade,
which will require the {\it cascade time axis} defined on page
\pageref{page cascade}. Indeed, we will say that particle $i\in\{1,\dots,N\}$
spikes \textit{twice before} $j$ if \ensuremath{\inf\{s\geq t,M_{s}^{i}\geq M_{t-}^{i}+2\}<\inf\{s\geq t,M_{s}^{j}\geq M_{t-}^{j}+2\}}
, or \ensuremath{\inf\{s\geq t,M_{s}^{i}\geq M_{t-}^{i}+2\}=\inf\{s\geq t,M_{s}^{j}\geq M_{t-}^{j}+2\}=:\rho}
 and \ensuremath{X_{\rho-}^{i}>X_{\rho-}^{j}}. This precisely means that particle
$i$ will spike twice before $j$ either in (usual) time, or before $j$
along the {\it cascade time axis}.

Define the set $I=\{i\in\{1,\dots,N\}:M_{t+h}^{i}-M_{t-}^{i}\geq2\}$
of particles that have spiked at least twice in the interval $[t-,t+h]$.
We prove the following claim:

\vspace{0.3cm}
 \textit{Claim:} Suppose $\tau(\beta)\leq t+h$. Then there exists
a set $I(\beta)\subset\{1,\dots,N\}$, such that $\beta N\leq|I(\beta)|\leq\beta N+1$
and, for all $i\in I(\beta)$, it holds that 
\begin{equation}
1\leq\alpha+\beta^{1/2}\biggl(\frac{1}{N}\sum_{j=1}^{N}
\bigl(M_{t+h}^{j}\bigr)^{2}\biggr)^{1/2}
+Ch\bigl(1+\sup_{s \in [0,t+h]}
\vert Z_{s}^{i}\vert\bigr)
+\sup_{s \in [t,t+h]}\bigl\vert W_{s}^{i}-W_{t}^{i}\bigr\vert.
\label{eq:16:2:14:1}
\end{equation}

 To prove 
\eqref{eq:16:2:14:1}, suppose $\tau(\beta)\leq t+h$. {By right-continuity
 of each $(M^i_{s})_{s \geq 0}$, 
$\vert I \vert \geq N \beta$}. For $i_{0}\in I$, let $I^{(i_{0})}$ be the set of particles
that have spiked twice before $i_{0}$ in the above sense. Whenever
$\vert I^{(i_{0})}\vert<\beta N$, the sum of the kicks received by particle
$i_{0}$ due to the effect of the particles in $I^{(i_{0})}$ spiking before
it (again in the previous sense) is bounded by 
\begin{equation}
\begin{split}\alpha\Bigl[1+\frac{1}{N}\sum_{i\in I^{(i_{0})}}\bigl(M_{t+h}^{i}-M_{t-}^{i}\bigr)\Bigr] & \leq\alpha\Bigl[1+\frac{1}{N}\sum_{i\in I^{(i_{0})}}M_{t+h}^{i}\Bigr]
 \leq\alpha+\beta^{1/2}\left(\frac{1}{N}\sum_{i=1}^{N}\bigl(M_{t+h}^{i}\bigr)^{2}\right)^{1/2}.
\end{split}
\label{eq:16:2:14:2}
\end{equation}
 The first $\alpha$ stands for the kick generated by particles that
have spiked once only. The other part corresponds to the particles
that have spiked twice or more. At the time when the particle $i_{0}$
spikes for the second time, \ensuremath{X^{i_{0}}}
 has to cross $1$, or equivalently, the $Z$-particle $i_{0}$ {crosses}
a new integer. Since it is its second spike in the interval $[t,t+h]$,
the $Z$-particle $i_{0}$ has run more than $1$ since $t-$, i.e.
{$1 \leq \sup_{t \leq s \leq t+h} \vert Z_{s} - Z_{t-}\vert$}, so that 
\[
1 \leq\int_{t}^{t+h}\vert b(X_{s}^{i_{0}})\vert ds
+\alpha+\beta^{1/2}
\biggl(\frac{1}{N}\sum_{j=1}^{N}\bigl(M_{t+h}^{j}\bigr)^{2}\biggr)^{1/2}
+\sup_{s \in [t,t+h]}\bigl\vert W_{s}^{i_{0}}-W_{t}^{i_{0}}\bigr\vert.
\]
Using the bound for the growth of $b$, 
\[
\begin{split}1 & \leq \alpha + \beta^{1/2}\biggl(\frac{1}{N}\sum_{j=1}^{N}\bigl(M_{t+h}^{j}\bigr)^{2}\biggr)^{1/2}+Ch\bigl(1+\sup_{s \in [0,t+h]}\vert Z_{s}^{i_{0}}\vert\bigr)+\sup_{s \in [t,t+h]}\bigl\vert W_{s}^{i_{0}}-W_{t}^{i_{0}}\bigr\vert.\end{split}
\]
Iterating the argument up until the number of particles that have
spiked more than twice is greater than $N\beta$, we can find an index \(i_1\) such that
$I^{(i_1)}\subset\{1,\dots,N\}$ and $\beta N\leq\vert I^{(i_1)}\vert\leq\beta N+1$.  This proves the claim.

To proceed, we can take the mean in \eqref{eq:16:2:14:1} over the particles
in $I(\beta)$. Using the bound $\vert I(\beta)\vert \leq 
N \beta +1$ and the Cauchy-Schwarz inequality, we see that
\[
\begin{split}
\beta &\leq \alpha \beta +   \beta^{1/2}\bigl(\beta+\frac{1}{N}\bigr)\biggl(\frac{1}{N}\sum_{j=1}^{N}\bigl(M_{t+h}^{j}\bigr)^{2}\biggr)^{1/2}
\\
 &\hspace{5pt} +\bigl(\beta+\frac{1}{N}\bigr)^{1/2}
\biggl[
Ch \biggl(1+\frac{1}{N}\sum_{j=1}^{N}\sup_{s \in [0,t+h]}\vert Z_{s}^{j}\vert^{2}\biggr)^{1/2}
+
 \biggl(\frac{1}{N}\sum_{j=1}^{N}\sup_{s \in [t,t+h]}\vert W_{s}^{j}-W_{t}^{j}\vert^{2}\biggr)^{1/2}
 \biggr].
\end{split}
\]
Since $\beta = h^{1/4} \geq1/\sqrt{N}$, we have  $1/(\beta N)\leq1/\sqrt{N}$.
Dividing both sides of the above inequality by $\beta$, we deduce that  $\tau(\beta) \leq t+h$ implies 
\[
\begin{split}1 &\leq\alpha+2\beta^{1/2}\biggl(\frac{1}{N}\sum_{j=1}^{N}\bigl(M_{t+h}^{j}\bigr)^{2}\biggr)^{1/2}\\
 & \hspace{5pt}+Ch\beta^{-1/2}\biggl(1+\frac{1}{N}\sum_{j=1}^{N}\sup_{s \in [0,t+h]}\vert Z_{s}^{j}\vert^{2}\biggr)^{1/2}+2\beta^{-1/2}\biggl(\frac{1}{N}\sum_{j=1}^{N}\sup_{s \in [t,t+h]}\vert W_{s}^{j}-W_{t}^{j}\vert^{2}\biggr)^{1/2}.
\end{split}
\]
We can now apply Markov's inequality with any exponent $p\geq1$.
By Lemma \ref{lem:particle:5}, we get that there exists a constant
$C_{p}$ such that 
\begin{equation}
\label{eq:19:02:14:4}
\P\bigl(\tau(\beta)\leq t+h\bigr)\leq C_{p}\bigl(\beta^{p/2}+h^{p}\beta^{-p/2}+h^{p/2}\beta^{-p/2}\bigr)
= C_{p} \bigl( h^{p/8} + h^{7p/8} + h^{3p/8} \bigr). 
\end{equation}
 On the event $\{\tau(\beta)>t+h\}\cap\{N^{-1}\sum_{i=1}^{N}(M_{t+h}^{i})^{2}\leq h^{-1/8}\}$,
we have, as in \eqref{eq:16:2:14:2}, 
\begin{equation}
\bar{e}^{N}(t+h)-\bar{e}^{N}(t-)\leq1+\beta^{1/2}\biggl(\frac{1}{N}\sum_{i=1}^{N}\bigl(M_{t+h}^{i}\bigr)^{2}\biggr)^{1/2}\leq1+\beta^{1/2}h^{-1/16}.
\label{eq:19:02:14:5}
\end{equation}
Since $\beta=h^{1/4}$, we deduce from Lemma \ref{lem:particle:5}
and \eqref{eq:19:02:14:4} that the first bound in the statement holds
with $N$ large enough.

Now, we can focus on the particles that spike no more than once between
$t$ and $t+h$. The set of such particles coincides with $I^{\complement}$.
In $I^{\complement}$, there are two kind of particles: The set $I^{\complement,0}$
denotes the set of particles that do not spike and the set $I^{\complement,1}$
the set of particles that spike once. In order to characterize the
sets $I^{\complement,0}$ and $I^{\complement,1}$, we can make use
of the ordering of the spikes again, as defined on
page \pageref{page cascade}. 

A particle $i_{1}\in I^{\complement,1}$ is to spike at some time $s \in [t,t+h]$, 
if, at some moment along the \emph{time cascade axis} at $s$, 
the kick it receives from the particles that spike before in the cascade is larger than  $1-X_{s-}^{i_{1}}$. 
Now, as $i_{1}$ doesn't spike between $t$ and $s-$, $1-X_{s-}^{i_{1}}$ is equal to 
\begin{equation*}
1- X_{s-}^{i_{1}} = 1-X_{t-}^{i_{1}} - \int_{t}^s b(X_{r}^{i_{1}}) dr - (W_{s}^{i_{1}}- W_{t}^{i_{1}})
- \alpha \bigl( \bar{e}^N(s-)- \bar{e}^N(t-) \bigr). 
\end{equation*}
We observe that $\alpha ( \bar{e}^N(s-)- \bar{e}^N(t-) )$ represents the kick $i_{1}$ receives from the other neurons between 
$t-$ and $s-$. Therefore, $i_{1}\in I^{\complement,1}$ is to spike at some time $s \in [t,t+h]$, if
the kick it receives between $t-$ and $s-$ plus the kick it receives along the \textit{time axis cascade} at $s$
before it spikes is greater than 
\begin{equation*}
1-X_{t-}^{i_{1}} - \int_{t}^s b(X_{r}^{i_{1}}) dr - (W_{s}^{i_{1}}- W_{t}^{i_{1}}). 
\end{equation*}
The sum of the two kicks is called the \emph{kick received by $i_{1}$ before it spikes}. 
Following \eqref{eq:16:2:14:2}, it can be bounded by 
\[
\alpha\frac{k}{N}+\biggl(\frac{\vert I\vert}{N}\times\frac{1}{N}\sum_{j=1}^{N}\bigl(M_{t+h}^{j}\bigr)^{2}\biggr)^{1/2},
\]
\color{black}
 where $k$ stands for the number of particles in $I^{\complement}$
that spike once before $i_{1}$. Therefore, for $i_{1}$ to
be in $I^{\complement,1}$, it must hold that
\begin{equation}
\label{eq:nixon}
\begin{split} & X_{t-}^{i_{1}}+\alpha\frac{k}{N}+\biggl(\frac{\vert I\vert}{N^2}\sum_{j=1}^{N}\bigl(M_{t+h}^{j}\bigr)^{2}\biggr)^{1/2}
+Ch\bigl(1+\sup_{s \in [0,t+h]}\vert Z_{s}^{i_{1}}\vert\bigr)+\sup_{s \in [t,t+h]}\vert W_{s}^{i_{1}}-W_{t}^{i_{1}}\vert\geq1,
\end{split}
\end{equation}
the two last terms in the left-hand side standing for bounds on the drift and
Brownian parts in the dynamics of $X^{i_{1}}$. Obviously, the number
of particles for which the above inequality holds must be larger than
$k+1$ (it must be true for the $k$ particles that spiked before
$i_{1}$ and for $i_{1}$ as well). On the model of Proposition \ref{prop:16:2:14:1},
this must be true for any $k<\vert I^{\complement,1}\vert$.

Now, following \eqref{eq:19:02:14:5} for estimating the overall kick on $[t,t+h]$, we deduce 
\[
\frac{\vert I^{\complement,1}\vert}{N}\leq\bar{e}^{N}(t+h)-\bar{e}^{N}(t-)\leq\frac{\vert I^{\complement,1}\vert}{N}+\biggl(\frac{\vert I\vert}{N^2}\sum_{j=1}^{N}\bigl(M_{t+h}^{j}\bigr)^{2}\biggr)^{1/2}.
\]
 Therefore, for an integer $0 \leq k<N(\bar{e}^{N}(t+h)-\bar{e}^{N}(t-))-[\vert I\vert\sum_{j=1}^{N}(M_{t+h}^{j})^{2}]^{1/2}$,
 we have $k < \vert I^{\complement,1} \vert$. By \eqref{eq:nixon}, we can deduce that
\begin{equation}
\label{eq:nixon2}
\begin{split} 
\frac{1}{N}\sum_{i=1}^{N}{\mathbf{1}}_{B^{i,1}(k)}\geq\frac{k}{N},
\quad \textrm{with} \
 & B^{i,1}(k)=\biggl\{ X_{t-}^{i}+\alpha\frac{k}{N}+\biggl(\frac{\vert I\vert}{N^2}\sum_{j=1}^{N}\bigl(M_{t+h}^{j}\bigr)^{2}\biggr)^{1/2}
 \\
 & \hspace{5pt}+Ch\bigl(1+\sup_{s \in [0,t+h]}\vert Z_{s}^{i}\vert\bigr)+\sup_{s \in [t,t+h]}\vert W_{s}^{i}-W_{t}^{i}\vert\geq1\biggr\}.
\end{split}
\end{equation}
Define now the events 
\[
\begin{split}
&B^{0}=\{\tau(\beta)>t+h\}\cap\biggl\{\frac{1}{N}\sum_{j=1}^{N}\bigl(M_{t+h}^{j}\bigr)^{2}\leq h^{-1/8}\biggr\}
\\
&B^{2}=\biggl\{\frac{1}{N}\sum_{i=1}^{N}{\mathbf{1}}_{B^{i,2}}\leq h\biggr\}, \quad  B^{i,2}=\Bigl\{ Ch\bigl(1+\sup_{s \in [0,t+h]}\vert Z_{s}^{i}\vert\bigr)+\sup_{s \in [t,t+h]}\vert W_{s}^{i}-W_{t}^{i}\vert\geq h^{1/16}\Bigr\}.
\end{split}
\]
On $B^0$, the term 
$((\vert I\vert/N^2)\sum_{j=1}^{N}(M_{t+h}^{j})^{2})^{1/2}$ is less than 
$h^{1/16}$.
Therefore, on $B^0$, 
for any $\lambda \leq\bar{e}^{N}(t+h)-\bar{e}^{N}(t-)- 3h^{1/16}$,
we have 
\[ 
\lfloor \lambda N + 2N h^{1/16}\rfloor  \leq
N \bigl(\bar{e}^{N}(t+h)-\bar{e}^{N}(t-)\bigr)-\biggl[\vert I\vert\sum_{j=1}^{N}(M_{t+h}^{j})^{2}\biggr]^{1/2}.\] 
For such a $\lambda$, we choose $k=\lfloor \lambda N 
+ 2 N h^{1/16}
\rfloor$, so that 
$k$ satisfies the required condition to apply \eqref{eq:nixon2}.
On $B^0 \cap B^{i,1}(k) \cap (B^{i,2})^{\complement}$, 
$X_{t-}^i + \alpha k/N \geq 1- 2 h^{1/16}$, so that
$$X_{t-}^i + \alpha \lambda \geq 
X_{t-}^i + \alpha \frac{k}N - 2 h^{1/16} \geq
1- 4 h^{1/16}.$$
Therefore, on $B^{0} \cap B^{2}$, 
\[
\frac{1}{N}\sum_{i=1}^{N}{\mathbf{1}}_{\{X_{t-}^{i}\geq1-\alpha\lambda-4h^{1/16}\}}
\geq\frac{1}{N}\sum_{i=1}^{N}{\mathbf{1}}_{B^{i,1}(k)}-\frac{1}{N}\sum_{i=1}^{N}{\mathbf{1}}_{B^{i,2}}\geq\frac{k}{N}- h \geq \lambda,
\]
 the last inequality following from the fact that, since $N > h^{-1/2}$ and $h \leq 1$,
$k/N-h \geq \lambda + 2 h^{1/16} - 1/N - h \geq \lambda$.  
 To complete the proof, notice from \eqref{eq:19:02:14:4} and Lemma 
 \ref{lem:particle:5}
that
\[
\begin{split}\P\Bigl(\bigl(B^{0}\cap B^{2}\bigr)^{\complement}\Bigr) & \leq C'h+\P\Bigl(\bigl(B^{2}\bigr)^{\complement}\Bigr)\leq C'h+h^{-1}\frac{1}{N}\sum_{i=1}^{N}\P\bigl(B^{i,2}\bigr)\leq C'h,\end{split}
\]
 the value of $C'$ being allowed to increase from one inequality
to another. 
 \end{proof}


\subsection{Tightness properties and convergent subsequences}
\label{sec:tightness}
This second section is now devoted to the proof of the tightness property of the family of measures $(\Pi_N)_{N\geq1}$ defined in Theorem~\ref{weak existence thm}.  It is for this result that the M1 Skorohod topology plays a key role.
Recall that  $\tilde{Z}^{i, N}\in\hat{\mathcal D}([0, T+1], \R)$ satisfies
\begin{equation}
\label{eq:Z-tilde3}
\tilde{Z}^{i, N}_t :=
\begin{cases}
Z^{i, N}_t, &\mbox{if } t \leq  T\\
W_{t}^{i} - W_{T}^{i} + Z^{i, N}_{T}, &\mbox{if } t\in (T, T+1]
\end{cases} 
\end{equation}
for $i\in\{i, \dots, N\}$, where $(Z^{i, N}_t)_{t\in[0, T]}$ is the (physical) solution to the particle system \eqref{particle system Z}, 
\[
\bar{\mu}_N := \frac{1}{N}\sum_{i=1}^N{\textrm{Dirac}(\tilde{Z}^{i, N})},
\]
so that $\bar{\mu}_N$ is a random variable taking values in $\mathcal{P}(\hat{\mathcal D}([0,T+1], \R))$ and $\Pi_N = \mathrm{Law}(\bar{\mu}_N)$.

\begin{lem}
\label{lem:tightness1}
For any $T>0$, the family of laws of $\tilde{Z}^{1, N}$, $N\geq 1$, is tight in $\mathcal{P}(\hat{\mathcal D}([0, T+1], \R))$ endowed with the weak topology inherited from the M1 topology on $\hat{\mathcal D}([0, T+1], \R)$.
\end{lem}

\begin{proof}
By definition of tightness, we must show that for any $\varepsilon \geq0$, there exists $K\subset\hat{\mathcal D}([0, T+1], \R)$ compact for the M1 topology, such
that
$\inf_{N\geq1}\mathbb{P}(\tilde{Z}^{1,N} \in K) \geq 1-\varepsilon$. 
By Theorem \ref{thm:M1:3}, $K$ is compact for M1 if and only if 
$
\lim_{\delta\to0}\sup_{f\in K} u_{T+1}(f, \delta) =0$,
where
\begin{equation}
\label{mod of con}
u_{T+1}(f, \delta) := \Big(\sup_{t\in[0,T+1]}w_{T+1}(f, t, \delta)\Big)\vee v_{T+1}(f, 0, \delta) \vee v_{T+1}(f, T+1, \delta)
\end{equation}
is the modulus of continuity appearing in that result, and the functions $w_{T+1}$ and $v_{T+1}$ are given by \eqref{M1 wT} and in Theorem \ref{thm:M1:3} respectively.  With this in mind, for any $R\geq1$, define the set
\[
K_R := \bigl\{f\in\hat{\mathcal D}([0, T+1], \R): u_{T+1}(f, \delta)\leq R\delta^\frac{1}{4},\ \delta\in(0, 1/R)\bigr\}.
\]
It is thus clear that $K_R$ is compact for every $R\geq1$.  It therefore suffices to show that
\begin{equation}
\label{M1 tightness to show}
\lim_{R\to\infty}\inf_{N\geq1}\mathbb{P}\Big( \forall  \delta\in(0, 1/R),\ u_{T+1}(\tilde{Z}^{1,N}, \delta)\leq R\delta^\frac{1}{4} \Big) =1.
\end{equation}
Since $u_{T+1}$ is a maximum of three terms, the above certainly holds if it also holds when $u_{T+1}$ is replaced by each of the three terms appearing in the maximum in \eqref{mod of con} individually.  This is what we aim to show now, starting with the first term.

To this end define the new process $(\tilde{U}^N_t)_{t\in[0, T+1]}$ as
\[
\tilde{U}^N_t = \tilde{Z}^{1, N}_t - \frac{1}{N}\sum_{i=1}^N M^{i, N}_{t\wedge T},\qquad t\in [0, T+1].
\] 
Then $\tilde{U}^N$ is the continuous part of $\tilde{Z}^{1, N}$.  We use the easily verified fact that
\[
w_{T+1}(f+g, t, \delta) \leq v_{T+1}(g, t, \delta)\quad t \in[0,T], \delta>0,
\]
whenever $f, g\in \hat{\mathcal D}([0, T+1], \R)$, and $f$ is monotone. 
Thus, since $\tilde{Z}^{1, N} - \tilde{U}^N$ is non-decreasing,
\[
\sup_{t\in[0,T+1]}w_{T+1}(\tilde{Z}^{1,N}, t, \delta) \leq \sup_{t\in[0,T+1]}v_{T+1}(\tilde{U}^N, t, \delta), \quad \delta>0,
\]
almost surely. Hence, in order to show that \eqref{M1 tightness to show} holds in the first case, 
it is sufficient to prove that 
$\inf_{N\geq1}\mathbb{P}( \forall  \delta\in(0, 1/R),\ \sup_{t\in[0,T+1]}v_{T+1}(\tilde{U}^N, t, \delta)\leq R\delta^{1/4} )$ converges to $1$ as $R\to\infty$. Since
$(\tilde{U}^N_t)_{t\in[0, T+1]}$ is a continuous process driven by a Lipschitz drift and 
a Brownian motion, this directly  follows from Lemma \ref{lem:particle:5}.  

To handle the second case, when $u_{T+1}(f, \delta) = v_{T+1}(f, 0, \delta)$, by Lemma \ref{cor:particle:1}, we know that, for any $\eta >0$ there exists {$\lambda(\eta)>0$} independent of $N$ (depending only on the $\varepsilon_0$ appearing in Assumption \ref{assumption 2}) such that
\begin{equation*}
\P \bigl( \forall t \in [0,\lambda(\eta)], \quad \bar{e}^N(t) \geq \bigl(\lambda(\eta)\bigr)^{-1}  t^{1/4} \bigr) \leq \eta.
\end{equation*}
In particular, by taking $R=\bigl(\lambda(\eta)\bigr)^{-1}$, it follows that
\begin{equation*}
\lim_{R \rightarrow + \infty} \inf_{N \geq 1}
\P \bigg( \forall \delta \in(0,1/R),\  v_{T+1} \biggl( \frac{1}{N} \sum_{i = 1}^N M^{i,N},0,\delta \biggr) 
\leq R  \delta^{1/4} \biggr) = 1,
\end{equation*}
where we have also used the definition of $v_{T+1}$, given in Theorem \ref{thm:M1:3}.
As the continuous part ($\tilde{U}^N$) of the dynamics of $\tilde{Z}^{1,N}$ can be handled in the standard way, we deduce that 
\begin{equation*}
\lim_{R \rightarrow + \infty }
\inf_{N \geq 1}
\P \bigg( \forall \delta \in (0, 1/R), \  v_{T+1} \bigl( \tilde{Z}^{1,N},0,\delta \bigr) 
\leq R  \delta^{1/4} \biggr) = 1.  
\end{equation*}

For the final term in the maximum in \eqref{mod of con},  by definition, $\tilde{Z}^{1,N}$ behaves as $W^1$ in  neighborhoods of $T+1$, so that (again in the standard way)
\begin{equation*}
\lim_{R \rightarrow + \infty }
\inf_{N \geq 1} {\mathbb P} \biggl(
\forall \delta \in (0,1/R), \ 
v_{T+1} \bigl(\tilde{Z}^{1,N},T+1, \delta \bigr) \leq R \delta^{1/4} \biggr) =1.  
\end{equation*}
This completes the proof.
\end{proof}

By \cite[Proposition 2.2]{sznitman}, we deduce:
\begin{lem}
\label{lem:tightness2}
For any $T>0$ the family $({\Pi}_N)_{N\geq 1} \subset  \mathcal{P}
(\mathcal{P}(\hat{\mathcal D}([0, T+1], \R)))$ is tight, where $\mathcal{P}(\hat{\mathcal D}([0, T+1], \R))$ is endowed with the weak topology deriving from the M1 topology on $\hat{\mathcal D}([0, T+1], \R)$.
\end{lem}


\subsection{Proof of Theorem \ref{weak existence thm}}
We now give the proof of Theorem \ref{weak existence thm}.  As we will see, the proof will rely on some key convergence results that will be proved afterwards.
Throughout the proof, as in the statement of the result, $(z_t)_{t\in[0, T+1]}$ will denote the canonical process on $\hat{\mathcal D}([0, T+1], \R)$ and $m_t = \lfloor(\sup_{s \in [0,t]} z_s)_+\rfloor$.

The first part of the theorem is contained in Lemma \ref{lem:tightness2}, namely that $({\Pi}_N)_{N\geq 1} $ is tight. We can therefore extract a convergent subsequence {as $N$ tends to $+ \infty$}, which we denote in the same way.  We set $\Pi_\infty$ to be the limit point of such a sequence.

Consider the function
\begin{equation*}
[0,T+1] \ni t \mapsto \int \mu \bigl\{ z_{t-} = z_{t} \bigr\} d\Pi_{\infty}(\mu).
\end{equation*}
Since the application $A \in {\mathcal B}(\hat{\mathcal D}([0,T+1],\R)) \mapsto \int \mu(A) d\Pi_{\infty}(\mu)$ 
defines a probability measure on $\hat{\mathcal D}([0,T+1],\R)$, the function matches $1$ for any $t$ in $[0,T+1]$ but in some countable subset $J\subset [0, T+1]$ (see Lemma 7.7, p. 131, chap. 3  in Ethier and Kurtz \cite{ethier:kurtz}). Therefore, for $t \not \in J$, for 
a.e. $\mu$ under the probability $\Pi_{\infty}$, $\mu \{z_{t-}=z_{t}\}=1$. 
Similarly, the function 
\begin{equation*}
[0,T+1] \ni t \mapsto \int \langle \mu,m_{t} \rangle d\Pi_{\infty}(\mu)
\end{equation*}
is non-decreasing and has at most a countable number of jumps. Up to a modification of $J$, we can assume that the jumps are all included in $J$. Then, for any $t \not \in J$, 
\begin{equation*}
\int \langle \mu,m_{t} \rangle d\Pi_{\infty}(\mu) = \int \langle \mu,m_{t-} \rangle d\Pi_{\infty}(\mu).
\end{equation*}
Therefore, for $t \not \in J$, for a.e. $\mu$ under the probability measure 
$\Pi_{\infty}$, $\langle \mu , m_{t-} \rangle =
\langle \mu,m_{t}\rangle$. 

For $p \geq 1$, $S_{1},\dots,S_{p} \not \in J$, $0 =S_{0} \leq S_{1} < \dots <S_{p} < T$
and $f_{0},\dots,f_{p}$ bounded and uniformly continuous functions from $\R$ into itself, put 
\begin{equation}
\label{eq:20:02:14:5}
F(z) = \prod_{i=0}^p f_{i}\bigl(z_{S_{i}}\bigr), \quad z \in \hat{\mathcal D}([0,T+1],\R). 
\end{equation}
For another bounded and uniformly continuous function $G$ from $\R$ into itself, consider 
\begin{equation*}
Q_{N} : = {\mathbb E} \biggl[ G\biggl( \biggl\langle \bar{\mu}_N, F 
\biggl(z - z_{0} - \int_{0}^{\cdot} b(z_{s}-m_{s}) ds -  \alpha
\langle \bar{\mu}_N,m_{\cdot} \rangle
 \biggr) \biggr\rangle \biggr) \biggr].
 \end{equation*}
The point is that we may write
\[
\tilde{Z}^{i, N}_t = Z^{i, N}_0 +\int_0^{t\wedge T} b(\tilde{Z}^{i, N}_s - M^{i, N}_s)ds + \alpha\langle\bar{\mu}_N, m_{t\wedge(T-)}\rangle + W^i_t
\] 
for all $t\in[0, T+1]$ and $i\in\{1, \dots, N\}$.  It then follows that
$Q_{N} = 
{\mathbb E} [ G ( N^{-1} \sum_{i=1}^N F
(W^i ))]$.
 By the law of large numbers, we deduce that
\begin{equation}
\label{eq:convergence:11}
\lim_{N \rightarrow + \infty}Q_{N}
 = G \Bigl( {\mathbb E} \bigl( F(W) \bigr) \Bigr).
 \end{equation}

The key result in order to proceed is contained in Lemma \ref{lem:continuity} below, where it is shown that under $\Pi_{\infty}$, the functional 
\begin{equation}
\label{eq:functional}
\begin{split}
&\mu \in {\mathcal P}\bigl(\hat{\mathcal D}([0,T+1],\R) \bigr)
 \mapsto  \biggl\langle {\mu}, F \biggl(z - z_{0} - \int_{0}^{\cdot} 
 b(z_{s}-m_{s}) ds - \alpha \langle \mu, m_{\cdot} 
 \rangle
  \biggr) \biggr\rangle,
\end{split}
\end{equation}
is a.e. continuous. Indeed, with this in hand, by the continuous mapping theorem, {we have}
\begin{equation*}
\lim_{N \rightarrow + \infty} Q_{N}= 
\int  G\biggl( \biggl\langle {\mu}, F \biggl(z - z_{0} -  \int_{0}^{\cdot} 
b(z_{s}-m_{s}) ds - \alpha \langle \mu, m_{\cdot}
\rangle
 \biggr) \biggr\rangle \biggr) d \Pi_{\infty}(\mu),
 \end{equation*}
 so that, from \eqref{eq:convergence:11},
 \begin{equation*}
\int  G\biggl( \biggl\langle {\mu}, F \biggl(z - z_{0} - \int_{0}^{\cdot} 
b(z_{s}-m_{s}) ds - \alpha \langle \mu, m_{\cdot}
\rangle
 \biggr) \biggr\rangle \biggr) d \Pi_{\infty}(\mu)
 = G \Bigl( {\mathbb E} \bigl( F(W) \bigr) \Bigr).
 \end{equation*}
Applying the above equality with $
\tilde{G}( \cdot) = [ G (\cdot) - G( {\mathbb E} ( F(W) )) ]^2$
instead of $G$ itself, we deduce that, $\Pi_{\infty}$ a.s., 
\begin{equation*}
G\biggl( \biggl\langle {\mu}, F \biggl(z - z_{0} -  \int_{0}^{\cdot} b(z_{s}-m_{s}) ds - \alpha 
\langle \mu, 
m_{\cdot} \rangle
 \biggr) \biggr\rangle \biggr) = G \Bigl( {\mathbb E} \bigl( F(W) \bigr) \Bigr),
\end{equation*}
so that, for a.a. probability measures $\mu$ under $\Pi_{\infty}$, under $\mu$, the process 
\begin{equation*}
\biggl(\Upsilon_{t} = z_{t} -  z_{0} - \int_{0}^{t} b(z_{s}-m_{s}) ds - \alpha \langle \mu,m \rangle_{t}
 \biggr)_{t \in [0,T+1]}
 \end{equation*}
 has the same finite-dimensional distributions as a Brownian motion at any points $0 \leq S_{1} < S_{2} < \dots < S_{p} <T$ 
 which are not in $J$. Since $(\Upsilon_{t})_{t \in [0,T)}$ has right-continuous paths under $\mu$ (and $[0,T) \cap 
 J^{\complement}$ is dense in $[0,T)$), it has the same
 finite-dimensional distributions as a Brownian motion. Moreover, since 
{the Borel $\sigma$-field generated by M1
 is also generated by the evaluation mappings}, we deduce that 
the distribution of $(\hat{\Upsilon}_{t})_{t \in [0,T]}$ 
on $\hat{\mathcal D}([0,T],\R)$  
is the Wiener distribution, where $\hat{\Upsilon}_{t}= \Upsilon_{t}$
if $t < T$ and $\hat{\Upsilon}_{T} = \Upsilon_{T-}$.  
This says that, for $\Pi_{\infty}$ a.e. $\mu$, the 
canonical process solves the reformulated equation \eqref{nonlinear eq Z} up until time $T$,
with $z_{0}$ as initial condition, which proves (2) in Theorem 
\ref{weak existence thm}. 

We now check that the law of $z_{0}$ under $\mu$ is the distribution of 
$X_{0}$ under $\P$. To this end, let $\pi_{0} : \hat{\mathcal D}([0,T+1],\R) \to \R$ be given by $\pi_0(z) = z_0$, and $\pi_{0} \sharp \mu$ be the push-forward measure of $\mu$ by $\pi_{0}$.  Then the mapping 
${\mathcal P}(\hat{\mathcal D}([0,T+1],\R)) \ni \mu \mapsto 
 \pi_{0} \sharp \mu$ is continuous (see Theorem \ref{thm:M1:1}).
Using the Skorohod representation theorem,
this allows the joint application of Theorem
\ref{thm:M1:2} at $t=0$. By the law of large numbers, 
$\pi_{0} \sharp \bar{\mu}_N$ converges towards the law of $X_{0}$. Therefore, for a.e. $\mu$ under $\Pi_{\infty}$, 
$\pi_{0} \sharp \mu$  matches the law of $X_{0}$, which proves (1) in Theorem 
\ref{weak existence thm}.

We finally prove that, for almost all $\mu$ under $\Pi_\infty$, the canonical process under $\mu$ satisfies the required conditions for
 defining a physical condition up until $T$.  This requires showing the conditions (2) and (4) of Definition \ref{def:nonlineareq} are satisfied for the canonical process under $\mu$. We make use of Proposition \ref{prop:20:02:14:1}, which says that, for $0 \leq t<t+h < T$ and for $N$ large enough,   
\begin{equation*}
\begin{split}
&\P \bigl( \bar{e}^N(t+h) - \bar{e}^N(t-)  > 1 + C h^{1/16} \bigr) \leq C h, 
\\
&\P \Bigl( \forall \lambda \leq  \bar{e}^N(t+h) - \bar{e}^N(t-) -C  h^{1/16},
\bar{\mu}_N \bigl( z_{t-} - m_{t-} 
\geq 1 - \alpha \lambda - C h^{1/16}\bigr)  \geq \lambda \Bigr) \geq 1- C h, 
\end{split}
\end{equation*}
where $ \bar{e}^N(t+h) - \bar{e}^N(t-) =  \langle \bar{\mu}_N,m_{t+h} - m_{t-} \rangle$.  

By the Skorohod representation theorem, we can assume that 
$\bar{\mu}_N$ converges almost surely to $\mu$. 
Choosing $t$ and $t+h$ in $J^\complement$,   
we make use of Lemma \ref{lem:2} below. It says that 
$\lim_{N \rightarrow + \infty}
\langle \bar{\mu}_N,m_{t+h} - m_{t-} \rangle
= \langle \mu,m_{t+h} - m_{t} \rangle$ (a.e. under $\Pi_{\infty}$).
Moreover, as $t \not \in J$, Theorem
\ref{thm:M1:2} says that the law of $z_{t-}-m_{t-}$ under $\bar{\mu}_N$
converges to the law of $z_{t-} - m_{t-}$ under $\mu$ (for almost all $\mu$ under $\Pi_{\infty}$). 
Therefore, following the proof of the Portmanteau theorem
and modifying the constant $C$ if necessary, we get
\begin{equation}
\label{eq:18:02:14:1}
\begin{split}
&\Pi_{\infty} \bigl(  \langle \mu,m_{t+h} - m_{t} \rangle > 1 + Ch^{1/16} \bigr) \leq C h,
\\
&\Pi_{\infty} \biggl( \forall \lambda \leq \langle \mu,m_{t+h} - m_{t} \rangle - Ch^{1/16}, \  
\mu \bigl( z_{t-} - m_{t-} \geq 1- \alpha \lambda -C h^{1/16} \bigr)
\geq \lambda \biggr) \geq 1 - Ch.
\end{split}
\end{equation}
The above inequalities are true for any $t,t+h$ that are not in $J$. Assume now that 
$t$ is some point in $[0,T) \cap J$. Then, we can find 
sequences $(t_{p})_{p \geq 1}$ and $(h_{p})_{p \geq 1}$ such that 
$0 \leq t_{p} < t < t_{p} + h_{p} <T$, $t_{p}$ and $t_{p}+h_{p}$ $\not\in J$, and $t_p\uparrow t$, $h_p\downarrow 0$ . 
Then, applying \eqref{eq:18:02:14:1} to any $(t_{p},t_{p}+h_{p})$
and letting $p$ tend to $+\infty$, we deduce that 
\begin{equation}
\label{eq:18:02:14:1b}
\begin{split}
&\Pi_{\infty} \bigl( \langle \mu,m_{t}-m_{t-}\rangle >1 \bigr) =0,
\\
&\Pi_{\infty} \biggl( \forall \lambda < \langle \mu,m_{t}-m_{t-}\rangle, \quad 
\mu \bigl( z_{t-}  - m_{t-}\geq 1- \alpha \lambda \bigr)
\geq \lambda \biggr) =1.
\end{split}
\end{equation}
The first equality shows that under $\mu$ the canonical process satisfies condition (2) of Definition \ref{def:nonlineareq} (a.e. under $\Pi_\infty$). Moreover, since $J$ is countable, we deduce that 
\begin{equation*}
\Pi_{\infty} \biggl(\forall t \in [0,T) \cap J, \quad \forall \lambda < \langle \mu,m_{t}-m_{t-} \rangle, \quad 
\mu \bigl( z_{t-} - m_{t-} \geq 1- \alpha \lambda \bigr)
\geq \lambda \biggr) =1,
\end{equation*}
which is enough to conclude that condition (4) of Definition \ref{def:nonlineareq} is also satisfied (a.e. under $\Pi_\infty$),  by invoking Proposition \ref{prop:assumption:physical:criterion}.


\subsection{Proof of Theorem \ref{convergence thm}}
Assume now that \eqref{nonlinear eq Z} admits a unique solution 
 $(Z_{t})_{t \in [0,T]}$ on $[0,T]$.
By identification, we deduce from the proof of Theorem \ref{weak existence thm} that, for
$\Pi_{\infty}$
a.e. $\mu$, the pair $(z_{t},m_{t})_{t \in [0,T)}$ has the same law, under 
$\mu$, as the pair 
 $(Z_{t},M_{t})_{t \in [0,T)}$ under $\P$. In particular, for 
 some fixed $S \in (0,T)$ such that ${\mathbb E}(M_{S})={\mathbb E}(M_{S-})$, 
 we have $\langle \mu,m_{S}\rangle = \langle \mu,m_{S-} \rangle$
 and $z_{S}=z_{S-}$
 for $\Pi_{\infty}$ a.e. $\mu$. 
 Define the map 
 \begin{equation}
 \label{eq:gammaS}
 \begin{split}
 \gamma_{S} :  \hat{\mathcal D}([0,T+1],\R) &\rightarrow  \hat{\mathcal D}([0,S],\R), \quad \gamma_S(z) =\hat{z},
 \end{split}
 \end{equation}
 where $\hat{z}$ is the element of $\hat{\mathcal D}([0,S],\R)$ given by ${z}_t$ if $t<S$ and $\hat{z}_S = z_{S-}$.
 The map $\gamma_{S}$ is measurable for the 
$\sigma$-fields
 on $\hat{\mathcal D}([0,T+1])$ and $\hat{\mathcal D}([0,S])$ generated by the evaluation mappings and thus 
for the Borel $\sigma$-fields generated by M1.  Moreover, the push forward of $\mu$ by $\gamma_{S}$, denoted by $\gamma_{S} 
\sharp \mu$, 
coincides, for $\Pi_{\infty}$ a.e. $\mu$, with the law of $(Z_{t})_{t \in [0,S]}$ under $\P$.

Thus, defining  $\Gamma_{S} (\mu) =\gamma_{S} \sharp \mu$ for an arbitrary $\mu \in {\mathcal P}(\hat{\mathcal D}([0,T+1],\R) )$, we have that $\Gamma_{S} \sharp \Pi_{\infty} = \textrm{Dirac}(\mu_{S})$, 
where $\mu_{S}:=\textrm{Law}((Z_{s})_{t \in [0,S]})$. Notice that $\Gamma_{S}$ is indeed measurable when both ${\mathcal P}(\hat{\mathcal D}([0,T+1],\R))$ and 
${\mathcal P}(\hat{\mathcal D}([0,S],\R))$ are equipped with the Borel $\sigma$-fields generated by the topology of weak convergence.

Moreover, as a consequence of Theorem \ref{thm:M1:2}, 
$\gamma_{S}$ is continuous at any $z \in \hat{\mathcal D}([0,T+1],\R)$ such that $z_{S-}=z_{S}$. 
In particular, if $\mu$ is a probability measure on 
$\hat{\mathcal D}([0,T+1])$ such that $\mu\{z_{S-}=z_{S}\}=1$, then $\mu$ is a point of continuity of $\Gamma_{S}$. 
Therefore, under the probability $\Pi_{\infty}$, 
a.e. $\mu$ is a continuity point of $\Gamma_{S}$. Since $(\Pi_{N})_{N \geq 1}$ converges towards $\Pi_{\infty}$ in the weak sense, we deduce that 
 \begin{equation*}
 \lim_{N \rightarrow + \infty} \Gamma_{S} \sharp \Pi_{N} = \Gamma_{S} \sharp \Pi_{\infty} = \textrm{Dirac}(\mu_{S}). 
 \end{equation*}
Since (with $\hat{Z}^{i,N} = \gamma_S({Z}^{i,N})$ given by  \eqref{eq:gammaS})
 \begin{equation*}
 \Gamma_{S} \sharp \Pi_{N} = 
 \P_{\gamma_{S} \sharp (N^{-1}
 \sum_{i=1}^N \textrm{Dirac}({\tilde{Z}^{i,N}}))}
 = \P_{N^{-1}\sum_{i=1}^N \textrm{Dirac}(\hat{Z}^{i,N})},
 \end{equation*}
 we  obtain that,
 in the weak sense,
 \begin{equation}
 \label{eq:convergence:51}
 \lim_{N \rightarrow + \infty}
 \P_{N^{-1}\sum_{i=1}^N \textrm{Dirac}({\hat{Z}^{i,N}})} = \textrm{Dirac}(\mu_{S}),
 \end{equation}
 where 
 $N^{-1}\sum_{i=1}^N \textrm{Dirac}(\hat{Z}^{i,N})$ and
 $\mu_{S}$ are seen as probability measures on $\hat{\mathcal D}([0,S],\R)$ equipped with M1. 
Since the law of the $N$-tuple $(\hat{Z}^{1,N},\dots,\hat{Z}^{N,N})$ is invariant by permutation, we
 deduce from \cite[Proposition 2.2]{sznitman} that the family $(\hat{Z}^{i,N})_{i = 1,\dots,N}$
 is chaotic on $\hat{\mathcal D}([0,S],\R)$ endowed with M1, that is, for any integer $k \geq 1$,
 \begin{equation}
 \label{eq:convergence:50}
 \bigl( \hat{Z}^{1,N},\dots,\hat{Z}^{k,N} \bigr) \Rightarrow \bigl( \mu_{S} \bigr)^{\otimes k} 
 = \P_{(Z_{s})_{s \in [0,S]}}^{\otimes k}\quad \textrm{as} \ N \rightarrow + \infty,
 \end{equation}
in the weak sense, on $[\hat{\mathcal D}([0,S],\R)]^{k}$ equipped with the product topology induced by M1. 
 By Lemma \ref{lem:13:7:2} below, this also proves that  
 \begin{equation}
 \label{eq:convergence:50:b}
 \bigl( (\hat{Z}^{1,N},\hat{M}^{1,N}),\dots,(\hat{Z}^{k,N},\hat{M}^{k,N}) \bigr) \Rightarrow \P_{({Z}_{s},\hat{M}_{s})_{s \in [0,S]}}^{\otimes k} \quad \textrm{as} \ N \rightarrow + \infty,
 \end{equation}
on $[\hat{\mathcal D}([0,S],\R) \times \hat{\mathcal D}([0,S],\R)]^{k}$ equipped with the product 
topology induced by M1. Indeed, assuming without loss of generality that the sequence representing the convergence in \eqref{eq:convergence:50} in the almost-sure sense is $( \hat{Z}^{1,N},\dots,\hat{Z}^{k,N}$) itself and denoting the a.s. limit  by $(Z^{1,\infty},\dots,Z^{k,\infty})$,
Lemma \ref{lem:13:7:2} says that, for any $\ell = 1,\dots,k$, a.s., for $t$ in a dense subset of $[0,S]$
\begin{equation}
\label{eq:limit:1}
\lfloor \bigl( \sup_{s \in [0,t] } \hat{Z}^{\ell,N}_{s} \bigr)_{+} \rfloor 
\rightarrow 
\lfloor \bigl( \sup_{s \in [0,t]} Z_{s}^{\ell,\infty} \bigr)_{+} \rfloor .
\end{equation}
Obviously, \eqref{eq:limit:1} holds at $t=0$ since both sides are zero. At $t=S$, we know that $\E(M_{S}^{\ell})
= \E(M_{S-}^{\ell})$ since $t \mapsto \E(M_{t})$ is continuous at $t=S$, so that $\P(M_{S}=M_{S-}=\hat{M}_{S})=1$.
Therefore, by Lemma \ref{lem:13:7:2}, \eqref{eq:limit:1} holds at $t=S$.
By Theorem \ref{thm:M1:2}, we deduce that, for any $\ell=1,\dots,k$, $((\hat{M}^{\ell,N}_{s})_{s \in [0,S]})_{N \geq 1}$ converges a.s. 
to $(M_{s}^{\ell,\infty})_{s \in [0,S]}$ in $\hat{\mathcal D}([0,S],\R)$, where 
$(M_{s}^{\ell,\infty})_{s \in [0,S]}$ is the counting process associated with 
$(Z_{s}^{\ell,\infty})_{s \in [0,S]}$. Actually, the Skorohod representation theorem says that the a.s. convergence holds for a representation sequence only, but, in any case, \eqref{eq:convergence:50:b} holds.

To complete the proof of Theorem \ref{convergence thm}, we use the Skorohod representation theorem once again.  In fact we can assume without loss of generality that the convergence in \eqref{eq:convergence:51} holds almost-surely, namely 
$N^{-1}\sum_{i=1}^N \textrm{Dirac}(\hat{Z}^{i,N})$ converges to $\mu_{S}$ almost surely. 
Lemma \ref{lem:2} below 
then guarantees that a.s.
\begin{equation*}
 \lim_{N \rightarrow + \infty} \frac{1}{N} \sum_{i=1}^N M_{s}^{i,N} = \E(M_{s}).
\end{equation*}
for all $s \in [0,S)$, except at any points of discontinuity of $\cdot \mapsto \E(M_{\cdot})$ (of which there are countably many).
Actually, the Skorohod representation theorem says that convergence holds almost-surely for a representation 
sequence only. Anyhow, the convergence always holds in probability (using for instance the fact that M1 
convergence is metrizable). 
Since $S$ can be chosen as close as needed to $T$, we can apply the above result to some $S' \in (S,T)$, where $S'$ is a continuity point of $\cdot \mapsto \E(M_{\cdot})$. This says that
the limit holds for all continuity points $s \in [0,S]$ since $S \in [0,S')$, and $S$ is also a continuity point. 
By Theorem \ref{thm:M1:2}, we deduce that the mapping 
$[0,S] \ni s \mapsto N^{-1} \sum_{i=1}^N \hat{M}_{s}^{i,N}$ converges 
to the mapping $[0,S] \ni s \mapsto \E(M_{s})$ in probability on 
$\hat{\mathcal D}([0,S],\R)$ equipped with the M1 topology, where 
$\hat{M}_{s}^{i,N} = M_{s}^{i,N}$ for $s \in [0,S)$
and $\hat{M}_{S}^{i,N} = M_{S-}^{i,N}$.


\subsection{Continuity of related mappings}

The aim of this section is to complete the proof of Theorems \ref{weak existence thm} and \ref{convergence thm} by providing the technical results needed therein.  In particular, we prove the key continuity result used in the proof of Theorem \ref{weak existence thm} i.e. the continuity of the functional given in \eqref{eq:functional} (see Lemma \ref{lem:continuity}). 

In the whole section, $S$ is a given positive real, and we make use of the notation and definitions of Section \ref{sec:M1top} for the M1 topology. Moreover, as above, $(z_{t})_{t \in [0,S]}$ will be the canonical process on $\hat{\mathcal D}([0,S],\R)$ 
and $m_{t} := \bigl\lfloor \bigl( \sup_{s \in [0,t]} z_{s} \bigr)_{+} \bigr\rfloor, \quad t \in [0,S]$.

We begin with the following continuity property.

\begin{lem}
\label{lem:13:7:2}
Consider a sequence $(z^n)_{n \geq 1}$ of functions in $\hat{\mathcal D}([0,S],\R)$, converging towards some 
$z \in \hat{\mathcal D}([0,S],\R)$ for {M1}. Assume that $z$ has the following 
\emph{crossing property}:
\begin{equation}
\label{eq:15:7:6}
\forall k \in {\mathbb N}^*, \ \forall h >0, \quad \tau^k < S \Rightarrow \sup_{t \in [\tau^k,\min(S,\tau^k+ h)]}
\bigl[ z_{t} - z_{\tau^k} \bigr] >0,
\end{equation}
where $\tau^k = \inf \{ t \in [0,S] : z_{t} \geq k \}$ ($\inf \emptyset =S$). 
Then, there exists an at most countable subset $J \subset [0,S]$, such that
\begin{equation*}
\forall t \in [0,S] \setminus J, \quad \lim_{n \rightarrow + \infty }
m^n_{t} = m_{t},
\end{equation*}
where
$m_{t}^n = 
\lfloor ( 
\sup_{s \in [0,t]} z_{s}^n )_{+} \rfloor$, $m_{t} =  
\lfloor ( 
\sup_{s \in [0,t]} z_{s} )_{+} \rfloor$.
The set $[0,S] \setminus J$ contains all the points $t$ of continuity of $z$ such that $(\sup_{s \in [0,t]} z_{s})_{+}$ is not in 
${\mathbb N} \setminus \{0\}$.  
\end{lem}

\begin{proof} Since $z^n \rightarrow z$ for the M1 topology, we can find 
a sequence of parametric representations $((u^n,r^n))_{n \geq 1}$ of $(z^n)_{n \geq 1}$ that converges towards
a parametric representation $(u,r)$ of $z$ uniformly on $[0,S]$ 
(see Theorem 12.5.1 in \cite{whitt:monograph}). 
For any $t \in [0,S]$, we thus have
\begin{equation}
\label{eq:15:7:1}
\lim_{n \rightarrow + \infty} \sup_{s \leq t} u^n_{s} = \sup_{s \leq t} u_{s}.
\end{equation}
Fix $t\in[0, S]$. Since ${\mathcal G}_{z^n} \subset (u^n,r^n)([0,S])$, we know that, for every $s \in[0,S]$, there exists 
$s' \in [0,S]$ such that 
$(u^n_{s'},r^n_{s'}) = (z_{s}^n,s)$. 
If $s< r_{t}^n$, it therefore must hold that $s' < t$, as $s' \geq t$ would imply $s=r^n_{s'} \geq r_{t}^n$ ($r^n$ is non-decreasing). 
Therefore, $\sup_{s \leq t} u^n_{s} \geq \sup_{s < r^n_{t}} z^n_{s}$.

Moreover, for any $s \in[0, S]$, $u^n_{s} \in [z^n_{r^n_{s}-},z^n_{r^n_{s}}]$, so that 
$u^n_{s} \leq \max(z^n_{r^n_{s}-},z^n_{r^n_{s}})$. Therefore, for $s\leq t$,
\begin{equation*}
u^n_{s} \leq \max \bigl( \sup_{s' \leq r^n_{t}} z^n_{s'-}, \sup_{s' \leq r^n_{t}} z^n_{s'} \bigr)
= \sup_{s' \leq r^n_{t}} z^n_{s'}.
\end{equation*}
In the end we see that
$\sup_{s < r^n_{t}} z^n_{s} \leq  
\sup_{s \leq t} u^n_{s} \leq \sup_{s \leq r^n_{t}} z^n_{s}$.  
Similarly, we have
$\sup_{s < r_{t}} z_{s} \leq  
\sup_{s \leq t} u_{s}\leq \sup_{s \leq r_{t}} z_{s}$. Thus, by \eqref{eq:15:7:1},
\begin{equation}
\label{eq:15:7:2}
\liminf_{n \rightarrow + \infty} \sup_{s \leq r^n_{t}} z^n_{s} \geq \sup_{s < r_{t}} z_{s},
\quad \limsup_{n \rightarrow + \infty} \sup_{s < r^n_{t}} z^n_{s} \leq \sup_{s \leq r_{t}} z_{s}. 
\end{equation}
If $r_{t}$ is a point of continuity of $z$, the two right-hand sides above coincide, and moreover, 
by Theorem \ref{thm:M1:2},
\begin{equation*}
\lim_{\delta \rightarrow 0}
\limsup_{n \rightarrow + \infty}
\sup_{\max(0,r_t-\delta) \leq s \leq \min(S,r_{t}+\delta)}
\vert z^n_{s} - z_{r_t} \vert = 0.
\end{equation*}
That is, for any $\varepsilon >0$, we can find $\delta >0$, such that, for $n$ large enough, 
\begin{equation}
\label{eq:15:7:4}
\sup_{\max(0,r_t-\delta) \leq s \leq \min(S,r_{t}+\delta)}
\vert z^n_{s} - z_{r_{t}} \vert \leq \varepsilon.
\end{equation}
In particular, for $n$ large enough that $\vert r^n_{t} - r_{t} \vert < \delta$,
\begin{equation}
\label{eq:15:7:5}
\begin{split}
\sup_{s \leq r_{t}} z^n_{s} &= \max \Bigl( \sup_{s \leq \max(0,r_{t}-\delta)} z^n_{s},
\sup_{\max(0,r_{t}- \delta) \leq s \leq r_{t}} z^n_{s} \Bigr)
\\
&\leq \max \Bigl( \sup_{s \leq \max(0,r_{t}-\delta)} z^n_{s},
z_{r_{t}} \Bigr) + \varepsilon
\leq \max \Bigl( \sup_{s < r_{t}^n} z^n_{s},
z_{r_{t}} \Bigr) + \varepsilon.
\end{split}
\end{equation}
Therefore, by \eqref{eq:15:7:2},
\begin{equation}
\label{eq:15:7:3}
\limsup_{n \rightarrow + \infty} \sup_{s \leq r_{t}} z^n_{s} 
\leq \sup_{s \leq r_{t}} z_{s}.
\end{equation}
Similarly, for $n$ large enough,
we have 
\begin{equation*}
\begin{split}
\sup_{s \leq r_{t}} z^n_{s} &\geq \max \bigl( \sup_{s \leq \max(0,r_{t}-\delta)} z^n_{s},z_{r_{t}} \bigr) - \varepsilon
\\
&\geq \max \bigl( \sup_{s \leq \max(0,r_{t}-\delta)} z^n_{s},
\sup_{\max(0,r_{t}-\delta) \leq s \leq r^n_{t}}
z_{s}^n \bigr) - 2\varepsilon
\geq \sup_{s \leq r_{t}^n} z^n_{s} - 2 \varepsilon. 
\end{split}
\end{equation*}
By \eqref{eq:15:7:2} again, 
$\liminf_{n \rightarrow + \infty} \sup_{s \leq r_{t}} z^n_{s} 
\geq \sup_{s < r_{t}} z_{s}$.
By \eqref{eq:15:7:3}, we deduce that 
\begin{equation*}
\lim_{n \rightarrow + \infty} \sup_{s \leq r_{t}} z^n_{s} 
=
\sup_{s \leq r_{t}} z_{s},
\end{equation*}
whenever $r_{t}$ is a continuity point of $z$. Since $r$ maps $[0,S]$ onto itself, we deduce that 
for any continuity point $t \in [0,S]$ of $z$, 
\begin{equation*}
\lim_{n \rightarrow + \infty} \sup_{s \leq t} z^n_{s} 
=
\sup_{s \leq t} z_{s}.
\end{equation*}

Now, if $(\sup_{s \leq t} z_{s})_{+}$ is not an integer, we deduce that 
\begin{equation*}
\lim_{n \rightarrow + \infty} \bigl\lfloor \bigl( \sup_{s \leq t} z^n_{s}\bigr)_{+} \bigr \rfloor = \bigl \lfloor 
\bigl( \sup_{s \leq t} z_{s} \bigr)_{+} \bigr \rfloor. 
\end{equation*}
If $(\sup_{s \leq t} z_{s})_{+}$ is an integer, there are two cases. If the integer is zero, it means that 
$\sup_{s \leq t} z_{s} \leq 0$, so that 
$\lim_{n \rightarrow + \infty} \sup_{s \leq t} z^n_{s} \leq 0$ and thus
$$\lim_{n \rightarrow + \infty} (\sup_{s \leq t} z^n_{s})_{+} = 0.$$ In this case, 
$\lim_{n \rightarrow + \infty} \lfloor (\sup_{s \leq t} z^n_{s})_{+} \rfloor = 0$. 
If the integer is not zero and $t < S$, the crossing property \eqref{eq:15:7:6} says that $t$ must be the first time when this integer is crossed (rather than just touched),
so that the number of points of continuity of $z$ for which the convergence of the integer parts can fail is finite. 
\end{proof}

The following {corollary} is simply a result of Lebesgue's dominated convergence theorem.
\begin{cor}
\label{corol:13:7:1}
Under the assumptions of Lemma \ref{lem:13:7:2}, the functions
$[0,S] \ni t \mapsto \int_{0}^{t} b \bigl(z_{s}^n - m_{s}^n) ds$
converge in ${\mathcal C}([0,S],\R)$ towards
$[0,S] \ni t \mapsto \int_{0}^{t} b \bigl(z_{s} - m_{s}) ds$.
\end{cor}

The next result guarantees the convergence of the expectation of $m_t$  whenever $\mu^n\to\mu$ and the canonical process $z$ satisfies the crossing property of Lemma \ref{lem:13:7:2} under $\mu$, under appropriate assumptions.
 
\begin{prop}
\label{prop:13:7:1}
Assume that $(\mu^n)_{n\geq1}$ is a sequence of probability measures on the space $\hat{\mathcal D}([0,S],\R)$, converging towards some $\mu$ for the weak topology deriving from M1. If
$
\sup_{n \geq 1} \langle \mu^n,\sup_{t \in [0,S]} \vert z_{t} \vert^2 \rangle < + \infty$,
then, at any point of continuity of the mapping $(0,S) \ni t \mapsto \langle \mu,m_{t} \rangle$, it holds that
\begin{equation*}
\lim_{n \rightarrow + \infty}
\langle \mu^n,m_{t} \rangle = \langle \mu,m_{t} \rangle,
\end{equation*}
provided that, for any integer $k \geq 1$ and for 
$\tau^k := \inf\{ t \in [0,S] : z_{t} \geq k\}$,
\begin{equation}
\label{eq:crossing}
\forall h >0, \quad \mu \Bigl\{ \tau^k < S, \sup_{s \in [\tau^k,\min(S,\tau^k +h)]} (z_{s} - z_{\tau^k}) = 0
\Bigr\} =0,
\end{equation}
\end{prop}

\begin{proof} Let $\delta >0$. Then, for every $n\geq1$, the function
\begin{equation*}
[0,S] \ni t \mapsto 
\bar{m}^n_{t} = 
\left\{
\begin{array}{l}
\langle \mu^n, m_{\delta} \rangle, \quad t \in [0,\delta],
\\
\langle \mu^n, m_{t} \rangle, \quad t \in [\delta,S- \delta],
\\
\langle \mu^n, m_{S-\delta} \rangle, \quad t \in [S - \delta,S]
\end{array}
\right.
\end{equation*}
is (deterministic) non-decreasing nonnegative and c\`adl\`ag such that
$\sup_{n \geq 1} \bar{m}^n_{S} < + \infty$.

Since each $\bar{m}^n$ is constant on both $[0,\delta]$ and $[S-\delta,S]$,
the sequence $(\bar{m}^n)_{n \geq 1}$ is relatively compact for the M1 topology by 
Theorem \ref{thm:M1:3}. 
Extracting a convergent subsequence, still indexed by $n$, and denoting the limit 
by $\bar{m}$, we have, for any bounded and measurable function
$f$ from $[0,S]$ into $\R$ vanishing outside $(\delta,S-\delta)$, that
\begin{equation*}
\lim_{n \rightarrow + \infty} \int_{0}^{S} f_{t} \bar{m}^n_{t} dt = \int_{0}^{S} f_{t} \bar{m}_{t} dt.
\end{equation*}
This follows from the dominated convergence theorem and the fact that
$(\bar{m}^n)_{n \geq 1}$ converges pointwise towards $\bar{m}$ on a subset of $[0,S]$ of full Lebesgue measure, see
 Theorem \ref{thm:M1:1}. Since $f$ vanishes outside $(\delta,S-\delta)$, we can forget the definition of 
$\bar{m}^n$ outside $(\delta,S-\delta)$. Therefore, by the Skorohod representation theorem 
\begin{equation*}
\int_{0}^{S} f_{t} \bar{m}^n_{t} dt = \E \int_{0}^{S} f_{t} \bigl\lfloor \bigl( 
\sup_{s \in [0,t]} \zeta_{s}^n \bigr)_{+} \rfloor dt, 
\end{equation*}
where $(\zeta^n)_{n \geq 1}$ is a sequence of processes distributed according to $(\mu^n)_{n \geq 1}$ and converging a.s. 
for {M1} towards some process $\zeta$ with {law} $\mu$. 
By assumption, we know that, a.s., 
\begin{equation*}
\forall k \geq 1, \quad \forall h >0, \qquad
\tau^k < S \Rightarrow
\sup_{s \in [\tau^k,\min(S, \tau^k +h)]} (\zeta_{s} - \zeta_{\tau^k}) >0, 
\end{equation*} 
so that almost all paths of $\zeta$ satisfy the assumption of Lemma \ref{lem:13:7:2}. 
We deduce that, a.s., the functions
$[0,S] \ni t \mapsto \lfloor ( 
\sup_{s \in [0,t]} \zeta^n_{s} )_{+} \rfloor$
converge pointwise towards 
$[0,S] \ni t \mapsto \lfloor ( 
\sup_{s \in [0,t]} \zeta_{s} )_{+} \rfloor$
on the complement of an at most countable set. Therefore, a.s. 
\begin{equation*}
\lim_{n \rightarrow + \infty} \int_{0}^{S} f_{t} 
\bigl\lfloor \bigl( 
\sup_{s \in [0,t]} \zeta^n_{s} \bigr)_{+} \bigr\rfloor dt = 
\int_{0}^{S} f_{t} 
\bigl\lfloor \bigl( 
\sup_{s \in [0,t]} \zeta_{s} \bigr)_{+} \bigr\rfloor dt,
\end{equation*}
by the Lebesgue dominated convergence theorem.  Using this result again, we deduce that 
\begin{equation*}
\begin{split}
\lim_{n \rightarrow + \infty} \E \left[\int_{0}^{S} f_{t} 
\bigl\lfloor \bigl( 
\sup_{s \in [0,t]} \zeta^n_{s} \bigr)_{+} \bigr\rfloor dt\right]
&= 
\E \left[\int_{0}^{S} f_{t} 
\bigl\lfloor \bigl( 
\sup_{s \in [0,t]} \zeta_{s} \bigr)_{+} \bigr\rfloor dt\right]
= \int_{0}^{S} f_{t} \langle \mu,m_{t} \rangle dt, 
\end{split}
\end{equation*}
the last equality following from the fact that the law of 
$(\zeta_{s})_{0 \leq s \leq S}$ is $\mu$. 
By right-continuity, this proves that $\bar{m}_{t} = \langle \mu,m_{t} \rangle$ for any $t \in (\delta,S-\delta)$. Therefore,
at any point of continuity of $(\delta,S-\delta) \ni t \mapsto \langle \mu,m_{t} \rangle$, Theorem \ref{thm:M1:2} says that
\begin{equation*}
\lim_{n \rightarrow + \infty} \langle \mu^n,m_{t} \rangle = \langle \mu,m_{t} \rangle.
\end{equation*}
Letting $\delta$ tend to $0$, we complete the proof, as the set of discontinuity points of $[0,S] \ni t \mapsto \langle \mu,m_{t} \rangle$ is at most countable. 
\end{proof}

The previous result is general.  In order to be able to apply it,
we need to check {that} whenever $\Pi_\infty$ is a limit point of the family $(\Pi_N)_{N\geq1}$ (see
Theorem \ref{weak existence thm}), any measure $\mu$ in the support of $\Pi_\infty$ must satisfy the crossing property.

\begin{lem}
\label{lem:2}
For a.e. $\mu$ under ${\Pi}_{\infty}$, for any integer $k \geq 1$ and any real $h>0$,
\begin{equation*}
\mu \Bigl\{\tau^k <T+1,  \sup_{s \in [\tau^k,\min(T+1,\tau^k +h)]} (z_{s} - z_{\tau^k}) =0
\Bigr\} =0, \quad 
{\rm with} \ \tau^k = \inf\{ t \geq 0 : z_{t} \geq k\}.
\end{equation*}
In particular, 
if $(\mu^n)_{n\geq1}$ is a sequence of probability measures on the space 
$\hat{\mathcal D}([0,S],\R)$, converging towards $\mu$ for the weak topology deriving from M1,
then, at any point of continuity of the mapping $(0,S) \ni t \mapsto \langle \mu,m_{t} \rangle$, it holds that
$\lim_{n \rightarrow + \infty}
\langle \mu^n,m_{t} \rangle = \langle \mu,m_{t} \rangle$.
\end{lem}

\begin{proof} 
In order to prove this result, we actually will introduce a second empirical measure
\begin{equation*}
\bar{\theta}_{N} := \frac{1}{N} \sum_{i=1}^N \textrm{Dirac}(\tZ^{i,N},W^i).
\end{equation*}
We will consider $\bar{\theta}_{N}$ as a random probability measure on $\hat{\mathcal D}([0,T+1],\R) \times {\mathcal C}([0,T+1],\R)$, endowed with the product of the M1 topology and of the standard topology of uniform convergence. 
Note that the marginal distribution of $\bar{\theta}_{N}$ on the first coordinate is $\bar{\mu}_{N}$.  We also define $\Xi_N := \mathrm{Law}(\bar{\theta}_N)$.
Following Lemma \ref{lem:tightness2}, we have that the family $({\Xi}_{N})_{N \geq 1}$ is tight on 
${\mathcal P}(\hat{\mathcal D}([0,T+1],\R) \times {\mathcal C}([0,T+1],\R))$ (endowed with the topology of weak convergence), so that we can extract a convergent subsequence of $({\Xi}_{N})_{N \geq 1}$, still indexed by $N$. We denote its limit 
by ${\Xi}_{\infty}$. 

Returning to the particle system, we first claim that there exists a constant $c \geq 0$ such that, for any $N \geq 1$ and $i=1,\dots,N$,
\begin{equation*}
\tilde{Z}^{i,N}_{t} - \tZ^{i,N}_{s} \geq W^i_{t} - W^i_{s} - c \bigl( 1+ \sup_{v \in [0,T]} \vert Z_{v}^{i,N} \vert \bigr) (t-s),
\end{equation*}
for $s,t \in [0,T+1]$ with $s \leq t$.  Indeed this follows from the fact that $b$ is Lipschitz and 
that $M^{i,N}_t\leq \sup_{v \in [0,T]} \vert Z_{v}^{i,N} \vert$, for all $t\in[0, T]$ and every $i= 1, \dots, N$.  

Denoting by $(z_{t},w_{t})_{t \in [0,T+1]}$ the canonical process on 
$\hat{\mathcal D}([0,T+1],\R) \times {\mathcal C}([0,T+1],\R)$, we get that, $\P$-a.s., 
\begin{equation}
\label{eq:8:6:1}
z_{t} - z_{s} \geq w_{t} - w_{s} - c \bigl( 1+ \sup_{v \in [0,T+1]} \vert z_{v} \vert \bigr) (t-s), \quad 0\leq s \leq t \leq T+1.
\end{equation}
holds under the empirical measure $\bar{\theta}_{N}$, or that
\eqref{eq:8:6:1} holds a.s. under a.e. probability measure $\theta$ under $\Xi_{N}$. Since $(\Xi_{N})_{N \geq 1}$ converges towards $\Xi_{\infty}$ in the weak sense, we know from the Skorohod representation theorem that, on some probability space (still denoted by $(\Omega,{\mathcal A},\P)$), there exists a sequence of random probability measures 
$(\vartheta_{N})_{N \geq 1}$ on the space $\hat{\mathcal D}([0,T+1],\R) \times {\mathcal C}([0,T+1],\R)$, such that 
$(\vartheta_{N})_{N \geq 1}$ converges towards some random probability measure $\vartheta_{\infty}$ a.s., with 
$\vartheta_{N}$ being distributed according to $\Xi_{N}$ and $\vartheta_{\infty}$ according to $\Xi_{\infty}$. 
Hence, a.s., under $\vartheta_{N}$ the canonical process $(z, w)$ has the property \eqref{eq:8:6:1}. 

\vspace{0.3cm}
\noindent\textit{Step 1:} The first step is to show that under $\vartheta_\infty$ the canonical process $(z, w)$ also satisfies \eqref{eq:8:6:1}, simply using the facts that it is true under $\vartheta_{N}$ for each $N$, and that $\vartheta_N$ converges to $\vartheta_\infty$ a.s.
Again, we can use the Skorohod representation theorem (still with
$(\Omega,{\mathcal A},\P)$ as the underlying probability space). Indeed, we can find a sequence of processes $(\zeta^N,\xi^N)$ with law $\vartheta_N$ under $\P$, converging a.s. towards some process $(\zeta,\xi)$ distributed according to $\vartheta_\infty$. Since \eqref{eq:8:6:1} holds under $\vartheta_{N}$ for each $N$, we have that $\P$ a.s., for any $0 \leq s \leq t\leq T+1$,
\begin{equation}
\label{under limit measure 1}
\zeta^N_{t} - \zeta^N_{s} \geq \xi^N_{t} - \xi^N_{s} - c \bigl( 1+ \sup_{v \in [0,T+1]} 
\vert \zeta^N_{v} \vert
\bigr) (t-s).
\end{equation}
We want to prove that, $\P$ a.s., for any $s\leq t$,
\begin{equation}
\label{under limit measure 2}
\zeta_{t} - \zeta_{s} \geq \xi_{t} - \xi_{s} - c \bigl( 1 + \sup_{v \in [0,T+1]}
\vert \zeta_{v} \vert
\bigr) (t-s).
\end{equation}
It is sufficient to prove that, for an arbitrary sequence $(\zeta^N,\xi^N)_{N \geq 1}$ satisfying \eqref{under limit measure 1} and converging towards $(\zeta,\xi)$ in $\hat{\mathcal D}([0,T+1],\R) \times {\mathcal C}([0,T+1],\R)$ equipped with the product topology derived from the M1 and uniform topologies, the limit $(\zeta,\xi)$ satisfies \eqref{under limit measure 2}. 

To this end, if $\zeta^N \to \zeta$ in $\hat{\mathcal D}([0,T+1],\R)$ with respect to the M1 topology, then by \cite[Theorem 12.5.1]{whitt:monograph}, there exist parametric representations $(u^N, r^N)$ of $\zeta^N$ and $(u, r)$ of $\zeta$ (see Section \ref{sec:M1top} for definition) such that $\|u^N - u\|\vee\|r^N - r\| \to 0$, where
\begin{equation*}
\forall t \in [0,T+1], \quad u^N_{t} \in \bigl[\zeta^N_{r^N_{t}-},\zeta^N_{r^N_{t}}\bigr], \quad 
u_{t} \in \bigl[\zeta_{r_{t}-},\zeta_{r_{t}}\bigr]. 
\end{equation*}
Noting that $\sup_{0 \leq v \leq T+1} \vert \zeta^N_{v} \vert 
= \sup_{0 \leq v \leq T+1} \vert u^N_{v} \vert$, we have by \eqref{under limit measure 1}, for any $s\leq t$,
\begin{equation*}
\max\bigl(\zeta^N_{r^N_{t}-},\zeta^N_{r^N_{t}}\bigr) - 
\min\bigl(
\zeta^N_{r^N_{s}-},\zeta^N_{r^N_{s}}\bigr) \geq \xi_{r^N_{t}}^N - \xi_{r^N_{s}}^N - c 
\bigl( 1+ \sup_{v \in [0,T+1]}
\vert u^N_{v} \vert \bigr)
\bigl(r^N_t-r^N_s\bigr),
\end{equation*}
since the right-hand side is continuous in the subscript parameter.
Thus
\begin{equation*}
u^N_{t} - u^N_{s} \geq\xi_{r^N_{t}}^N - \xi_{r^N_{s}}^N - c 
\bigl( 1+ \sup_{v \in [0,T+1]}
\vert u^N_{v} \vert \bigr)
\bigl(r^N_t-r^N_s\bigr).
\end{equation*}
Passing to the limit as $N\to\infty$ and using the continuity of $\xi$, we obtain, for any $s \leq t$, 
\begin{equation*}
u_{t} - u_{s} \geq \xi_{r_{t}} - \xi_{r_{s}} - c
\bigl( 1+ \sup_{v \in [0,T+1]}
\vert u_{v} \vert \bigr)
(r_t-r_{s}). 
\end{equation*}
For any $t',s' \in [0,T+1]$, $s' \leq t'$, we can find $t,s \in [0,T+1]$ with $s \leq t$ such that $\zeta_{t'}=u_{t}, \zeta_{s'}=u_{s}$ and $t'=r_{t},s'=r_{s}$. This proves that
\begin{equation*}
\zeta_{t'} - \zeta_{s'} \geq \xi_{t'} - \xi_{s'} - c \bigl( 1+ \sup_{v \in [0,T+1]}
\vert \zeta_{v} \vert \bigr) (t'-s'),
\end{equation*}
and thus that $(\zeta, \xi)$ satisfies \eqref{under limit measure 2}.

\vspace{0.3cm}
\noindent\textit{Step 2:}  We now use the previous step to prove the lemma.  Indeed, by Step 1 we have that \eqref{eq:8:6:1} holds a.s. under $\theta$, for almost all $\theta$ under $\Xi_\infty$.  Thus for almost all $\theta$ under $\Xi_\infty$, we have for any real $R$, integer $k$ and $h>0$
\begin{equation*}
\begin{split}
&\theta \bigl\{\tau^k  < T+1,  \sup_{s \in [\tau^k,\min(T,\tau^k +h)]} (z_{s} - z_{\tau^k}) =0
\bigr\} \leq  \theta \bigl\{ \sup_{v \in [0,T+1]} \vert z_{v} \vert \geq R 
\bigr\}
\\
&\hspace{30pt} + \theta \bigl\{\tau^k < T+1,  \sup_{s \in [\tau^k ,\min(T,\tau^k +h)]} [w_{s} - w_{\tau^k} - c (1+ R) (s-\tau^k)] =0
\bigr\}. 
\end{split}
\end{equation*}
Clearly, $\tau^k$ is a stopping time for the filtration generated by $(z,w)$. Assume that $w$ is a Brownian motion w.r.t. 
this filtration under $\theta$ (this is proved below). Then the
strong Markov property says that the second term in the right-hand side of the above is zero
for any $h>0$.  Letting $R$ tend to $+ \infty$, we get
\begin{equation}
\label{eq:17.04.1}
\theta \bigl\{\tau^k  < T+1,  \sup_{s \in [\tau^k,\min(T,\tau^k +h)]} (z_{s} - z_{\tau^k}) =0
\bigr\} =0 \quad 
\textrm{for a.e. $\theta$ under $\Xi_\infty$}. 
\end{equation}

We finally claim that \eqref{eq:17.04.1} then also holds for all $\mu$ under $\Pi_\infty$.  For $\theta\in \mathcal{P}(\hat{\mathcal D}([0,T+1],\R)\times{\mathcal C}([0,T+1],\R))$ let $\Psi(\theta)\in\mathcal{P}(\hat{\mathcal D}([0,T+1],\R))$ be the marginal of $\theta$ on the first coordinate.  Since
 $\theta\mapsto \Psi(\theta)$ is obviously continuous and $\Pi_N = \Psi\sharp \Xi_N$, where $\sharp$ indicates the push forward map, we have $\Pi_{\infty} = \Psi \sharp \Xi_{\infty}$. 
 Then, for any Borel subset $A \subset 
 \hat{\mathcal D}([0,T+1],\R)$, 
 $\int \theta \{(z_{t})_{t \in [0,T+1]} \in A\} d \Xi_{\infty}(\theta) =
 \int \mu \{(z_{t})_{t \in [0,T+1]} \in A\} d \Pi_{\infty}(\mu)$. 
 Choosing $A = 
 \{\tau^k  < T+1,  \sup_{s \in [\tau^k,\min(T,\tau^k +h)]} (z_{s} - z_{\tau^k}) =0\}$, 
we complete the proof.  

It thus remains to prove that $w$ is a Brownian motion under the filtration generated by $(z,w)$, for which it is sufficient to prove that 
\begin{equation*}
\begin{split}
&\int f(w_{t}-w_{s}) \prod_{i=1}^n g_{i}(z_{s_{i}},w_{s_{i}}) d\theta(z,w)
= \int f(w_{t}-w_{s}) d\theta(z,w)  \int \prod_{i=1}^n g_{i}(z_{s_{i}},w_{s_{i}}) d\theta(z,w),
\end{split}
\end{equation*}
for bounded and continuous functions $g_{i}$ and $f$, and for $0 \leq s_{1} < \dots < s_{n} \leq s < t$ i.e. $w$ is a continuous martingale. This follows from the convergence of the finite-dimensional distributions up to a countable subset under the weak convergence for the M1 Skorokhod topology (see \cite[Theorem 3.2.2]{skorohod}) and from the right-continuity of the paths. 
\end{proof}

We then finally arrive at the required continuity lemma:
\begin{lem}
\label{lem:continuity}
Under $\Pi_\infty$, the functional \eqref{eq:functional}
is a.e. continuous.
\end{lem}

\begin{proof} 
Let $\mu$ be in the support of $\Pi_\infty$, and let $(\mu^n)_{n\geq1}$ be a sequence of measures on $\hat{\mathcal D}([0,T+1],\R)$ converging towards $\mu$ (in the weak topology induced by the M1 topology).
By choice of $S_{1}<\dots<S_{p}=S$ in the definition 
{\eqref{eq:20:02:14:5}}
of $F$, both the canonical process and $\langle \mu,m\rangle$ are a.s. continuous at $S_{1},\dots,S_{p}$. We want to prove that  
\begin{equation*}
\begin{split}
&\lim_{n \rightarrow + \infty}
\biggl\langle \mu^n,F\biggl( z_{\cdot} - z_{0} - \int_{0}^{\cdot} b(z_{s} - m_{s}) ds - \alpha \langle \mu^n,m_{\cdot} \rangle \biggr) \biggr\rangle
\\
&= \biggl\langle \mu,F\biggl( z_{\cdot}  - z_{0}- \int_{0}^{\cdot} b(z_{s} - m_{s}) ds - \alpha \langle \mu,m_{\cdot} \rangle \biggr) \biggr\rangle.
\end{split}
\end{equation*}
Again, we make use of the Skorohod representation theorem. We consider a sequence of processes $(\zeta^n)_{n \geq 1}$, 
with $(\mu^n)_{n \geq 1}$ as distributions on $\hat{\mathcal D}([0,T+1],\R)$, converging a.s. towards 
some process $\zeta$, admitting $\mu$ as distribution. We set, for any $t \in [0,T]$, 
\begin{equation*}
\eta^n_{t} = \bigl\lfloor \bigl( \sup_{s \in [0,t]} \zeta^n_{s} \bigr)_{+} \bigr\rfloor, \quad 
\eta_{t} = \bigl\lfloor \bigl(\sup_{s \in [0,t]} \zeta_{s} \bigr)_{+} \bigr\rfloor.  
\end{equation*}
Since, a.s., $S$ is a point of continuity of $\zeta$, we deduce 
from Theorems \ref{thm:M1:1} and \ref{thm:M1:2}
that, a.s.,  
$(\hat{\zeta}^n_{s})_{s \in [0,S]}$ converges towards 
$(\zeta_{s})_{s \in [0,S]}$ in $\hat{\mathcal D}([0,S],\R)$, where $\hat{\zeta}^n_s = \zeta^n_s$ if $s<S$ and $\hat{\zeta}^n_S = \zeta^n_{S-}$ as usual.

By Corollary \ref{corol:13:7:1}, a.s.,
\begin{equation*}
\lim_{n \rightarrow + \infty}
\sup_{v \in [0,S]} \biggl\vert \int_{0}^v b(\zeta^n_{s}-\eta^n_{s}) ds 
- \int_{0}^v b(\zeta_{s}-\eta_{s}) ds \biggr\vert = 0. 
\end{equation*}
Moreover, since $S_{0},S_{1},\dots,S_{p}$ are almost-surely points of continuity of $\zeta$, we deduce from Theorem \ref{thm:M1:2}
that $\zeta^n_{S_{i}} \rightarrow \zeta_{S_{i}}$ as $n \rightarrow + \infty$, for any $i \in \{0,\dots,p\}$. 

Similarly, since $S_{1},\dots,S_{p}$ are points of continuity of $[0,T] \ni t \mapsto \langle  \mu,m_{t} \rangle$ (which coincides with the expectation of 
$\eta_{t}$ under $\mu$), we deduce from Proposition \ref{prop:13:7:1} that 
$\langle \mu^n,m_{S_{i}} \rangle \rightarrow
\langle \mu, m_{S_{i}} \rangle$ as $n \rightarrow + \infty$, 
for any $i \in \{1,\dots,p\}$. This proves the lemma.

\end{proof}

\subsection{Proof of 
Theorems \ref{weak existence thm:delay}
and \ref{convergence thm:delay}}

Theorems \ref{weak existence thm:delay}
and \ref{convergence thm:delay} are proved in a completely similar way to Theorems \ref{weak existence thm}
and \ref{convergence thm}. 

We first discuss the proof 
of Theorem \ref{weak existence thm:delay}, using the same notation as in the statement. 
Following the proof Lemma \ref{lem:particle:5}, we can prove that 
$\sup_{\delta \in (0, 1)}
{\mathbb E} [ \sup_{t \in [0,T]} \vert Z_{t}^{\delta} \vert^p + 
( M_{T}^{\delta})^p ]$ is finite for any $p \geq 1$. Following the proof of 
\cite[Lemma 5.2]{DIRT}, we know that 
$e_{\delta}(t) \leq C t^{1/2}$, for a constant $C$ independent of 
$\delta$. 
Following the proof of  Lemma \ref{lem:tightness1}, we deduce that 
the {laws} of $(\mu^{\delta})_{\delta \in (0,1)}$ are tight in 
$\hat{\mathcal D}([0,T+1],\R)$, where $\mu^{\delta}$ 
is defined in the statement of Theorem \ref{weak existence thm:delay}. 
In order to pass to the limit along convergent subsequences, 
we consider, for a given weak limit $\mu$, the countable set of points $J$ in $[0,T+1]$
at which the function
\begin{equation*}
[0,T+1] \ni t \mapsto \mu \bigl\{ z_{t-} = z_{t} \bigr\}
\end{equation*}
differs from one. Checking that $\mu$ satisfies the `crossing' property in 
Lemma \ref{lem:2}, it is then quite straightforward to pass to the limit 
in the identity
\begin{equation*}
Z_{S_{i+1}}^{\delta} = Z_{S_{i}}^{\delta} + \int_{S_{i}}^{S_{i+1}}
b\bigl(Z_{s}^{\delta} - M_{s}^{\delta}) ds + \alpha \langle \mu^{\delta},m_{S_{i+1}}
- m_{S_{i}} \rangle + W_{S_{i+1}} - W_{S_{i}},
\end{equation*}
for points $0=S_{0} < S_{1} < \dots < S_{p} < T$ that are not in $J$. 
This permits to prove that, under $\mu$, 
the canonical process $(z_{t})_{t \in [0,T]}$ {
satisfies (2) in Theorem
\ref{weak existence thm}.}

The most difficult point is to check (3).  It follows from Lemma
\ref{lem:last one} {below}, which is the counterpart of Proposition 
\ref{prop:20:02:14:1} for the particle system. 
The end of the proof is then similar. 
The proof of Theorem \ref{convergence thm:delay} works on the same model as the one 
of Theorem \ref{convergence thm}. 

\begin{lemma}
\label{lem:last one}
For a given $T>0$, consider $0 \leq t<t+h \leq T$, $h\in(0, 1)$. Then, we can find 
$C \geq 0$, independent of $h$, such that, for any $\delta \in (0,1)$, 
\begin{equation*}
\begin{split}
&e_{\delta}(t+h) -e_{\delta}(t-)  = e_{\delta}(t+h) -e_{\delta}(t)  \leq 1 + C h^{{1/8}}, 
\\
& \forall \lambda \leq e_{\delta}(t+h) - e_{\delta}(t), \quad \P \bigl(
\ X_{t-}^{\delta}
\geq 1 - \alpha \lambda -C  h^{{1/8}}  \bigr) \geq {\lambda}
+ e_{\delta}(t) - e_{\delta}(t+\delta) - C h^{{1/8}}. 
\end{split}
\end{equation*}
\end{lemma}

Notice that the second statement in Lemma 
\ref{lem:last one}
slightly differs from the second statement in Proposition \ref{prop:20:02:14:1}. However, the application 
for passing to the limit is the same. 
With the same notation as in 
\eqref{eq:18:02:14:1}, it says that, for $t,t+h,t+\delta' \in J^{\complement}$ and 
$\lambda < \langle \mu,m_{t+h}-m_{t}\rangle$, 
it holds $\mu( z_{t-} - m_{t-} \geq 1 - \alpha \lambda - C h^{1/8}) \geq 
\lambda + \langle \mu,m_{t} - m_{t+\delta'}\rangle - Ch^{1/8}$. Letting $\delta' \downarrow 0$, it permits to 
recover \eqref{eq:18:02:14:1}. 

\begin{proof}
Given some $\delta >0$ and $0 \leq t < t+h \leq T$, consider 
${\sigma^{\delta}_{t}} := \inf\{s \geq t : M^{\delta}_{s} - M^{\delta}_{t-} = 1\}$ 
and
{$\tau^{\delta}_{t} := \inf\{s > t : M^\delta_{s} - M^\delta_{t-}=2\}$}. As $e_{\delta}$ is continuously differentiable, 
we know from \cite[Lemma 4.2]{DIRT} that  
$\sigma^{\delta}_{t}$ and $\tau^{\delta}_{t}$ have differentiable cumulative distribution functions. 
Recalling the definition of \eqref{delayed equation} and setting $Z^\delta = X^\delta + M^\delta$ as usual, we see that
\begin{equation*}
\sup_{s \in [t,t+h]} \vert Z_{s}^{\delta} - Z_{t}^{\delta} \vert \leq \alpha \bigl( e_{\delta}(t+h) - e_{\delta}(t) \bigr) + 
\sup_{s \in [t,t+h]} \vert W_{s}^{\delta} - W_{t}^{\delta} \vert+ C h \bigl( 1 + \sup_{s \in [0,T]} \vert Z_{s}^{\delta} \vert\bigr). 
\end{equation*}
On $\{\tau^{\delta}_{t} \leq t +h\}$, 
$(Z^{\delta}_{s})_{s \in [t,t+h]}$ crosses at least two new integers, i.e.
$\sup_{s \in [t,t+h]} \vert Z_{s}^{\delta} - Z_{t}^{\delta} \vert \geq 1$: 
\begin{equation}
\label{eq:clinton}
1 \leq \alpha \bigl( e_{\delta}(t+h) - e_{\delta}(t) \bigr) + 
  \sup_{s \in [t,t+h]} \vert W_{s}^{\delta} - W_{t}^{\delta} \vert +
C h \bigl( 1 + \sup_{s \in [0,T]} \vert Z_{s}^{\delta} \vert\bigr),
\end{equation}
where (with the convention that $M^{\delta}_{r}=M_{0}^{\delta}$, 
$\sigma^{\delta}_{r}=\sigma_{0}^{\delta}$
 and $\tau^{\delta}_{r}=\tau_{0}^{\delta}$
if $r\leq0$),

\begin{equation}
\label{eq:clinton2}
\begin{split}
e_{\delta}(t+h) - e_{\delta}(t) &\leq \P \bigl(M_{t+h-\delta}^{\delta} - M_{t-\delta}^{\delta}\geq 1 \bigr) + {\mathbb E}\bigl[ \bigl(M_{t+h-\delta}^{\delta} - M_{t-\delta}^{\delta} \bigr) 
{\mathbf 1}_{\{\tau^{\delta}_{t-\delta} \leq t+h - \delta\}}
\bigr] 
\\
&\leq \P \bigl(\sigma_{t-\delta}^{\delta}\leq t+h - \delta \bigr) 
+ C \bigl( {\mathbb P}(\tau^{\delta}_{t-\delta} \leq t + h - \delta ) \bigr)^{1/2}
\\
&\leq 1+ C \bigl( {\mathbb P}(\tau^{\delta}_{t-\delta} \leq t + h - \delta) \bigr)^{1/2}. 
\end{split}
\end{equation}
Therefore, taking the expectation 
in \eqref{eq:clinton}
on the event $\{\tau_{t}^{\delta} \leq t +h\}$, applying the Cauchy-Schwarz inequality
and noticing that ${\mathbb P}(\tau^{\delta}_{t} \leq t + h) \not =0$, we deduce:
\begin{equation*}
1 \leq \alpha  + C \bigl( {\mathbb P}(\tau^{\delta}_{t-\delta}   \leq t - \delta + h) \bigr)^{1/2} + 
C h^{1/2} \bigl[{\mathbb P}(\tau^{\delta}_{t} \leq t + h) \bigr]^{-1/2}. 
\end{equation*}

For $t=0$, this says that 
\begin{equation*}
1 \leq \alpha  + C \bigl( {\mathbb P}(\tau^{\delta}_{0}   \leq h) \bigr)^{1/2} + 
C h^{1/2} \bigl[{\mathbb P}(\tau^{\delta}_{0} \leq  h) \bigr]^{-1/2}. 
\end{equation*}
We claim that (with $C$ the constant found in the above equation)
\begin{equation}
\label{eq:clinton5}
{\mathbb P}\bigl(\tau^{\delta}_{0}   \leq h\bigr)  \leq h^{1/4},
\quad \textrm{for} \ h \in [0,h_{0}), \quad \textrm{where} \  h_0=\frac{1-\alpha}{2C} \ \text{.}
\end{equation}
We argue by contradiction. 
If there exists $h^* \in (0,h_{0})$  such that 
${\mathbb P}(\tau^{\delta}_{0} \leq h^*) > (h^*)^{1/4}$,
we can take, by differentiability of the cumulative distribution function of $\tau_{0}^{\delta}$, 
$h^*$ so that in fact equality is satisfied i.e. ${\mathbb P}(\tau^{\delta}_{0} \leq h^*) = (h^*)^{1/4}$.
We deduce that 
$1 \leq \alpha + 2C (h^*)^{1/8} < \alpha + 2 C h_{0}^{1/8}$, which is a contradiction.

For $t \in [0,\delta]$, we deduce that 
${\mathbb P}(\tau^{\delta}_{t-\delta}   \leq t - \delta + h) =
{\mathbb P}(\tau^{\delta}_{0}   \leq h)  
 \leq h^{1/4}$. 
Assume then that, for some integer $1 \leq k \leq \lceil T/\delta \rceil$ and for all $t \in [0,k \delta]$, 
we have 
${\mathbb P}(\tau^{\delta}_{t-\delta}   \leq t - \delta + h)  \leq h^{1/4}$,
for $h \in [0,h_{0})$.
We then claim 
that ${\mathbb P}(\tau^{\delta}_{t}   \leq t  + h)  \leq h^{1/4}$ 
for all $t \in [0, (k \delta) \wedge T]$ and $h \in [0,h_{0})$. This can be proved by contradiction
again by considering $h^*$ such that ${\mathbb P}(\tau^{\delta}_{t} \leq t + h^*) = (h^*)^{1/4}$.
We finally deduce that
\begin{equation}
\label{eq:clinton3}
{\mathbb P}(\tau^{\delta}_{t} \leq t + h) \leq C h^{1/4}
\quad \textrm{so that} \quad  
e_{\delta}(t+h) - e_{\delta}(t) \leq 1 + Ch^{1/8},
\end{equation}
the second claim following from \eqref{eq:clinton2} and proving the first statement. 

For the second statement, 
notice that, on the event $\{\sigma^{\delta}_{t} \leq s\}$,
for $t \leq s \leq t+h$,
\begin{equation*}
\begin{split}
X^{\delta}_{\sigma_{t}^{\delta}-}- X_{t-}^{\delta} 
\leq \alpha \bigl( e_{\delta}(s) - e_{\delta}(t) \bigr)
+ \sup_{s \in [t,t+h]} \vert W_{s} - W_{t} \vert
+ C' h \bigl( 1 + \sup_{s \in [t,t+h]} \vert Z_{s}^{\delta}\vert \bigr).
\end{split}
\end{equation*}
Since $X^{\delta}_{\sigma_{t}^{\delta}-}=1$ ($e_{\delta}$ being continuous, see 
\eqref{e_delta}), we have 
for $ e_{\delta}(s) - e_{\delta}(t) \leq \lambda$:
\begin{equation*}
\begin{split}
&\P \bigl( X_{t-}^{\delta} \geq 1 - \alpha \lambda - C'h^{1/8}
\bigr) 
\\
&\geq \P \bigl( \sigma_{t}^{\delta} \leq s \bigr) 
- 
\P \bigl( \sup_{s \in [t,t+h]} \vert W_{s} - W_{t} \vert
+ C' h \bigl( 1 + \sup_{s \in [t,t+h]} \vert Z_{s}^{\delta} \vert\bigr)
\geq C'h^{1/8} \bigr)
\\
&\geq \P \bigl( \sigma_{t}^{\delta} \leq s \bigr) - C' h
\geq  e_{\delta}(s+\delta) -e_{\delta}(t+\delta) - 2C'h^{1/8},
\end{split}
\end{equation*}
the last inequality following from the second line in \eqref{eq:clinton2} and from \eqref{eq:clinton3}.  
Now, for $\lambda \leq e_{\delta}(t+h) -e_{\delta}(t) $, 
we can find $s^* \in [t,t+h]$ such that
$\lambda = e_{\delta}(s^*) -e_{\delta}(t)$, which completes the proof
since
$e_{\delta}(s^*+\delta) -e_{\delta}(t+\delta) \geq \lambda + e_{\delta}(t)  
-e_{\delta}(t+\delta)$.
\end{proof}

%

\bibliographystyle{siam}
\bibliography{NeuroMcKV}

\begin{thebibliography}{10}

\bibitem{baladron-fasioli-touboul-faugeras}
{\sc J.~Baladron, D.~Fasoli, O.~Faugeras, and J.~Touboul}, {\em Mean-field
  description and propagation of chaos in networks of {H}odgkin-{H}uxley and
  {F}itz{H}ugh-{N}agumo neurons}, J. Math. Neurosci., 2 (2012), pp.~Art. 10,
  50.

\bibitem{billingsley}
{\sc P.~Billingsley}, {\em Convergence of probability measures}, Wiley Series
  in Probability and Statistics: Probability and Statistics, John Wiley \& Sons
  Inc., New York, second~ed., 1999.
\newblock A Wiley-Interscience Publication.

\bibitem{CCP}
{\sc M.~J. C{\'a}ceres, J.~A. Carrillo, and B.~Perthame}, {\em Analysis of
  nonlinear noisy integrate \& fire neuron models: blow-up and steady states},
  J. Math. Neurosci., 1 (2011), pp.~Art. 7, 33.

\bibitem{caceres:perthame}
{\sc M.~J. C{\'a}ceres and B.~Perthame}, {\em Beyond blow-up in excitatory
  integrate and fire neuronal networks: Refractory period and spontaneous
  activity}, J. Theor. Biol., 350 (2014), pp.~81 -- 89.

\bibitem{carillo:gonzales:gualdani:schonbek}
{\sc J.~A. Carrillo, M.~D.~M. Gonz{\'a}lez, M.~P. Gualdani, and M.~E.
  Schonbek}, {\em Classical solutions for a nonlinear {F}okker-{P}lanck
  equation arising in computational neuroscience}, Comm. Partial Differential
  Equations, 38 (2013), pp.~385--409.

\bibitem{catsigeras_guiraud_2013}
{\sc E.~Catsigeras and P.~Guiraud}, {\em Integrate and fire neural networks,
  piecewise contractive maps and limit cycles}, J. Math. Biol., 67 (2013),
  pp.~609--655.

\bibitem{masi-galves-locherbach-presutti}
{\sc A.~De~Masi, A.~Galves, E.~L\"ocherbach, and E.~Presutti}, {\em
  Hydrodynamic limit for interacting neurons}, J. Stat. Phys.,  (2014),
  pp.~1--37.

\bibitem{DIRT}
{\sc F.~Delarue, J.~Inglis, S.~Rubenthaler, and E.~Tanr{\'e}}, {\em Global
  solvability of a networked integrate-and-fire model of {M}c{K}ean-{V}lasov
  type}.
\newblock \url{arxiv:1211.0299v4}, to appear in Ann. Appl. Probab., 2014.

\bibitem{ethier:kurtz}
{\sc S.~N. Ethier and T.~G. Kurtz}, {\em Markov processes}, Wiley Series in
  Probability and Mathematical Statistics: Probability and Mathematical
  Statistics, John Wiley \& Sons Inc., New York, 1986.
\newblock Characterization and convergence.

\bibitem{faugeras-touboul-cessac}
{\sc O.~Faugeras, J.~Touboul, and B.~Cessac}, {\em A constructive mean-field
  analysis of multi population neural networks with random synaptic weights and
  stochastic inputs}, Front. Comput. Neurosci.,  (2009).

\bibitem{gisecke_2013}
{\sc K.~Giesecke, K.~Spiliopoulos, and R.~B. Sowers}, {\em Default clustering
  in large portfolios: typical events}, Ann. Appl. Probab., 23 (2013),
  pp.~348--385.

\bibitem{lucon-stannat}
{\sc E.~Lu\c{c}on and W.~Stannat}, {\em Mean field limit for disordered
  diffusions with singular interactions}, Ann. Appl. Probab., 24 (2014),
  pp.~1946--1993.

\bibitem{meleard}
{\sc S.~M{\'e}l{\'e}ard}, {\em Asymptotic behaviour of some interacting
  particle systems; {M}c{K}ean-{V}lasov and {B}oltzmann models}, in
  Probabilistic models for nonlinear partial differential equations
  ({M}ontecatini {T}erme, 1995), vol.~1627 of Lecture Notes in Math., Springer,
  Berlin, 1996, pp.~42--95.

\bibitem{ostojic:brunel:hakim}
{\sc S.~Ostojic, N.~Brunel, and V.~Hakim}, {\em Synchronization properties of
  networks of electrically coupled neurons in the presence of noise and
  heterogeneities}, J. Comput. Neurosci., 26 (2009), pp.~369--392.

\bibitem{skorohod}
{\sc A.~V. Skorohod}, {\em Limit theorems for stochastic processes}, Teor.
  Veroyatnost. i Primenen., 1 (1956), pp.~289--319.

\bibitem{spiliopoulos}
{\sc K.~Spiliopoulos, J.~A. Sirignano, and K.~Giesecke}, {\em Fluctuation
  analysis for the loss from default}, Stochastic Process. Appl., 124 (2014),
  pp.~2322--2362.

\bibitem{sznitman}
{\sc A.-S. Sznitman}, {\em Topics in propagation of chaos}, in \'{E}cole
  d'\'{E}t\'e de {P}robabilit\'es de {S}aint-{F}lour {XIX}---1989, vol.~1464 of
  Lecture Notes in Math., Springer, Berlin, 1991, pp.~165--251.

\bibitem{whitt:paper}
{\sc W.~Whitt}, {\em The reflection map with discontinuities}, Math. Oper.
  Res., 26 (2001), pp.~447--484.

\bibitem{whitt:monograph}
\leavevmode\vrule height 2pt depth -1.6pt width 23pt, {\em Stochastic-process
  limits}, Springer Series in Operations Research, Springer-Verlag, New York,
  2002.
\newblock An introduction to stochastic-process limits and their application to
  queues.

\end{thebibliography}

\end{document}